\documentclass[10pt,reqno,final]{amsart}
\usepackage{epsfig,amssymb,amsmath,version}
\usepackage{amssymb,version,graphicx,fancybox,mathrsfs}
\usepackage[notcite,notref]{showkeys}
\usepackage{url,hyperref}
\usepackage{subfigure}
\usepackage{color}
\usepackage{stmaryrd}
\usepackage{multirow}
\usepackage{booktabs,siunitx}
\usepackage{upgreek}
\usepackage{booktabs}
\usepackage{setspace}
\usepackage{enumitem}
\allowdisplaybreaks[4]

\catcode`\@=11 \theoremstyle{plain}
\@addtoreset{equation}{section}   
\renewcommand{\theequation}{\thesection.\arabic{equation}}
\@addtoreset{figure}{section}
\renewcommand\thefigure{\thesection.\@arabic\c@figure}

\newtheorem{thm}{\bf Theorem}

\@addtoreset{thm}{section}

\newenvironment{theorem}{\begin{thm}} {\end{thm}}
\newtheorem{cor}{\bf Corollary}
\newtheorem{prop}{Proposition}[section]

\newtheorem{lmm}{\bf Lemma}

\newenvironment{lemma}{\begin{lmm}}{\end{lmm}}
\@addtoreset{lmm}{section}
\theoremstyle{remark}
\newtheorem{rem}{\bf Remark}[section]

\def \epsilon {{\varepsilon}}

\definecolor{bgblue}{rgb}{0.04,0.39,0.54}
\definecolor{lired}{rgb}{0.3, 0.0, 0.0}
\definecolor{ligreen}{rgb}{0.0, 0.3, 0.0}
\definecolor{liblue}{rgb}{0.9, 1.0, 1.0}
\definecolor{gray}{rgb}{0.6, 0.6, 0.6}
\definecolor{sky}{rgb}{0.3, 1.0, 1.0}
\definecolor{bunhong}{rgb}{1.0, 0.3, 1.0}
\definecolor{yellow}{rgb}{0.97, 1, 0.0}
\definecolor{liyellow}{rgb}{0.9, 0.8, 0.0}
\definecolor{cengse}{rgb}{0.00,0.40,0.29}

\newcommand{\bs}[1]{\boldsymbol{#1}}

\def  \sig  {\boldsymbol{\upsigma}}
\def  \tt {\boldsymbol{\uptau}}
\def  \om {\boldsymbol{\Upomega}}

\renewcommand \wedge \times

\begin{document}
\bibliographystyle{plain}

{\title[Rotational discrete gradient scheme for Ericksen--Leslie model] {A geometric reformulation and structure-preserving rotational discrete gradient scheme for the full Ericksen--Leslie model}
\author[
	H. Wang,\,    J. Xu\,  $\&$\,  Z. Yang
	]{
		\;\; Hanbin Wang${}^{1}$,   \;\;  Jie Xu${}^{2}$ \;\; and\;\; Zhiguo Yang${}^{*3}$
		}
	\thanks{
	${}^{*}$ Corresponding author.
	 \\
		\indent ${}^{1}$ School of Mathematical Sciences, Shanghai Jiao Tong University, Shanghai 200240, China. Email: hb.wang@sjtu.edu.cn (H. Wang) 
	\\
	\indent${}^{2}$ SKLMS \& NCMIS, Institute of Computational Mathematics and Scientific/Engineering Computing (ICMSEC), Academy of Mathematics and Systems Science (AMSS), Chinese Academy of Sciences, Beijing 100190, China. Email: xujie@lsec.cc.ac.cn (J. Xu) \\
	 \indent ${}^{3}$ School of Mathematical Sciences, Shanghai Jiao Tong University, Shanghai 200240, China. Email: yangzhiguo@sjtu.edu.cn (Z. Yang).
		}
		}


\begin{abstract}
In this paper, we present a structure-preserving rotational discrete gradient  method of the full Ericksen--Leslie model for nematic liquid crystal flows with general anisotropic Oseen--Frank elasticity and full Leslie stress. The main difficulty resides in preserving both the pointwise unit-length constraint on the director field and the energy-dissipation at the discrete level, which involves high nonlinearities in the interplay between the length constraint, anisotropic elastic energy, and hydrodynamic couplings. To address this issue, we first reformulate the Ericksen--Leslie system into an equivalent rotational form in which the director evolution is intrinsically tangent to the unit sphere and the coupling terms in the Leslie and Ericksen stresses are reorganized so that the energy exchange and dissipation structure become explicit. Based on this reformulation, we construct rotational discrete gradient schemes that preserve the director length and satisfy an unconditional discrete energy law, together with a fully discrete scheme based on an exact divergence-free spectral approximation. Numerical experiments verify the accuracy, unit-length preservation and energy-stability of the proposed method, and illustrate dynamical effects induced by anisotropic elastic coefficients and shear flow in the full Ericksen--Leslie model.
\end{abstract}

 \maketitle
\section{Introduction}

In liquid crystal flows, the velocity gradient affects the local anisotropy generated by nonuniform orientational distribution of constituent rigid molecules. 
This in turn induces distinctive flow patterns that characterize the properties of various materials. Such anisotropic couplings play an essential role in a variety of applications, including 
adaptive optics~\cite{Peng2011LCSpatialLightModulator}, biosensing and optofluidics~\cite{sengupta2012opto}, microfluidic control~\cite{Sengupta2013FlowShaping}, and active-nematic systems~\cite{walton2025orienting}.

For the ideal uniaxial nematic phase of liquid crystals, the orientational distribution is axisymmetric and spatially homogeneous. 
It is then often assumed in the dynamic models that the orientational distribution is the same as the above ideal case, while the axis is allowed to vary in space. 
The local anisotropy is now described by a unit vector field, and the corresponding dynamics is known as the Ericksen--Leslie (EL) model~\cite{ericksen1962hydrostatic,leslie1979theory}, which is an energy dissipative system consisting of the director equation and the incompressible Navier--Stokes equations. 
As it turns out, the presence of the local anisotropy leads to a highly nonlinear and strongly coupled structure in the Ericksen--Leslie model. 
Experiments indicate that the splay, twist, and bend elastic coefficients in the Oseen--Frank energy \cite{Frank1958} can be highly disparate (see, for example \cite[Chapter 4]{deGennesProst1993}). 
Moreover, the couplings between the vector and the velocity gradient in the director equation bring about quite a few conjugate terms in the stress in the Navier--Stokes equations. 
The stress also contains several nonlinear, anisotropic dissipation terms.
These terms can be derived from molecular or tensor models~\cite{Han2015from, li2023frame}, so that they can be determined by molecular parameters. 
The connection between the stress and the molecular architecture is a fundamental problem that has been discussed in several works~\cite{forest2004flow,forest2004weak,sircar2009dynamics,toth2002hydrodynamics}.
The validity of these terms significantly depends on the unit vector constraint.
The above features call for effective computational approaches to deal with the four aspects simultaneously: (i) the unit vector constraint; (ii) the nonlinear couplings; (iii) the energy dissipation; and (iv) the incompressibility. 

Numerical methods for simplified EL models and related compound systems have been discussed extensively~\cite{alouges1997new,becker2008finite,bouck2024projection,cabrales2015time,chen2016kinematic,cheng2023energy,du2001fourier,du2025semi,lin2007energy,liu2001liquid,liu2000approximation,nochetto2018ericksen,walker2020finite,xu2024second,Zhao_Yang_Li_and_Wang-2016,zheng2024novel,zou2023extrapolated}. 
Except for a couple of works on gradient flows (by setting the velocity as zero) considering disparate elastic constants~\cite{bouck2024projection,xu2024second}, most works on the EL models take the one-constant simplification with the energy $\mathcal{F}[\bs n]=(\kappa/2)\int_{\Omega} |\nabla \bs n|^2\,dV$, exactly for the harmonic map~(see \cite{alouges1997new,lin1995nonparabolic}). 
Numerical methods aiming at the unit-vector constraint tailored for the one-constant case include projection methods~\cite{bai2025convergence,du2025semi,gui2022convergence}, polar-coordinate reformulations~\cite{bao2021constraint,wang2023error}, and Lagrange-multiplier formulations~\cite{Badia2011a,cao2025length}. 
The unit-vector constraint can be relaxed by the
Ginzburg--Landau penalization~\cite{becker2008finite,cheng2023energy,du2001fourier,liu2000approximation}, in this direction, a variety of energy stable schemes have been developed~\cite{cabrales2015time,cheng2023energy,lin2007energy,Zhao_Yang_Li_and_Wang-2016,zou2023extrapolated}.

The purpose of this work is to develop numerical methods for the full EL model that respect both the unit-vector constraint and the energy-dissipation structure. The original form of the full EL model is derived primarily from symmetry arguments. In this formulation, the unit-vector constraint is imposed as an explicit nonlinear constraint, while the energy law is encoded through delicate cancellations among the director equation, the Ericksen stress, and the Leslie stress. These cancellations are strongly tied to the pointwise constraint and the Parodi relation, and therefore may be easily destroyed by a direct discretization of the individual terms. To the best of our knowledge, there is no existing numerical method for the full EL model with general Oseen--Frank elasticity and full Leslie stress that preserves the unit-length constraint and satisfies an unconditional discrete energy law simultaneously.

The central idea of this paper is to first recast the full EL system into a geometric form in which the constraint, the anisotropic elastic response, and the hydrodynamic energy exchange are represented in a way that is compatible with structure-preserving discretization. Based on this observation, we introduce an equivalent geometric reformulation of the full EL model and construct corresponding rotational discrete gradient (Rdg) schemes. The novelty of the present work lies in the following three closely related aspects:
\begin{enumerate}[label=(\roman*)]
	\item We derive a geometric reformulation of the full EL system in which the director evolution is intrinsically tangent to the unit sphere and the coupling terms in the Leslie and Ericksen stresses are reorganized so that the energy exchange and dissipation structure become explicit. This reformulation provides the foundation for developing structure-preserving numerical discretization.	
 	\item We construct a time-discrete Rdg scheme for the full EL model in which the Oseen--Frank discrete gradient reproduces the increment of the anisotropic elastic energy, while the rotational treatment preserves the unit-length constraint without projection or penalization. As a result, the scheme satisfies an unconditional discrete energy-dissipation law at the time-discrete level.
        \item We further develop a fully discrete Rdg scheme by combining the proposed time discretization with an exact divergence-free spectral approximation for the incompressible velocity field and a quadrature-compatible LGL nodal discretization for the director field. The resulting method preserves the incompressibility constraint, the unit-length constraint at the quadrature points, and a fully discrete energy-dissipation law.         
\end{enumerate}
The rest of the paper is organized as follows. 
In Section~2, we present the original full Ericksen--Leslie model and derive its rotational reformulation. 
In Section~3, we introduce the time-discrete structure-preserving schemes and establish their key discrete properties. 
In Section~4, we describe the fully discrete scheme based on divergence-free spectral method. 
Section~5 is devoted to numerical experiments, including convergence tests, verification of discrete structure-preserving properties, and simulations highlighting anisotropic and shear-flow effects. 
Finally, Section~6 contains concluding remarks.

\subsection{Notations}
Denote vectors and matrices by bold italic and non-italic letters, e.g. $\bs{n}=(n_i)$ and $\mathbf{p}=(p_{i j})$, respectively. For conciseness of notation, we adopt the suffix notation on repeated indices such that 
\begin{equation}\label{eq: dotp}
\bs n \cdot \bs m =n_im_i:=\sum_{i=1}^3 n_im_i,\quad \mathbf{p} : \mathbf{q}=p_{i j} q_{i j}:=\sum_{i, j=1}^3 p_{i j} q_{i j},\;\;\;\; \text{for}\;\; i,j=1,2,3.
\end{equation}
The left and right dot product of a matrix and a vector are defined by $(\mathbf{p} \cdot \bs n)_i=p_{ij} n_j$ and $(\bs n \cdot \mathbf{p})_j=n_i p_{ij}$, respectively. We also adopt the notations of Kronecker product $(\bs m \otimes \bs n)_{ij}:=m_i n_j$, the gradient of vector $(\nabla \bs v)_{ij}:=\partial_i v_j$, the divergence of a matrix $(\nabla \cdot \mathbf{p})_{j}=\partial_i p_{i j}$, respectively.  For a matrix field $\mathbf a=(a_{ij})$, its gradient $\nabla \mathbf a$ is the third-order tensor with components $(\nabla \mathbf a)_{ijk}=\partial_i a_{jk}$, and for a third-order tensor $\mathfrak a=(a_{ijk})$ and a matrix $\mathbf b=(b_{jk})$, the contraction $\mathfrak a:\mathbf b$ is the vector with components $(\mathfrak a:\mathbf b)_i=a_{ijk}b_{jk}$ and the contraction $\mathbf b:\mathfrak a$ is the vector with components $(\mathbf b:\mathfrak a)_k=b_{ij}a_{ijk}$. Let us also introduce the notations of Kronecker delta $\delta_{ij}$ and the alternating tensor $\epsilon_{ijk}$ defined by 
\begin{equation}\label{eq: altenosr}
\delta_{ij}=
\begin{cases}
1, & i=j,\\
0, & i\neq j,
\end{cases},\quad 
\epsilon_{ijk}=
\begin{cases}
1, & (i,j,k)=(1,2,3), \,(2,3,1),\,(3,1,2),\\
-1,& (i,j,k)=(2,1,3),\,(3,2,1),\,(1,3,2),\\
0, & \text{if any of} \,\,i,j,k\,\, \text{are equal}.
\end{cases}
\end{equation}
These two symbols are related by the following identity
\begin{equation}\label{eq: dellevi}
\epsilon_{ijk} \epsilon_{klm}=\delta_{il}\delta_{jm}-\delta_{im}\delta_{jl}, 
\end{equation}
which will be frequently used for manipulating expressions involving the cross product and the curl operator. In particular, we shall repeatedly use the following standard identities for the cross product: for any 
$\bs a, \bs b, \bs c \in \mathbb{R}^3$,
\begin{subequations}\label{id: crossid}
	\begin{align}
&(\bs a \times \bs b)\times \bs a = |\bs a|^2 \bs b - (\bs a\cdot \bs b)\bs a ,\label{id: crossid1}\\
&(\bs a \times \bs b)\cdot \bs c = (\bs b \times \bs c)\cdot \bs a = (\bs c \times \bs a)\cdot \bs b 
,\label{id: crossid2}\\
&\bs a \times \bs a =\bs 0.\label{id: crossid3}
\end{align}
\end{subequations}

For a specific domain $\Omega$, we use $(\cdot,\cdot)_\Omega$ and $\|\cdot\|_{\Omega}$ to denote the inner product and the $L^2$-norm of $\Omega$, and use $\langle\cdot,\cdot\rangle_{\Omega}$ to denote the inner product on the boundary $\partial \Omega$ as usual. In this paper we may drop the subscript $\Omega$ for conventions. $\bs H^1(\Omega;\mathbb{R}^3)$ and $\bs H_0^1(\Omega;\mathbb{R}^3)$ are denoted as the usual Sobolev spaces, and we omit the reference to $\mathbb{R}^3$ if no ambiguity occurs.

\section{The full Ericksen--Leslie model and its reformulation}

\subsection{The full Ericksen--Leslie model}

We consider the full EL system for nematic liquid crystal flows; see, for example, \cite{deGennesProst1993}. The unknowns are the fluid velocity $\bs v$, the pressure $p$, and the director field $\bs n$ representing the averaged molecular orientation. The dimensionless EL system is given by
\begin{subequations}\label{eq: elmodel}
\begin{align}
& \bs v_t+\bs v\cdot \nabla \bs v
= -\nabla p+\frac{\gamma}{\rm Re}\Delta \bs v+\frac{1-\gamma}{\rm Re}\nabla\cdot \sig,
\label{eq: eleq}\\
& \nabla\cdot \bs v=0,
\label{eq: incomp}\\
& \bs n\times \big(\bs \mu-\gamma_1 \bs N-\gamma_2 \tt\cdot \bs n\big)=\bs 0,
\label{eq: offlow}\\
& |\bs n|=1,
\label{eq: unitv}
\end{align}
\end{subequations}
together with the initial and boundary conditions
\begin{subequations}\label{eq: boundary}
\begin{align}
& \bs v(0,\bs x)=\bs v_0,\qquad \bs n(0,\bs x)=\bs n_0,\\
& \bs v(t,\bs x)\big|_{\partial\Omega}=\bs g_1(\bs x),\qquad
  \bs n(t,\bs x)\big|_{\partial\Omega}=\bs g_2(\bs x).
\end{align}
\end{subequations}
Here, ${\rm Re}$ is the Reynolds number, and $\gamma\in(0,1)$ is a given dimensionless parameter. The elastic contribution is described by the Oseen--Frank energy $\mathbf{F}[\bs n]=\int_\Omega f(\bs n,\nabla \bs n)\,dV$
with energy density
\begin{equation}
f(\bs n,\nabla \bs n)
=
\frac12\Big\{
\kappa_1(\nabla\cdot \bs n)^2
+\kappa_2(\bs n\cdot\nabla\times \bs n)^2
+\kappa_3|\bs n\times (\nabla\times \bs n)|^2
+(\kappa_2+\kappa_4)\rm{tr}\big((\nabla \bs n)^2-(\nabla \cdot \bs n)^2\big)\Big\}.
\label{eq:OF-density-intro}
\end{equation}
These first three terms correspond to the excess elastic energy associated with splay, twist, and bend deformations. The contribution of the last term to the total energy is determined only by the boundary values of $\bs n$ and its tangential derivatives, and hence remains constant under the Dirichlet boundary condition~\eqref{eq: boundary}. For this reason, we omit this term in the rest of the manuscript. The molecular field $\bs \mu$ is defined by
\begin{equation}\label{eq: mu}
\bs \mu
=
-\frac{\delta \mathcal{F}}{\delta \bs n}
=
\nabla \cdot \frac{\partial f}{\partial(\nabla \bs n)}
-\frac{\partial f}{\partial \bs n}.
\end{equation}
The tensors $\tt$, $\bs \Omega$, and the corotational derivative $\bs N$ are given by
\begin{equation}\label{eq: tN}
\tt=\frac{\nabla \bs v+(\nabla \bs v)^{\intercal}}{2},
\qquad
\bs \Omega=\frac{\nabla \bs v-(\nabla \bs v)^{\intercal}}{2},
\qquad
\bs N=\bs n_t+\bs v\cdot \nabla \bs n+\bs \Omega\cdot \bs n.
\end{equation}
The total stress tensor $\sig$ is decomposed as $\sig=\sig^{\rm E}+\sig^{\rm L}$,
where $\sig^{\rm E}$ is the Ericksen stress, namely the elastic stress tensor,
\begin{equation}\label{eq: elastress}
\sig^{\rm E}
=
-\frac{\partial f}{\partial(\nabla \bs n)}\cdot (\nabla \bs n)^{\intercal},
\end{equation}
while $\sig^{\rm L}$ is the Leslie stress, namely the viscous stress tensor,
\begin{equation}\label{eq: Lstress}
\sig^{\rm L}
=
\alpha_1(\bs n\otimes \bs n:\tt)\,\bs n\otimes \bs n
+\alpha_2 \bs n\otimes \bs N
+\alpha_3 \bs N\otimes \bs n
+\alpha_4 \tt
+\alpha_5 \bs n\otimes (\tt\cdot \bs n)
+\alpha_6 (\tt\cdot \bs n)\otimes \bs n.
\end{equation}
Here $\alpha_1,\ldots,\alpha_6$ are the Leslie coefficients. These coefficients, together with $\gamma_1$ and $\gamma_2$, satisfy
\begin{align}
& \alpha_2+\alpha_3=\alpha_6-\alpha_5,
\label{eq:Parodi}\\
& \gamma_1=\alpha_3-\alpha_2,\quad \gamma_2=\alpha_6-\alpha_5,
\label{eq: coeff2}
\end{align}
where \eqref{eq:Parodi} is the Parodi relation derived from Onsager's reciprocal principle.

The EL system \eqref{eq: elmodel} satisfies the following energy dissipation law given in Theorem~\ref{thm: enebalaw}.
\begin{theorem}\label{thm: enebalaw}
Assume either homogeneous boundary conditions $\bs v(t,\bs x)=\bs 0$ and $\bs n(t,\bs x)=\bs g_2(\bs x)$ for all $(t,\bs x)\in (0,T)\times \partial\Omega$, or periodic boundary conditions. Then the following energy dissipation law holds: 
\begin{equation}\label{eq: endisslaw}
\begin{aligned}
\frac{d}{dt}\Big\{
\frac{{\rm Re}}{2(1-\gamma)}\int_\Omega |\bs v|^2\,dV+&\mathcal{F}[\bs n]
\Big\}
=
-\int_\Omega \Big\{
\frac{\gamma}{1-\gamma}|\nabla \bs v|^2
+\Big(\alpha_1+\frac{\gamma_2^2}{\gamma_1}\Big)(\tt:\bs n\otimes \bs n)^2
\\&\qquad \qquad+\alpha_4(\tt:\tt) 
+\Big(\alpha_5+\alpha_6-\frac{\gamma_2^2}{\gamma_1}\Big)|\tt\cdot \bs n|^2
+\frac{1}{\gamma_1}|\bs n\times \bs \mu|^2
\Big\}\,{\mathrm d}V.
\end{aligned}
\end{equation}
\end{theorem}

\begin{rem}
Throughout the paper, whenever we refer to the energy dissipation law, we assume
\begin{equation}\label{eq: paracons}
\gamma_1>0,\qquad
\alpha_4\ge 0,\qquad
\alpha_1+\frac{\gamma_2^2}{\gamma_1}\ge 0,\qquad
\alpha_5+\alpha_6-\frac{\gamma_2^2}{\gamma_1}\ge 0,
\end{equation}
so that all dissipative terms on the right-hand side of \eqref{eq: endisslaw} are nonnegative.
\end{rem}

The derivation of the energy law \eqref{eq: endisslaw} relies on cancellations and dissipative relations among the director equation, the Ericksen stress, and the Leslie stress, where the unit-length constraint $|\bs n|=1$ and the Parodi relation \eqref{eq:Parodi} play essential roles. One target in the reformulation is to make the cancellation and dissipation relations independent of $|\bs{n}|=1$. 

\subsection{Reformulation of the Ericksen--Leslie model}

For the full EL model, two intrinsic properties are particularly important for numerical simulation: (i) the pointwise unit-length constraint $|\bs n|=1$ and (ii) the energy dissipation law of the coupled system. In the original formulation~\eqref{eq: elmodel}, however, the unit-length constraint appears as an explicit nonlinear constraint, while the dissipative structure is hidden among the director equation, the Leslie stress, and the Ericksen stress. As a consequence, a direct time-centered discretization of the original EL model does not naturally lead to a numerical scheme that preserves unit-length and energy dissipation law.

The key to the development of the proposed unconditionally energy stable and length-preserving schemes is to reformulate the full EL model into an equivalent rotational form that is better suited for structure-preserving discretization. The following lemmas provide the main ingredients of this reformulation. 

We begin with the reformulation of equation for the director field. The next lemma shows that, once the initial condition satisfies $|\bs n(\bs x,0)|=1$, the original constrained director equation is equivalent to a rotational evolution equation that preserves the unit-length automatically.

\begin{lemma}\label{lm: nreform}
Given the initial condition $|\bs n(\bs x,0)|=1$, equation \eqref{eq: offlow} together with the unit-vector constraint \eqref{eq: unitv} is equivalent to the following rotational form:
\begin{equation}\label{eq: offlowrot}
\bs n_t
=
\frac{1}{\gamma_1}
\Big(
\bs n\times
\big(
\bs \mu-\gamma_1(\bs v\cdot \nabla \bs n+\om\cdot \bs n)-\gamma_2\tt\cdot \bs n
\big)
\Big)\times \bs n.
\end{equation}
\end{lemma}

\begin{proof}
We first show that \eqref{eq: offlowrot} implies \eqref{eq: offlow} and \eqref{eq: unitv}. Taking the dot product of \eqref{eq: offlowrot} with $\bs n$ and using 
$
\big((\bs n\times \bs a)\times \bs n\big)\cdot \bs n=0
$
for any vector $\bs a$, we obtain
\begin{equation}
	\frac12 (|\bs n|^2)_t=\bs n_t \cdot \bs n=0.
\end{equation}
Since $|\bs n(\bs x,0)|=1$, it follows that $|\bs n(\bs x,t)|=1$ for all $t$, which proves \eqref{eq: unitv}. Moreover, taking the cross product of \eqref{eq: offlowrot} with $\bs n$ from the left and using $|\bs n|=1$, identities \eqref{id: crossid1} and \eqref{id: crossid3}, we obtain
\begin{equation}\label{eq:ng1}
\bs n\times(\gamma_1\bs n_t)
=
\bs n\times
\big(
\bs \mu-\gamma_1(\bs v\cdot \nabla \bs n+\om\cdot \bs n)-\gamma_2\tt\cdot \bs n
\big).
\end{equation}
By the definition of $\bs N$ in \eqref{eq: tN}, this is exactly \eqref{eq: offlow}.

Conversely, assume that \eqref{eq: offlow} and \eqref{eq: unitv} hold. Taking the right cross product of \eqref{eq: offlow} with $\bs n$ and recalling the definition of $\bs N$, we get
\begin{equation}\label{eq: ng2}
\Big(
\bs n\times
\big(
\bs \mu-\gamma_1(\bs v\cdot \nabla \bs n+\om\cdot \bs n)-\gamma_2\tt\cdot \bs n
\big)
\Big)\times \bs n
=
(\gamma_1\bs n\times\bs n_t)\times\bs n .
\end{equation}
Using $|\bs n|=1$ and $\bs n_t\cdot\bs n=(|\bs n|^2/2)_t=0$, the right-hand side becomes $\gamma_1\bs n_t$. Hence \eqref{eq: ng2} reduces to \eqref{eq: offlowrot}. This completes the proof.
\end{proof}

Lemma~\ref{lm: nreform} shows that the unit-length constraint can be recovered directly from the rotational structure of the director equation. This is already useful for the design of numerical schemes, since one may aim at a discretization that preserves $|\bs n|=1$ through the form of the equation itself.  To recover the dissipative structure at the discrete level, however, it is also necessary to reformulate the stress tensors. We first consider the Leslie stress.

\begin{lemma}\label{eq: Nnew}
Assume $|\bs n(\bs x,t)|=1$. Then $\bs N$ given by \eqref{eq: tN} and ${\sig}^{\rm L}$ given by \eqref{eq: Lstress} can be rewritten into rotational formulations as
\begin{equation}\label{eq: NtN}
\bs N=\widehat{\bs N}
:=
\bs n_t+\big(\bs n\times (\bs v\cdot \nabla \bs n)\big)\times \bs n+\big(\bs n\times (\om\cdot \bs n)\big)\times \bs n,
\end{equation}
and
\begin{equation}\label{eq: sigLnew}
\begin{aligned}
{\sig}^{\rm L}=\widehat{\sig}^{\rm L}
:=
&\;
\Big(\alpha_1+\frac{\gamma_2^2}{\gamma_1}\Big)
(\bs n\otimes \bs n:\tt)\,\bs n\otimes \bs n
+\frac{\alpha_2}{\gamma_1}\bs n\otimes\big((\bs n\times \bs \mu)\times \bs n\big)\\
&+\frac{\alpha_3}{\gamma_1}\big((\bs n\times \bs \mu)\times \bs n\big)\otimes \bs n
+\alpha_4\tt+\Big(\frac{\alpha_5+\alpha_6}{2}-\frac{\gamma_2^2}{2\gamma_1}\Big)
\Big(\bs n\otimes (\tt\cdot \bs n)+(\tt\cdot \bs n)\otimes \bs n\Big).
\end{aligned}
\end{equation}
\end{lemma}

\begin{proof}
Since $\om$ is antisymmetric, we have $\bs n\cdot \om\cdot \bs n=0$. Using \eqref{id: crossid1} and $|\bs n|=1$, we obtain
\begin{equation}\label{eq: omegan}
\big(\bs n\times(\om\cdot \bs n)\big)\times \bs n
=|\bs n|^2\om \cdot \bs n-(\bs n \cdot \om \cdot \bs n)\bs n=
\om\cdot \bs n .
\end{equation}
Similarly, since
$
(\bs v\cdot \nabla \bs n)\cdot \bs n
=
\bs v\cdot \nabla (|\bs n|^2/2)
=
0,
$
we obtain
\begin{equation}\label{eq: nvn}
\big(\bs n\times(\bs v\cdot \nabla \bs n)\big)\times \bs n
=
\bs v\cdot \nabla \bs n .
\end{equation}
Combining \eqref{eq: omegan} and \eqref{eq: nvn} yields \eqref{eq: NtN}.

It remains to derive \eqref{eq: sigLnew}. By using \eqref{id: crossid1} and the unit-length condition, we obtain
\begin{equation}\label{eq: tautensor}
	\begin{aligned}
		\bs n\otimes\big(\bs n\times(\tt \cdot \bs n)\big)\times \bs n
		&=
		\bs n\otimes(\tt\cdot \bs n)-(\bs n\otimes \bs n:\tt)\,\bs n\otimes \bs n,
		\\
		\big(\bs n\times(\tt\cdot \bs n)\big)\times \bs n\otimes \bs n
		&=
		(\tt\cdot \bs n)\otimes \bs n
		-(\bs n\otimes \bs n:\tt)\,\bs n\otimes \bs n .	
	\end{aligned}
\end{equation}
Substituting the rotational director equation \eqref{eq: offlowrot} into
\eqref{eq: NtN}, we obtain
\begin{equation}\label{eq: Neli}
\widehat{\bs N}
=
\frac{1}{\gamma_1}(\bs n\times\bs\mu)\times\bs n
-\frac{\gamma_2}{\gamma_1}
(\bs n\times(\tt\cdot\bs n))\times\bs n .
\end{equation}
Inserting \eqref{eq: Neli} into the Leslie stress
\eqref{eq: Lstress} and using \eqref{eq: tautensor}, one has
\begin{equation}\label{eq: sigLnewpre}
	\begin{aligned}
		\sig^{\rm L}
		=&
		\Big(\alpha_1+\frac{(\alpha_2+\alpha_3)\gamma_2}{\gamma_1}\Big)(\bs n\otimes\bs n:\tt)\bs n\otimes\bs n
		+\frac{\alpha_2}{\gamma_1}
		\bs n\otimes\big((\bs n\times\bs\mu)\times\bs n\big)
		+\frac{\alpha_3}{\gamma_1}
		\big((\bs n\times\bs\mu)\times\bs n\big)\otimes\bs n
		\\
		&+\alpha_4\tt
		+\Big(\alpha_5-\frac{\alpha_2\gamma_2}{\gamma_1}\Big)
		\bs n\otimes(\tt\cdot\bs n)
		+\Big(\alpha_6-\frac{\alpha_3\gamma_2}{\gamma_1}\Big)
		(\tt\cdot\bs n)\otimes\bs n.
	\end{aligned}
\end{equation}
By the Parodi relation \eqref{eq:Parodi} and the coefficients definitions \eqref{eq: coeff2}, one has 
\begin{equation*}
	\frac{(\alpha_2+\alpha_3)\gamma_2}{\gamma_1}=\frac{\gamma_2^2}{\gamma_1}, \quad \alpha_5-\frac{\alpha_2\gamma_2}{\gamma_1}
	=
	\alpha_6-\frac{\alpha_3\gamma_2}{\gamma_1}
	=
	\frac{\alpha_5+\alpha_6}{2}
	-\frac{\gamma_2^2}{2\gamma_1}.
\end{equation*}
Therefore \eqref{eq: sigLnewpre} reduces exactly to
\eqref{eq: sigLnew}. This completes the proof.
\end{proof}

\begin{rem}
If one substitutes only \eqref{eq: NtN} into \eqref{eq: Lstress} and keeps the auxiliary variable $\widehat{\bs N}$, rather than further eliminating $\bs n_t$ through \eqref{eq: offlowrot}, one obtains another equivalent form of the Leslie stress. Our subsequent analysis can also be carried out for that formulation. However, eliminating $\bs n_t$ through \eqref{eq: offlowrot} removes several auxiliary terms and leads to a more compact form, which is advantageous for the construction and implementation of the numerical scheme.
\end{rem}

Although Lemma~\ref{eq: Nnew} rewrites the Leslie stress into a form better adapted to the rotational director equation, this alone is still not sufficient for deriving a discrete energy law in a transparent way. To complete the reformulation, we also need a corresponding identity for the Ericksen stress.

\begin{lemma}\label{lm: sigEreform}
Assume $|\bs n(\bs x,t)|=1$. Then
\begin{equation}\label{eq: sigEreform}
\nabla\cdot \sig^{\mathrm{E}}
+\nabla \bs n\cdot \big((\bs n\times \bs \mu)\times \bs n\big)
=
-\nabla f(\bs n,\nabla \bs n).
\end{equation}
\end{lemma}

\begin{proof}
We first prove the identity
\begin{equation}\label{eq: equalEricksenmu}
\nabla\cdot \sig^{\mathrm{E}}+\nabla \bs n\cdot \bs \mu
=
-\nabla f(\bs n,\nabla \bs n).
\end{equation}
Indeed, by the definition of the left and right dot product of a matrix and a vector, for any vector $\bs a$, we have 
\begin{equation*}
	\nabla \bs n \cdot \bs a=\bs a \cdot (\nabla \bs n)^{\intercal}.
\end{equation*}
Moreover, by the definition of contraction between matrix and vector, and using the commutativity of partial derivatives, for any matrix $\mathbf p$, one has 
\begin{equation*}
	\nabla (\nabla \bs n):\mathbf p=\mathbf p :\nabla \big[(\nabla \bs n)^{\intercal}\big].
\end{equation*}
Therefore,
\begin{equation*}
\begin{aligned}
\nabla\cdot \sig^{\mathrm{E}}+\nabla \bs n\cdot \bs \mu
&=
\nabla\cdot\bigg(-\frac{\partial f}{\partial(\nabla \bs n)}\cdot (\nabla \bs n)^{\intercal}\bigg)
+\nabla \bs n\cdot
\bigg(
\nabla\cdot \frac{\partial f}{\partial(\nabla \bs n)}-\frac{\partial f}{\partial \bs n}
\bigg)\\
 &=
-\nabla \cdot \frac{\partial f}{\partial(\nabla \bs n)}\cdot (\nabla \bs n)^{\intercal}-\frac{\partial f}{\partial(\nabla \bs n)}:\nabla \big[(\nabla \bs n)^{\intercal}\big]+\nabla \bs n \cdot \bigg(
\nabla\cdot \frac{\partial f}{\partial(\nabla \bs n)}-\frac{\partial f}{\partial \bs n}
\bigg)
\\
&=
-\nabla(\nabla \bs n):\frac{\partial f}{\partial(\nabla \bs n)}
-\nabla \bs n\cdot \frac{\partial f}{\partial \bs n}=
-\nabla f(\bs n,\nabla \bs n).
\end{aligned}
\end{equation*}

Next, since $|\bs n|=1$, we have $\nabla \bs n\cdot \bs n=0$. Therefore,
\begin{equation}\label{eq: equalFmu}
\begin{aligned}
\nabla \bs n\cdot \bs \mu
&=
\nabla \bs n\cdot \bs \mu-(\bs \mu\cdot \bs n)(\nabla \bs n\cdot \bs n)=
\nabla \bs n\cdot \big((\bs n\times \bs \mu)\times \bs n\big),
\end{aligned}
\end{equation}
where we used \eqref{id: crossid1}. Substituting \eqref{eq: equalFmu} into \eqref{eq: equalEricksenmu} yields \eqref{eq: sigEreform}.
\end{proof}

We can now combine Lemmas~\ref{lm: nreform}--\ref{lm: sigEreform} to obtain an equivalent reformulation of the full EL model. The importance of this reformulation is that the geometric structure of the director equation and the energy dissipative structure of the stress coupling are now both explicit, which substantially simplifies the design of a numerical method that is simultaneously length-preserving and energy-stable.

\begin{theorem}
The original Ericksen--Leslie model \eqref{eq: elmodel} is equivalent to the following reformulated system:
\begin{subequations}\label{eq: elmodelrot}
\begin{align}
& \bs v_t+\bs v\cdot \nabla \bs v
=
-\nabla P+\frac{\gamma}{\rm Re}\Delta \bs v
+\frac{1-\gamma}{\rm Re}\nabla\cdot \widehat{\sig}^{\mathrm{L}}
-\frac{1-\gamma}{\rm Re}\nabla \bs n\cdot \big((\bs n\times \bs \mu) \times \bs n\big),
\\
& \nabla\cdot \bs v=0,
\\
& \bs n_t
=
\frac{1}{\gamma_1}
\Big(
\bs n\times
\big(
\bs \mu-\gamma_1(\bs v\cdot \nabla \bs n+\om\cdot \bs n)-\gamma_2\tt\cdot \bs n
\big)
\Big)\times \bs n,
\end{align}
\end{subequations}
where we recall that $\bs \mu$, $\tt$ and $\om$ are given in \eqref{eq: mu} and \eqref{eq: tN}, respectively, $\widehat{\sig}^{\rm L}$ is defined in \eqref{eq: sigLnew} and $P:=p+\big({(1-\gamma)}/{\rm Re} \big)f(\bs n,\nabla \bs n).$
\end{theorem}

\begin{proof}
The reformulated EL system follows directly from Lemmas~\ref{lm: nreform}, \ref{eq: Nnew}, and \ref{lm: sigEreform}.
\end{proof}

\begin{rem}
The reformulated system \eqref{eq: elmodelrot} directly reveals two essential ingredients that will be used in the design of structure-preserving schemes in the next section. First, the director equation is written in a fully rotational form, so that the unit-length constraint is encoded in the algebraic structure of the evolution law itself, rather than imposed as a constraint equation. Second, the momentum equation separates the reformulated Leslie stress from the Ericksen contribution
$
-\nabla \bs n \cdot \big((\bs n\times \bs \mu)\times \bs n\big),
$
while the scalar term $f(\bs n,\nabla \bs n)$ is absorbed into the modified pressure $P$. 

The reformulated system explicitly separates the convective and dissipative parts under the unit-length constraint and the Parodi relation \eqref{eq:Parodi}. This structural clarity is directly reflected in the proof of the discrete energy dissipation law below in Theorem \ref{thm: main-scheme}.
\end{rem}

\section{Structure-preserving rotational discrete gradient time-discrete scheme}\label{sect: 3}

In this section, we present the proposed rotational discrete gradient scheme for the reformulated EL system \eqref{eq: elmodelrot}. The scheme is designed to preserve simultaneously the unit-length constraint and a discrete energy-dissipation law. Let
$
0=t_0<t_1<\cdots<t_M=T
$
be a partition of the time interval $[0,T]$, and write $\tau_m=t_{m+1}-t_m$ for the time-step size. For a generic variable $\chi$, we denote by $\chi^m$ its numerical approximation at time $t_m$, and define the arithmetic midpoint by
$
\chi^{m+\frac12}:={(\chi^{m+1}+\chi^m)}/{2}.
$

\subsection{A continuous--discrete energy variation pair for the Oseen--Frank energy}

A key difficulty in the time discretization of the reformulated EL system is the treatment of the Oseen--Frank energy. In order for the discrete scheme to reproduce the exact energy increment, one needs a time-discrete counterpart of the variational derivative that mimics the continuous identity
$$
\frac{d}{dt}\mathcal F[\bs n]
=
\int_\Omega
\frac{\delta \mathcal F}{\delta \bs n}\cdot \bs n_t\,dV.
$$
For this purpose, we employ a discrete gradient of the Oseen--Frank energy introduced in \cite{xu2024second} for Oseen--Frank gradient flows. We recall it here because, in the present paper, it serves as a crucial discretization ingredient in coupling the Oseen--Frank energy with the rotational director equation and the reformulated stress tensors of the full EL system.

\begin{prop}
\label{prop:OF-discrete-gradient}
The variational derivative of the Oseen--Frank energy is
\begin{equation}\label{eq: varienergy}
\begin{aligned}
\frac{\delta \mathcal{F}[\bs n]}{\delta \bs n}
=
&
-\kappa_1 \nabla(\nabla \cdot \bs n)
+\kappa_2\Big(
(\bs n\cdot \nabla\times \bs n)(\nabla\times \bs n)
+\nabla\times\big((\bs n\cdot \nabla\times \bs n)\bs n\big)
\Big) \\
&+\kappa_3\Big(
(\nabla\times \bs n)\times(\bs n\times \nabla\times \bs n)
+\nabla\times\big((\bs n\times \nabla\times \bs n)\times \bs n\big)
\Big).
\end{aligned}
\end{equation}
For two consecutive states $\bs n^m$ and $\bs n^{m+1}$, define the Oseen--Frank discrete gradient by
\begin{equation}
\begin{aligned}
D_{\mathcal{F}}^O(\bs n)\big|^{m+\frac12}
=
&-\kappa_1 \nabla\Big(\nabla \cdot \bs n^{m+\frac12}\Big) \\
&+\kappa_2\bigg(
\beta^{m+\frac12}\nabla\times \bs n^{m+\frac12}
+\nabla\times\Big(\beta^{m+\frac12}\bs n^{m+\frac12}\Big)
\bigg) \\
&+\kappa_3\bigg(
\Big(\nabla\times \bs n^{m+\frac12}\Big)\times \boldsymbol{\omega}^{m+\frac12}
+\nabla\times\Big(\boldsymbol{\omega}^{m+\frac12}\times \bs n^{m+\frac12}\Big)
\bigg),
\end{aligned}
\end{equation}
where
\begin{equation}\label{eq: beta}
\beta^k
=
\bs n^k\cdot \nabla\times \bs n^k,\qquad
\boldsymbol{\omega}^k
=
\bs n^k\times \nabla\times \bs n^k,\quad k=m,m+1,
\end{equation}
and $\beta^{m+\frac12}$ and $\boldsymbol{\omega}^{m+\frac12}$ are defined by the midpoint convention above.
Then $D_{\mathcal F}^O(\bs n)\big|^{m+\frac12}$ satisfies the energy difference identity
\begin{equation}\label{eq: endiff}
\int_\Omega
D_{\mathcal F}^O(\bs n)\big|^{m+\frac12}\cdot (\bs n^{m+1}-\bs n^m)\,dV
=
\mathcal F[\bs n^{m+1}]-\mathcal F[\bs n^m].
\end{equation}
\end{prop}

\begin{proof}
The construction of $D_{\mathcal F}^O$ and the identity \eqref{eq: endiff} were established in \cite{xu2024second}. We therefore omit the proof here and only record the result in a form tailored to the full EL discretization. 
\end{proof}

\subsection{A midpoint rotational discrete gradient time-discrete scheme}

We now introduce the time-semidiscrete scheme for the reformulated EL system \eqref{eq: elmodelrot}. The scheme combines a midpoint discretization of the velocity equation, a rotational midpoint treatment for the reformulated director equation and the Ericksen and Leslie stresses, and the discrete gradient $D_{\mathcal F}^O$ for the Oseen--Frank energy. In particular, the scheme accommodates the fully anisotropic fluid--director coupling through the general Ericksen and Leslie stress tensors, rather than relying on isotropic or one-constant simplifications. This combination is designed to retain, at the time-discrete level, the geometric structure of the director evolution and the dissipative structure of the full EL system. 
\vspace{5pt}

\paragraph{\textbf{Time-discrete scheme}.}\label{sch: sch1}
For each $m\ge 0$, given $(\bs v^m,\bs n^m)$ with $\nabla\cdot \bs v^m=0$ and $|\bs n^m|=1$, find $(\bs v^{m+1},\bs n^{m+1},P^{m+\frac12})$ such that
\begin{subequations}\label{eq: mainscheme}
\begin{align}
	\begin{split}
		\frac{\bs v^{m+1}-\bs v^m}{\tau_m}
		&=
		-\bs v^{m+\frac12}\cdot \nabla \bs v^{m+\frac12}
		-\nabla P^{m+\frac12}
		+\frac{\gamma}{\rm Re}\Delta \bs v^{m+\frac12}
		+\frac{1-\gamma}{\rm Re}\nabla\cdot \widehat{\sig}^{\rm L}\big|^{m+\frac12}
		\\
		&\quad
		+\frac{1-\gamma}{\rm Re}\nabla \bs n^{m+\frac12}\cdot
		\Big(
		\big(
		\bs n^{m+\frac12}\times D_{\mathcal F}^O(\bs n)\big|^{m+\frac12}
		\big)\times \bs n^{m+\frac12}
		\Big),
	\end{split}\label{eq: mainscheme-v}\\
	\begin{split}
		\frac{\bs n^{m+1}-\bs n^m}{\tau_m}
		&=
		-\frac{1}{\gamma_1}
		\Big(
		\bs n^{m+\frac12}\times
		\Big[
		D_{\mathcal F}^O(\bs n)\big|^{m+\frac12}
		+\gamma_1\Big(
		\bs v^{m+\frac12}\cdot \nabla \bs n^{m+\frac12}
		+\om^{m+\frac12}\cdot \bs n^{m+\frac12}
		\Big)
		\\
		&\qquad\qquad\qquad\qquad
		+\gamma_2 \tt^{m+\frac12}\cdot \bs n^{m+\frac12}
		\Big]
		\Big)\times \bs n^{m+\frac12},
	\end{split}\label{eq: mainscheme-n}\\
	\begin{split}
		\nabla\cdot \bs v^{m+1}&=0.
	\end{split}\label{eq: mainscheme-div}
\end{align}
\end{subequations}
Here the discrete Leslie stress is defined by
\begin{equation}\label{eq: main-discrete-sigmaL}
\begin{aligned}
\widehat{\sig}^{\rm L}\big|^{m+\frac12}
:=
&\;
\Big(\alpha_1+\frac{\gamma_2^2}{\gamma_1}\Big)
\big(\bs n^{m+\frac12}\otimes \bs n^{m+\frac12}:\tt^{m+\frac12}\big)
\bs n^{m+\frac12}\otimes \bs n^{m+\frac12}
\\
&-\frac{\alpha_2}{\gamma_1}
\bs n^{m+\frac12}\otimes
\Big(
\big(\bs n^{m+\frac12}\times D_{\mathcal F}^O(\bs n)\big|^{m+\frac12}\big)
\times \bs n^{m+\frac12}
\Big)
\\
&-\frac{\alpha_3}{\gamma_1}
\Big(
\big(\bs n^{m+\frac12}\times D_{\mathcal F}^O(\bs n)\big|^{m+\frac12}\big)
\times \bs n^{m+\frac12}
\Big)\otimes \bs n^{m+\frac12}
+\alpha_4 \tt^{m+\frac12}
\\
&+\Big(\frac{\alpha_5+\alpha_6}{2}-\frac{\gamma_2^2}{2\gamma_1}\Big)
\Big(
\bs n^{m+\frac12}\otimes (\tt^{m+\frac12}\cdot \bs n^{m+\frac12})
+(\tt^{m+\frac12}\cdot \bs n^{m+\frac12})\otimes \bs n^{m+\frac12}
\Big).
\end{aligned}
\end{equation}

\begin{thm}\label{thm: main-scheme}
Assume that $|\bs n^0|=1$, $\nabla\cdot \bs v^0=0$, and either $\bs v^m|_{\partial\Omega}=\bs 0$, $\bs n^m|_{\partial\Omega}=\bs g_2(\bs x)$ for all $m\ge 0$, or periodic boundary conditions are imposed. If in addition the material coefficients satisfy \eqref{eq: paracons} and $\gamma\in(0,1)$, then the scheme consisting of equations \eqref{eq: mainscheme}--\eqref{eq: main-discrete-sigmaL} satisfies

\smallskip
\noindent
(i) \textbf{Unit-length preservation:}  
\begin{equation}\label{eq: lengn}
|\bs n^{m+1}|=1.
\end{equation}

\smallskip
\noindent
(ii) \textbf{Discrete energy dissipation:} 
\begin{equation}\label{eq: main-energy-lawfinal}
\frac{\rm Re}{2(1-\gamma)}
\|\bs v^{m+1}\|^2
+\mathcal F[\bs n^{m+1}]
\leq   \frac{\rm Re}{2(1-\gamma)}
\|\bs v^{m}\|^2
+\mathcal F[\bs n^{m}] .
\end{equation}
\end{thm}

\begin{proof}
We first prove the length-preserving property. Taking the dot product of \eqref{eq: mainscheme-n} with $\bs n^{m+\frac12}$, using the cross-product identities \eqref{id: crossid} and the difference-of-squares identity, we obtain
\begin{equation}
|\bs n^{m+1}|^2-|\bs n^m|^2=0,
\end{equation}
it follows by induction from $|\bs n^0|=1$ that $|\bs n^m|=1$ for all $m\ge 0$.

Next, since $\nabla\cdot \bs v^{m+1}=0$ and $\nabla\cdot \bs v^m=0$, we also have $\nabla\cdot \bs v^{m+\frac12}=0$. Taking the $L^2$ inner product of \eqref{eq: mainscheme-v} with $({\rm Re}\tau_m \bs v^{m+\frac12})/(1-\gamma)$, and using the divergence-free condition together with the boundary conditions, yields
\begin{equation}\label{eq: main-proof-v}
\begin{aligned}
\frac{\rm Re}{2(1-\gamma)}
\Big(
\|\bs v^{m+1}\|^2-\|\bs v^m\|^2
\Big)
=&
-\frac{\gamma \tau_m}{1-\gamma}\|\nabla \bs v^{m+\frac12}\|^2
+\tau_m\Big(\nabla\cdot \widehat{\sig}^{\rm L}\big|^{m+\frac12},\bs v^{m+\frac12}\Big)
\\
&+\tau_m\Big(
\nabla \bs n^{m+\frac12}\cdot
\Big[
\big(
\bs n^{m+\frac12}\times D_{\mathcal F}^O(\bs n)\big|^{m+\frac12}
\big)\times \bs n^{m+\frac12}
\Big],
\bs v^{m+\frac12}
\Big).
\end{aligned}
\end{equation}
Taking the $L^2$ inner product of \eqref{eq: mainscheme-n} with $D_{\mathcal F}^O(\bs n)\big|^{m+\frac12}$, using \eqref{eq: endiff} and identity \eqref{id: crossid2}, we obtain
\begin{equation}\label{eq: main-proof-n}
\begin{aligned}
&\mathcal F[\bs n^{m+1}]-\mathcal F[\bs n^m]
=
-\frac{\tau_m}{\gamma_1}
\Big\|
\bs n^{m+\frac12}\times D_{\mathcal F}^O(\bs n)\big|^{m+\frac12}
\Big\|^2
\\
&\qquad \qquad-\tau_m\Big(
\nabla \bs n^{m+\frac12}\cdot
\Big[
\big(
\bs n^{m+\frac12}\times D_{\mathcal F}^O(\bs n)\big|^{m+\frac12}
\big)\times \bs n^{m+\frac12}
\Big],
\bs v^{m+\frac12}
\Big)
\\
&\qquad \qquad-\tau_m\Big(
\om^{m+\frac12}\cdot \bs n^{m+\frac12}
+\frac{\gamma_2}{\gamma_1}\tt^{m+\frac12}\cdot \bs n^{m+\frac12},
\;
\big(
\bs n^{m+\frac12}\times D_{\mathcal F}^O(\bs n)\big|^{m+\frac12}
\big)\times \bs n^{m+\frac12}
\Big).
\end{aligned}
\end{equation}
The second term on the right-hand side of \eqref{eq: main-proof-n} and the discrete Ericksen-stress coupling term in \eqref{eq: main-proof-v} cancel with each other. It remains to combine the Leslie stress contribution with the last term in \eqref{eq: main-proof-n}.
By integration by parts and the definition \eqref{eq: main-discrete-sigmaL} of the discrete Leslie stress, one finds
\begin{equation*}
\begin{aligned}
&\Big(\nabla\cdot \widehat{\sig}^{\rm L}\big|^{m+\frac12},\bs v^{m+\frac12}\Big)
-\Big(
\om^{m+\frac12}\cdot \bs n^{m+\frac12}
+\frac{\gamma_2}{\gamma_1}\tt^{m+\frac12}\cdot \bs n^{m+\frac12},
\;
\big(
\bs n^{m+\frac12}\times D_{\mathcal F}^O(\bs n)\big|^{m+\frac12}
\big)\times \bs n^{m+\frac12}
\Big)
\\
=&
-\Big(\alpha_1+\frac{\gamma_2^2}{\gamma_1}\Big)
\Big\|
\tt^{m+\frac12}:\bs n^{m+\frac12}\otimes \bs n^{m+\frac12}
\Big\|^2
-\alpha_4\|\tt^{m+\frac12}\|^2-\Big(\alpha_5+\alpha_6-\frac{\gamma_2^2}{\gamma_1}\Big)
\|\tt^{m+\frac12}\cdot \bs n^{m+\frac12}\|^2.
\end{aligned}
\end{equation*}
In the above equation, the second term on the left-hand side is canceled by the asymmetric part (terms with $D_{\mathcal F}^O$) of the Leslie stress, and the remaining symmetric terms give the Leslie dissipation.

Adding \eqref{eq: main-proof-v}, \eqref{eq: main-proof-n}, and the above equation and noting that the constraint on material parameters \eqref{eq: paracons} and $\gamma \in(0,1)$, it finally leads to
\begin{equation*}
\begin{aligned}
&\frac{\rm Re}{2(1-\gamma)}
\Big(
\|\bs v^{m+1}\|^2-\|\bs v^m\|^2
\Big)
+\mathcal F[\bs n^{m+1}]-\mathcal F[\bs n^m]
\\
=&
-\frac{\gamma \tau_m}{1-\gamma}\|\nabla \bs v^{m+\frac12}\|^2
-\frac{\tau_m}{\gamma_1}
\Big\|
\bs n^{m+\frac12}\times D_{\mathcal F}^O(\bs n)\big|^{m+\frac12}
\Big\|^2
\\
&-\Big(\alpha_1+\frac{\gamma_2^2}{\gamma_1}\Big)\tau_m
\Big\|
\tt^{m+\frac12}:\bs n^{m+\frac12}\otimes \bs n^{m+\frac12}
\Big\|^2
-\alpha_4\tau_m\|\tt^{m+\frac12}\|^2
\\
&-\Big(\alpha_5+\alpha_6-\frac{\gamma_2^2}{\gamma_1}\Big)\tau_m
\|\tt^{m+\frac12}\cdot \bs n^{m+\frac12}\|^2\leq 0.
\end{aligned}
\end{equation*}
This completes the proof.
\end{proof}

\begin{rem}
The rotational discrete gradient scheme for the reformulated EL system is posed on a divergence-free velocity space and is specifically designed to preserve the unit-length constraint and the energy-dissipation law. This formulation is fully consistent with the exact divergence-free spectral approximation to be introduced in Appendix~\ref{app: space}.   The Rdg method is not restricted to the divergence-free spectral discretization. In fact, we further provide in Appendix~\ref{app: pro} a projection-type Rdg variant for spatial discretizations that do not impose the divergence-free constraint exactly. The corresponding proofs of length preservation and energy stability are included there as well. Though the numerical experiments in this paper are based on the divergence-free spectral discretization, the Rdg method can be adapted to more general spatial discretizations and computational domains while preserving the unit-vector constraint and retaining energy stability at the discrete level.
\end{rem}

\section{A fully-discrete exact divergence-free spectral scheme}\label{sect: 4}

We now present a fully-discrete rotational discrete gradient spectral method based on the time-discrete scheme in Section~\ref{sect: 3}. The spatial discretization is designed to retain three intrinsic properties of the EL system at the fully discrete level: the exact incompressibility of the velocity field, the nodal unit-length preservation of the director field, and the discrete energy-dissipation law. To this end, we approximate the velocity in an exact divergence-free spectral space introduced in \cite{qin2023exact}, while the director is approximated in a tensor-product Legendre--Gauss--Lobatto (LGL) nodal space.

Let $\Lambda:=(-1,1)$ and let $d\in\{2,3\}$ denote the spatial dimension. Denote by $\mathbb Q_N(\Lambda^d)$ the tensor-product polynomial space of degree at most $N$ in each variable on $\Lambda^d$. Let $\mathcal T:\Lambda^d\to\Omega$ be the standard affine map from the reference cube onto the rectangular physical domain $\Omega$ and define
\begin{equation*}
X_N(\Omega)
=
\{\phi:\Omega\to\mathbb R\big|\ \phi\circ \mathcal T\in \mathbb Q_N(\Lambda^d)\},
\quad
X_{N,0}(\Omega)=X_N(\Omega)\cap H_0^1(\Omega).
\end{equation*}
For later use, we also denote
$
\mathbb D_N=[X_N(\Omega)]^3,
$ and
$
\mathbb D_{N,0}=[X_{N,0}(\Omega)]^3.
$
For nonhomogeneous Dirichlet data $\bs g$, we use the affine space
$
\mathbb D_{N,\bs g}=\mathbf I_N\bs g+\mathbb D_{N,0},
$
where $\mathbf I_N$ denotes the componentwise interpolation operator based on the LGL points. 

Let $\{\hat x_j\}_{j=0}^N$ and $\{\hat w_j\}_{j=0}^N$ be the LGL nodes and weights on $\Lambda$. The physical quadrature nodes and weights on
$\Omega:=[a_1,b_1]\times[a_2,b_2]\times[a_3,b_3]$ are defined by
\begin{equation*}
x_{i;j_i}=\frac{(b_i-a_i)\hat x_{j_i}+a_i+b_i}{2},
\quad
w_{i;j_i}=\frac{b_i-a_i}{2}\hat w_{j_i},
\quad
0\le j_i\le N,\quad i=1,2,3.
\end{equation*}
Let
$
\mathbb G_N
=
\Big\{
(x_{1;j_1},x_{2;j_2},x_{3;j_3})
,\;\ 0\le j_1,j_2,j_3\le N
\Big\}
$
denote the tensor-product LGL grid on $\Omega$, and $
\mathbb G_N^\partial
=
\mathbb G_N\cap \partial\Omega
$
be the set of boundary LGL nodes. Define the discrete inner product
\begin{equation*}
(\bs u,\bs v)_N
=
\sum_{j_1,j_2,j_3=0}^N
w_{1;j_1}w_{2;j_2}w_{3;j_3}\,
\bs u(x_{1;j_1},x_{2;j_2},x_{3;j_3})
\cdot
\bs v(x_{1;j_1},x_{2;j_2},x_{3;j_3}),
\end{equation*}
and the corresponding discrete norm $\|\bs u\|_N=\sqrt{(\bs u,\bs u)_N}$.

For the velocity field, we adopt the exact divergence-free spectral space $\mathbb V_{N,0}$ introduced in \cite{qin2023exact}, such that
$
\mathbb V_{N,0} \subset \bs H_0^1(\mathrm{div}0;\Omega)
=
\{\bs u\in \bs H_0^1(\Omega):\nabla\cdot \bs u=0\}.
$
Since the divergence-free constraint is built into the approximation space, the pressure is eliminated from the fully discrete weak formulation. For brevity and to keep the presentation focused, we defer the explicit formulas for the divergence-free basis functions to Appendix~\ref{app: space}. For $\bs v$ with nonhomogeneous Dirichlet boundary, i.e. $ \bs v(t,\bs x)\big|_{\partial\Omega}=\bs g(\bs x)$, we use the affine space $\mathbb V_{N,\bs g}=\mathbf I_N\bs g+\mathbb V_{N,0}$ as the approximation space.

The convective term may be written in the standard form
$
(\bs v^{m+\frac12}\cdot \nabla \bs v^{m+\frac12},\bs \phi)
$ at the time-discrete level, as the corresponding energy cancellation follows from incompressibility and integration by parts. However, in order to control quadrature aliasing error and to recover an exact discrete kinetic-energy cancellation relation at the fully discrete level, it is crucial to approximate the convective term by the skew-symmetric trilinear form $b_N$ defined in Lemma \ref{lem: skew-convection-energy}.

\begin{lemma}
\label{lem: skew-convection-energy}
Define \begin{equation}\label{eq: discrete-convection}
b_N(\bs u,\bs v,\bs w)
=
\frac12\big(
(\bs u\cdot \nabla \bs v,\bs w)_N
-
(\bs u\cdot \nabla \bs w,\bs v)_N
\big),
\end{equation}
then it holds that  for any sequence \(\{\bs v^m\}_{m\ge 0}\subset \mathbb V_N\),
\begin{equation}\label{eq: midpoint-kinetic}
\Big(
\frac{\bs v^{m+1}-\bs v^m}{\tau_m},
\bs v^{m+\frac12}
\Big)_N
+
b_N\!\big(\bs v^{m+\frac12},\bs v^{m+\frac12},\bs v^{m+\frac12}\big)
=
\frac{1}{2\tau_m}
\Big(
\|\bs v^{m+1}\|_N^2-\|\bs v^m\|_N^2
\Big).
\end{equation}
\end{lemma}

\begin{proof}
By definition of $b_N(\bs u,\bs v,\bs w)$ in \eqref{eq: discrete-convection}, one can interchange $\bs v$ and $\bs w$ to obtain the property of skew-symmetry with respect to its last two arguments of $b_N$, namely,
\begin{equation}\label{eq: skew-last-two}
b_N(\bs u,\bs v,\bs w)=-\,b_N(\bs u,\bs w,\bs v).
\end{equation}
Taking $\bs w=\bs v$ in the above equation further gives the identity $b_N(\bs u,\bs v,\bs v)=0$. 
Together with the fact that 
$$
\Big(
\bs v^{m+1}-\bs v^m,\,
\bs v^{m+\frac12}
\Big)_N
=
\frac12
\Big(
\|\bs v^{m+1}\|_N^2-\|\bs v^m\|_N^2
\Big),
$$
we finally arrive at \eqref{eq: midpoint-kinetic}.
\end{proof}
\begin{rem}
The nonlinear convection term contributes exactly zero to the discrete kinetic-energy balance. It is worthwhile to point out that this cancellation is purely algebraic and does not rely on the exactness of the quadrature rule. Hence the aliasing error caused by the quadrature rule  is bypassed with the help of this skew-symmetric formulation of the convection term.
\end{rem}

\subsection{Fully-discrete rotational discrete gradient spectral scheme}

\vspace{5pt}
\paragraph{\textbf{Fully-discrete scheme}.}\label{sch: fully1}
For each $m\ge 0$, given
$
(\bs v^m,\bs n^m)\in \mathbb V_{N,\bs g_1} \times \mathbb D_{N,\bs g_2},
$
find
$
(\bs v^{m+1},\bs n^{m+1},\bs\mu_0^{m+\frac12})
\in
\mathbb V_{N,\bs g_1}\times \mathbb D_{N,\bs g_2}\times \mathbb D_{N,0}
$
such that
\begin{subequations}\label{weakformdiv0}
\begin{align}
\begin{split}
&\frac{1}{\tau_m}\big(\bs v^{m+1}-\bs v^m,\bs \phi\big)_N
=
-b_N\!\big(\bs v^{m+\frac12},\bs v^{m+\frac12},\bs \phi\big)
-\frac{\gamma}{\rm Re}\big(\nabla \bs v^{m+\frac12}, \nabla \bs \phi\big)_N
\\
&-\frac{1-\gamma}{\rm Re}\big( \widehat{\sig}^{\mathrm{L}}_0\big|^{m+\frac12},\nabla \bs \phi\big)_N
-\frac{1-\gamma}{\rm Re}\big( \nabla \bs n^{m+\frac12} \cdot [\bs n^{m+\frac12} \times \bs \mu_0^{m+\frac12}] \times \bs n^{m+\frac12}, \bs \phi\big)_N, \;\;\; \forall \bs \phi\in \mathbb V_{N,0},
\end{split}\label{weakvdiv0}\\[5pt]
\begin{split}
& \big(\bs n^{m+1}-\bs n^m,\bs \psi\big)_N
=
\frac{\tau_m}{\gamma_1}
\Big(
\Big\{
\bs n^{m+\frac12}\times
\big[
\bs \mu_0^{m+\frac12}
-\gamma_1(\bs v^{m+\frac12}\cdot \nabla \bs n^{m+\frac12}
+\bs \Omega^{m+\frac12}\cdot \bs n^{m+\frac12})
\\
&
\qquad \qquad \qquad \qquad \qquad \qquad   -\gamma_2 \tt^{m+\frac12}\cdot \bs n^{m+\frac12}
\big]
\Big\}
\times \bs n^{m+\frac12},
\bs \psi
\Big)_N,
\;\;\;
\forall \bs \psi\in \mathbb D_{N,0},
\end{split}\label{weakndiv0}
\end{align}
\end{subequations}
where $\bs\mu_0^{m+\frac12}\in \mathbb D_{N,0}$ is introduced as the constrained weak representative of the Oseen--Frank discrete gradient, and is determined by
\begin{equation}\label{weakom}
\begin{aligned}
 \big(-\bs\mu_0^{m+\frac12},&\bs\theta\big)_N
=
\kappa_1\big(\nabla\cdot \bs n^{m+\frac12},\nabla\cdot \bs\theta\big)_N\\
&+
\kappa_2\big(\beta^{m+\frac12}\nabla\times \bs n^{m+\frac12},\bs\theta\big)_N
+
\kappa_2\big(\beta^{m+\frac12}\bs n^{m+\frac12},\nabla\times \bs\theta\big)_N
\\
&+
\kappa_3\big((\nabla\times \bs n^{m+\frac12})\times
\bs\omega^{m+\frac12}    ,\bs\theta\big)_N+
\kappa_3\big(\bs\omega^{m+\frac12}\times \bs n^{m+\frac12},\nabla\times \bs\theta\big)_N,
\;\;\; \forall\,\bs\theta\in \mathbb D_{N,0}.
\end{aligned}
\end{equation}
In the above, the discrete Leslie stress
$\widehat{\sig}^{\mathrm{L}}_0\big|^{m+\frac12}$ is obtained from
\eqref{eq: main-discrete-sigmaL} by replacing
$D_{\mathcal F}^O(\bs n)\big|^{m+\frac12}$ with
$-\bs\mu_0^{m+\frac12}$.
 
\begin{lemma}\label{lemma: constrained-dg-identity}
Denote the discrete Oseen-Frank energy at time step $k$ by
\begin{equation}\label{eq: discrete-OF-energy}
\mathcal F_N[\bs n^k]=
\frac12\big(
\kappa_1\|\nabla\cdot \bs n^k\|_N^2
+\kappa_2\|\bs n^k\cdot \nabla\times \bs n^k\|_N^2
+\kappa_3\|\bs n^k\times \nabla\times \bs n^k\|_N^2
\big),\;\;\; k=m,\, m+1,
\end{equation}
and assume that $\bs g_2$ is independent of time and
$\bs n^m,\bs n^{m+1}\in \mathbb D_{N,\bs g_2}$.
Then for $\bs \mu_0^{m+\frac12}\in \mathbb D_{N,0}
$ defined in equation \eqref{weakom}, it satisfies the discrete energy difference relation:
\begin{equation}\label{eq: constrained-dg-identity}
\big(-\bs \mu_0^{m+\frac12},\bs n^{m+1}-\bs n^m\big)_N
=
\mathcal F_N[\bs n^{m+1}]
-
\mathcal F_N[\bs n^m].
\end{equation}
\end{lemma}

\begin{proof}
Since $\bs g_2$ is time-independent and $\bs n^m,\bs n^{m+1}\in \mathbb D_{N,\bs g_2}$, we can take $\bs\theta=\bs n^{m+1}-\bs n^m \in \mathbb D_{N,0}$  in \eqref{weakom}. 
We detail the manipulation for the $\kappa_3$ term (bend) as a representative case; the $\kappa_1$ term (splay) and $\kappa_2$ term (twist) follow by analogous computations.
\begin{equation*}
\begin{aligned}
\bs\omega^{m+1}-\bs\omega^m
=
(\bs n^{m+1}-\bs n^m)\times \nabla\times \bs n^{m+\frac12}
+
\bs n^{m+\frac12}\times \nabla\times (\bs n^{m+1}-\bs n^m).
\end{aligned}
\end{equation*}
Then by the above equation and the cyclic identity~\eqref{id: crossid2}, one can verify that
\begin{equation}\label{eq: bend}
\begin{aligned}
&\big((\nabla\times \bs n^{m+\frac12})\times
\bs\omega^{m+\frac12},\bs n^{m+1}-\bs n^m\big)_N
+
\big(\bs\omega^{m+\frac12}\times \bs n^{m+\frac12},\nabla\times (\bs n^{m+1}-\bs n^m)\big)_N
\\
&=
\big(\bs\omega^{m+\frac12},
(\bs n^{m+1}-\bs n^m)\times \nabla\times \bs n^{m+\frac12}
+
\bs n^{m+\frac12}\times \nabla\times (\bs n^{m+1}-\bs n^m)
\big)_N
\\
&=
\big(\bs\omega^{m+\frac12},\bs\omega^{m+1}-\bs\omega^m\big)_N
=
\frac12\Big(
\|\bs\omega^{m+1}\|_N^2-\|\bs\omega^m\|_N^2
\Big).
\end{aligned}
\end{equation}
Similarly, the $\kappa_1$ and $\kappa_2$ terms are  
\begin{equation}\label{eq: splay-twist}
\frac12\Big(
\|\nabla\cdot\bs n^{m+1}\|_N^2-\|\nabla\cdot\bs n^m\|_N^2
\Big),\quad
\frac12\Big(
\|\beta^{m+1}\|_N^2-\|\beta^m\|_N^2
\Big).
\end{equation}
Combining $\kappa_1$, $\kappa_2$, $\kappa_3$ terms \eqref{eq: bend}, \eqref{eq: splay-twist}, we obtain \eqref{eq: constrained-dg-identity}.
\end{proof}

\subsection{Structure-preserving properties}

We now establish the preservation of the unit-length constraint at nodal points and the discrete energy law for the fully-discrete scheme. The preservation of the unit-length constraint relies on the diagonal structure of the discrete mass matrix on the LGL nodal basis. For later use, we record this property in the following lemma.

\begin{lemma}\label{lem: diagonal-mass}
Let $\{\ell_{j_i}\}_{j_i=0}^N$ be the one-dimensional LGL cardinal basis on the $i$th coordinate direction, and define the tensor-product nodal basis on $\Omega$ by
\begin{equation}\label{eq: tpnb}
L_{j_1j_2j_3}(x_1,x_2,x_3)
=
\ell_{j_1}(x_1)\ell_{j_2}(x_2)\ell_{j_3}(x_3),
\;\; 0\le j_1,j_2,j_3\le N.
\end{equation}
Assume that the discrete inner product $(\cdot,\cdot)_N$ is defined by the tensor-product LGL quadrature rule on the same set of nodes. Then the corresponding discrete mass matrix is diagonal, namely
\begin{equation}\label{eq: diagonal-mass}
\big(L_{j_1j_2j_3},L_{k_1k_2k_3}\big)_N
=
\delta_{j_1k_1}\delta_{j_2k_2}\delta_{j_3k_3}\,
w_{1;j_1}w_{2;j_2}w_{3;j_3}.
\end{equation}
Consequently, for any vector-valued nodal polynomial $\bs u_N\in \mathbb D_N$ and any fixed vector $\bs a\in \mathbb R^3$, one has
\begin{equation}\label{eq: nodal-extraction}
\big(\bs u_N,\,L_{j_1j_2j_3}\bs a\big)_N
=
w_{1;j_1}w_{2;j_2}w_{3;j_3}\,
\bs u_N(x_{1;j_1},x_{2;j_2},x_{3;j_3})\cdot \bs a.
\end{equation}
\end{lemma}

\begin{proof}
Since $\ell_{j_i}(x_{i;k_i})=\delta_{j_i k_i}$ for each coordinate direction, we have
\[
L_{j_1j_2j_3}(x_{1;k_1},x_{2;k_2},x_{3;k_3})
=
\delta_{j_1k_1}\delta_{j_2k_2}\delta_{j_3k_3}.
\]
Substituting this into the definition of $(\cdot,\cdot)_N$ gives
\[
\big(L_{j_1j_2j_3},L_{k_1k_2k_3}\big)_N
=
\sum_{m_1,m_2,m_3=0}^N
w_{1;m_1}w_{2;m_2}w_{3;m_3}\,
\delta_{j_1m_1}\delta_{j_2m_2}\delta_{j_3m_3}\,
\delta_{k_1m_1}\delta_{k_2m_2}\delta_{k_3m_3},
\]
which immediately yields \eqref{eq: diagonal-mass}. Then \eqref{eq: nodal-extraction} follows by expanding $\bs u_N$ in the nodal basis and applying \eqref{eq: diagonal-mass} componentwise.
\end{proof}

We present the properties of the fully-discrete schemes in Theorem \ref{thm: fully-discrete-structure}.

\begin{thm}\label{thm: fully-discrete-structure}
Assume that the material coefficients satisfy \eqref{eq: paracons} and $\gamma\in(0,1)$. Assume further that the director field is approximated by the tensor-product LGL nodal basis $L_{j_1j_2j_3}(x_1,x_2,x_3)$ in \eqref{eq: tpnb} associated with the same tensor-product LGL quadrature points $\mathbb G_N$ for $(\cdot,\cdot)_N$, and that the initial and boundary data satisfy
\begin{equation}\label{eq: fullyini}
\begin{aligned}
& |\bs n^0(x_{1;j_1},x_{2;j_2},x_{3;j_3})|=1,
\quad \forall (x_{1;j_1},x_{2;j_2},x_{3;j_3})\in \mathbb G_N,\\
& |\mathbf I_N\bs g_2(x_{1;j_1},x_{2;j_2},x_{3;j_3})|=1,
\quad \forall (x_{1;j_1},x_{2;j_2},x_{3;j_3}) \in \mathbb G_N^\partial,
\end{aligned}
\end{equation}
where $\bs g_2$ is independent of time. Then the fully-discrete scheme \eqref{weakformdiv0}--\eqref{weakom} preserves the unit-length constraint at the LGL nodes, namely
\begin{equation}\label{eq: nodal-length}
|\bs n^m(x_{1;j_1},x_{2;j_2},x_{3;j_3})|=1,
\quad \forall (x_{1;j_1},x_{2;j_2},x_{3;j_3})\in \mathbb G_N,\quad m\ge 0.
\end{equation}
Assume further that the velocity field is under homogeneous Dirichlet boundary condition, i.e. $\bs g_1=\bs 0$, it also satisfies the discrete energy law
\begin{equation}\label{eq: fullyenergydissp}
\frac{\rm Re}{2(1-\gamma)}\|\bs v^{m+1}\|_N^2+\mathcal F_N[\bs n^{m+1}]
\le
\frac{\rm Re}{2(1-\gamma)}\|\bs v^m\|_N^2+\mathcal F_N[\bs n^m],
\quad \forall m\ge 0.
\end{equation}
\end{thm}

\begin{proof}
For an interior LGL node
$
\bs x_{j_1j_2j_3}=(x_{1;j_1},x_{2;j_2},x_{3;j_3})
\in \mathbb G_N^{\rm in}:=\mathbb G_N\setminus \mathbb G_N^\partial,
$
take $\bs\psi=L_{j_1j_2j_3}\bs e_r$ in \eqref{weakndiv0}, where $\bs e_r$ is the $r$th canonical basis vector of $\mathbb R^3$. By Lemma~\ref{lem: diagonal-mass}, we obtain the nodal relation
$$
\bs n^{m+1}(\bs x_{j_1j_2j_3})-\bs n^m(\bs x_{j_1j_2j_3})
=
\frac{\tau_m}{\gamma_1}\,
\bs q^{m+\frac12}(\bs x_{j_1j_2j_3}),
\quad \bs x_{j_1j_2j_3}\in \mathbb G_N^{\rm in},
$$
where
$$
\bs q^{m+\frac12}
=
\Big\{
\bs n^{m+\frac12}\times
\Big[
\bs \mu_0^{m+\frac12}
-\gamma_1\big(\bs v^{m+\frac12}\cdot \nabla \bs n^{m+\frac12}
+\bs\Omega^{m+\frac12}\cdot \bs n^{m+\frac12}\big)
-\gamma_2\tt^{m+\frac12}\cdot \bs n^{m+\frac12}
\Big]
\Big\}\times \bs n^{m+\frac12}.
$$
Since $\bs q^{m+\frac12}\cdot \bs n^{m+\frac12}=0$ pointwise, it follows that
$$
|\bs n^{m+1}(\bs x_{j_1j_2j_3})|^2-|\bs n^m(\bs x_{j_1j_2j_3})|^2=0,
\quad \forall\, \bs x_{j_1j_2j_3}\in \mathbb G_N^{\rm in}.
$$
At boundary nodes, the director values are fixed by the time-independent Dirichlet data $\mathbf I_N\bs g_2$. Hence \eqref{eq: nodal-length} follows by induction from \eqref{eq: fullyini}.

We next prove the discrete energy law. Taking
$
\bs \phi=({\rm Re\bs v^{m+\frac12}})/({1-\gamma})
$
in \eqref{weakvdiv0}, and using the skew-symmetry identity \eqref{eq: midpoint-kinetic}, yields
\begin{equation}\label{eq: fd-v-energy-mu0}
\begin{aligned}
\frac{\rm Re}{2(1-\gamma)}
\Big(\|\bs v^{m+1}\|_N^2-\|\bs v^m\|_N^2\Big)
&=
-\frac{\gamma\tau_m}{1-\gamma}\|\nabla \bs v^{m+\frac12}\|_N^2
-\tau_m\big(\widehat{\sig}^{\rm L}_0\big|^{m+\frac12},\nabla \bs v^{m+\frac12}\big)_N
\\
&\quad
-\tau_m\Big(
\nabla \bs n^{m+\frac12}\cdot
\big[(\bs n^{m+\frac12}\times \bs \mu_0^{m+\frac12})\times \bs n^{m+\frac12}\big],
\bs v^{m+\frac12}
\Big)_N.
\end{aligned}
\end{equation}
Next, take $\bs\psi=\bs\mu_0^{m+\frac12}\in \mathbb D_{N,0}$ in \eqref{weakndiv0}, one has
\begin{equation}\label{eq: ndiff}
\begin{aligned}
& \big(\bs n^{m+1}-\bs n^m,\bs\mu_0^{m+\frac12} \big)_N
=
\frac{\tau_m}{\gamma_1}
\Big(
\Big\{
\bs n^{m+\frac12}\times
\big[
\bs \mu_0^{m+\frac12}
-\gamma_1(\bs v^{m+\frac12}\cdot \nabla \bs n^{m+\frac12}
+\bs \Omega^{m+\frac12}\cdot \bs n^{m+\frac12})
\\
&
\qquad \qquad \qquad \qquad \qquad \qquad   -\gamma_2 \tt^{m+\frac12}\cdot \bs n^{m+\frac12}
\big]
\Big\}
\times \bs n^{m+\frac12},
\bs\mu_0^{m+\frac12}
\Big)_N.
\end{aligned}
\end{equation}
By Lemma~\ref{lemma: constrained-dg-identity}, it holds
$$
\big(-\bs \mu_0^{m+\frac12},\bs n^{m+1}-\bs n^m\big)_N
=
\mathcal F_N[\bs n^{m+1}] - \mathcal F_N[\bs n^m].
$$
Moreover, it is straightforward to verify that
\begin{equation*}
\Big(\bs n^{m+\frac12}\times
\bs \mu_0^{m+\frac12}\times \bs n^{m+\frac12}, \bs\mu_0^{m+\frac12} \Big)_N=
\|\bs n^{m+\frac12}\times \bs \mu_0^{m+\frac12}\|_N^2.
\end{equation*}
By repeatedly using identity \eqref{id: crossid2}, we have that
\begin{equation*}
\begin{aligned}
&\Big(\bs n^{m+\frac12}\times (\bs v^{m+\frac12}\cdot \nabla \bs n^{m+\frac12})\times \bs n^{m+\frac12},
\bs\mu_0^{m+\frac12}
\Big)_N=\Big( \bs v^{m+\frac12}\cdot \nabla \bs n^{m+\frac12},
\bs n^{m+\frac12}\times \bs\mu_0^{m+\frac12}\times \bs n^{m+\frac12}
\Big)_N\\
&\qquad \qquad \qquad \qquad \qquad  \qquad \qquad \qquad \qquad \qquad \;\; =\Big(
\nabla \bs n^{m+\frac12}\cdot
(\bs n^{m+\frac12}\times \bs\mu_0^{m+\frac12}\times \bs n^{m+\frac12}),
\bs v^{m+\frac12}
\Big)_N.
\end{aligned}
\end{equation*}
Hence \eqref{weakndiv0} gives
\begin{equation*}
\begin{aligned}
\big(\bs n^{m+1}-\bs n^m,&\bs \mu_0^{m+\frac12}\big)_N
=
\frac{\tau_m}{\gamma_1}
\|\bs n^{m+\frac12}\times \bs \mu_0^{m+\frac12}\|_N^2
-\tau_m\Big(\!
\nabla \bs n^{m+\frac12}\!\cdot\!
\big[(\bs n^{m+\frac12}\times \bs \mu_0^{m+\frac12})\times \bs n^{m+\frac12}\big],
\bs v^{m+\frac12}\!
\Big)_N
\\
&\qquad \qquad \;\;\;
-\tau_m\Big(
\bs \Omega^{m+\frac12}\cdot \bs n^{m+\frac12}
+\frac{\gamma_2}{\gamma_1}\tt^{m+\frac12}\cdot \bs n^{m+\frac12},
(\bs n^{m+\frac12}\times \bs \mu_0^{m+\frac12})\times \bs n^{m+\frac12}
\Big)_N.
\end{aligned}
\end{equation*}
Substituting the above equations into \eqref{eq: ndiff}, we arrive at
\begin{equation}\label{eq: fd-n-energy-mu0}
\begin{aligned}
\mathcal F_N[\bs n^{m+1}] - \mathcal F_N[\bs n^m]
&=
-\frac{\tau_m}{\gamma_1}
\|\bs n^{m+\frac12}\times \bs \mu_0^{m+\frac12}\|_N^2
\\
&
+\tau_m\Big(
\nabla \bs n^{m+\frac12}\cdot
\big[(\bs n^{m+\frac12}\times \bs \mu_0^{m+\frac12})\times \bs n^{m+\frac12}\big],
\bs v^{m+\frac12}
\Big)_N
\\
&
+\tau_m\Big(
\bs \Omega^{m+\frac12}\cdot \bs n^{m+\frac12}
+\frac{\gamma_2}{\gamma_1}\tt^{m+\frac12}\cdot \bs n^{m+\frac12},
(\bs n^{m+\frac12}\times \bs \mu_0^{m+\frac12})\times \bs n^{m+\frac12}
\Big)_N.
\end{aligned}
\end{equation}
Finally, using the definition of $\widehat{\sig}^{\rm L}_0\big|^{m+\frac12}$ together with the relations $\alpha_2+\alpha_3=\gamma_2$ and $\alpha_3-\alpha_2=\gamma_1$, the same pointwise algebra structure as in the proof of Theorem~\ref{thm: main-scheme} yields
\begin{equation*}
\begin{aligned}
&-\big(\widehat{\sig}^{\rm L}_0\big|^{m+\frac12},\nabla \bs v^{m+\frac12}\big)_N
+\Big(
\bs \Omega^{m+\frac12}\cdot \bs n^{m+\frac12}
+\frac{\gamma_2}{\gamma_1}\tt^{m+\frac12}\cdot \bs n^{m+\frac12},
(\bs n^{m+\frac12}\times \bs \mu_0^{m+\frac12})\times \bs n^{m+\frac12}
\Big)_N
\\
&=
-\Big(\alpha_1+\frac{\gamma_2^2}{\gamma_1}\Big)
\Big\|
\tt^{m+\frac12}:\bs n^{m+\frac12}\otimes \bs n^{m+\frac12}
\Big\|_N^2
-\alpha_4\|\tt^{m+\frac12}\|_N^2
-\Big(\alpha_5+\alpha_6-\frac{\gamma_2^2}{\gamma_1}\Big)
\|\tt^{m+\frac12}\cdot \bs n^{m+\frac12}\|_N^2.
\end{aligned}
\end{equation*}
Adding \eqref{eq: fd-v-energy-mu0}, \eqref{eq: fd-n-energy-mu0}, and the above equation, we obtain
\begin{equation*}
\begin{aligned}
&\frac{\rm Re}{2(1-\gamma)}
\Big(\|\bs v^{m+1}\|_N^2-\|\bs v^m\|_N^2\Big)
+\mathcal F_N[\bs n^{m+1}] - \mathcal F_N[\bs n^m]
\\
&=
-\frac{\gamma\tau_m}{1-\gamma}\|\nabla \bs v^{m+\frac12}\|_N^2
-\frac{\tau_m}{\gamma_1}\|\bs n^{m+\frac12}\times \bs \mu_0^{m+\frac12}\|_N^2
\\
&\quad
-\Big(\alpha_1+\frac{\gamma_2^2}{\gamma_1}\Big)\tau_m
\Big\|
\tt^{m+\frac12}:\bs n^{m+\frac12}\otimes \bs n^{m+\frac12}
\Big\|_N^2
-\alpha_4\tau_m\|\tt^{m+\frac12}\|_N^2
\\
&\quad
-\Big(\alpha_5+\alpha_6-\frac{\gamma_2^2}{\gamma_1}\Big)\tau_m
\|\tt^{m+\frac12}\cdot \bs n^{m+\frac12}\|_N^2
\le 0,
\end{aligned}
\end{equation*}
which ends the proof. 
\end{proof}

\subsection{Nonlinear solver and adaptive time-stepping strategy}\label{sect: ink}

The fully discrete system \eqref{weakvdiv0}-\eqref{weakom} is nonlinear.
At each time step, we therefore solve the coupled algebraic system for the unknown variables
$(\bs v^{m+1},\bs n^{m+1},\bs \mu_0^{m+\frac12})$
by an inexact Newton--Krylov (INK) method; see \cite{kelley1995iterative}.
In all computations, the nonlinear iteration is terminated when the relative residual is below $10^{-10}$.

Theorem~\ref{thm: fully-discrete-structure} shows that the fully-discrete Rdg-spectral scheme preserves
the discrete unit-length constraint and satisfies the discrete energy law for any positive time step $\tau_m$.
This makes it natural to combine the scheme with a variable-step strategy.
To resolve rapid transient stages efficiently while avoiding unnecessarily small time steps in slowly varying regimes,
we adopt the energy-based adaptive rule of \cite{xu2024second}:
\begin{equation}\label{eq: timeadpt}
\tau_{m+1}
=
\max\left\{
\tau_{\min},
\frac{\tau_{\max}}
{\sqrt{1+\alpha_{\tau}\left|\left(E_N^m-E_N^{m-1}\right)/\tau_m\right|^2}}
\right\},
\quad m\ge 1,
\end{equation}
where
\begin{equation}\label{eq: discrete-total-energy}
E_N^m
=
\frac{\mathrm{Re}}{2(1-\gamma)}\|\bs v^m\|_N^2+\mathcal F_N[\bs n^m].
\end{equation}
Here $\alpha_\tau>0$ controls the sensitivity of the time-step update, and
$\tau_{\min}$ and $\tau_{\max}$ denote prescribed lower and upper bounds of the time-step, respectively.
For the first adaptive update, we set $\tau_1=\tau_0$.

In the convergence tests of Section~\ref{sect: conv}, we use uniform time steps in order to measure the temporal order accurately. In Sections~\ref{sect: sp}-\ref{sect: sf}, we switch to the adaptive time-step strategy \eqref{eq: timeadpt},
which is particularly effective for the multistage dynamics induced by anisotropic elasticity and shear flow.

\section{Representative numerical experiments}
In this section we present representative numerical results for the full Ericksen--Leslie model to assess the structure-preserving Rdg-spectral method developed in the previous sections. The tests include demonstration of the spatial and temporal convergence rates, verification of the discrete unit-length preservation and energy dissipation law, and simulations illustrating the effects of anisotropic Oseen--Frank elasticity and shear flows. The results demonstrate that the proposed method provides an accurate and robust numerical tool for the simulation of full Ericksen--Leslie system beyond the one-constant  approximation and reduced-stress settings.

Unless otherwise specified, the Leslie coefficients are fixed as
$$
(\alpha_1,\alpha_2,\alpha_3,\alpha_4,\alpha_5,\alpha_6)=(1,0.25,1.25,1,1.5,3), 
$$
so that $\gamma_1=1$ and $\gamma_2=1.5$ based on the requirement \eqref{eq: coeff2}.
We also take $\gamma=0.5$ and $\mathrm{Re}=0.8$.
At each time step, the nonlinear algebraic system is solved by the INK method
described in Section~\ref{sect: ink}, and the nonlinear iteration is terminated once the relative residual falls below $10^{-10}$.

\subsection{Convergence tests}\label{sect: conv}

The goal of this subsection is to demonstrate numerically the spatial and temporal convergence rates of the proposed Rdg-spectral method developed herein using a contrived analytic solution.

Consider the computational domain $\Omega=[-1,1]^3$, we assume the following analytic expressions for $(\bs v,\bs n)$ of the reformulated EL system~\eqref{eq: elmodelrot}:
\begin{equation}\label{eq: contrivd}
	\begin{aligned}
	v_1=&10x_2x_3\big(x_1^2-1\big)^5\big(x_2^2-1\big)^4\big(x_3^2-1\big)^4t+e^{x_2x_3},\\
		v_2=&-5x_1x_3\big(x_1^2-1\big)^4\big(x_2^2-1\big)^5\big(x_3^2-1\big)^4t+e^{x_1x_3},\\
		v_3=&-5x_1x_2\big(x_1^2-1\big)^4\big(x_2^2-1\big)^4\big(x_3^2-1\big)^5t+e^{x_1x_2},\\
		n_1=&\sin\Big(\big(x_1^2-1\big)\big(x_2^2-1\big)\big(x_3^2-1\big)t\Big)\cos\Big(\big(x_1^2-1\big)\big(x_2^2-1\big)\big(x_3^2-1\big)t+\frac{\pi}{3}\Big),\\
		n_2=&\sin\Big(\big(x_1^2-1\big)\big(x_2^2-1\big)\big(x_3^2-1\big)t\Big)\sin\Big(\big(x_1^2-1\big)\big(x_2^2-1\big)\big(x_3^2-1\big)t+\frac{\pi}{3}\Big),\\
		n_3=&\cos\Big(\big(x_1^2-1\big)\big(x_2^2-1\big)\big(x_3^2-1\big)t\Big).
	\end{aligned}
\end{equation}
It is straightforward to verify  that the contrived velocity field $\bs v$ is divergence-free, while the contrived  director field $\bs n$ satisfies the unit-length constraint $|\bs n|=1$. 
The forcing terms $\bs h_1(x,t)$ and $\bs h_2(x,t)$ are then defined so that \eqref{eq: contrivd}
solves the forced reformulated system associated with \eqref{eq: elmodelrot}, namely
\begin{align*}
		& \bs h_1(\bs x,t)=\bs v_t+\bs v \cdot \nabla \bs v+\nabla P-\frac{\gamma}{\rm Re} \Delta \bs v-\frac{1-\gamma}{\rm Re} \nabla \cdot \widehat{\sig}^{\mathrm{L}}-\frac{1-\gamma}{\rm Re} \nabla \bs n \cdot\bigg[\bigg(\bs n \times \frac{\delta \mathcal{F}}{\delta \bs n}\bigg) \times \bs n\bigg], \\*
		& \bs h_2(\bs x,t)=\bs n_t-\frac{1}{\gamma_1}\Big( \bs n \times\big(\bs \mu-\gamma_1(\bs v \cdot \nabla \bs n+\om \cdot \bs n)-\gamma_2 \tt \cdot \bs n\big)\Big) \times \bs n .
\end{align*}
Although the pressure is eliminated from the discrete formulation through the exact divergence-free spectral basis, it is needed in the evaluation of $\bs h_1$ in the convergence test. For this purpose, we take
\begin{equation*}
	P=\big(x_1^2-1\big)\big(x_2^2-1\big)\big(x_3^2-1\big)t.
\end{equation*}

For the spatial convergence test, we fix the time step size at $\tau=10^{-4}$ and increase the polynomial order from $N=4$ to $N=24$. Figure~\ref{figs: spatialtest} shows the $L^\infty$-errors of $\bs n$ and $\bs v$ at $t=0.2$ as functions of $N$. The errors decay exponentially as $N$ increases until they level off at around $10^{-10}$,  which is likely caused by the combined effect of temporal discretization, nonlinear solver tolerance, and roundoff.
\begin{figure}[tbp]
\begin{center}
	\subfigure[$L^{\infty}$-error of $\bs n$]{ \includegraphics[scale=.54]{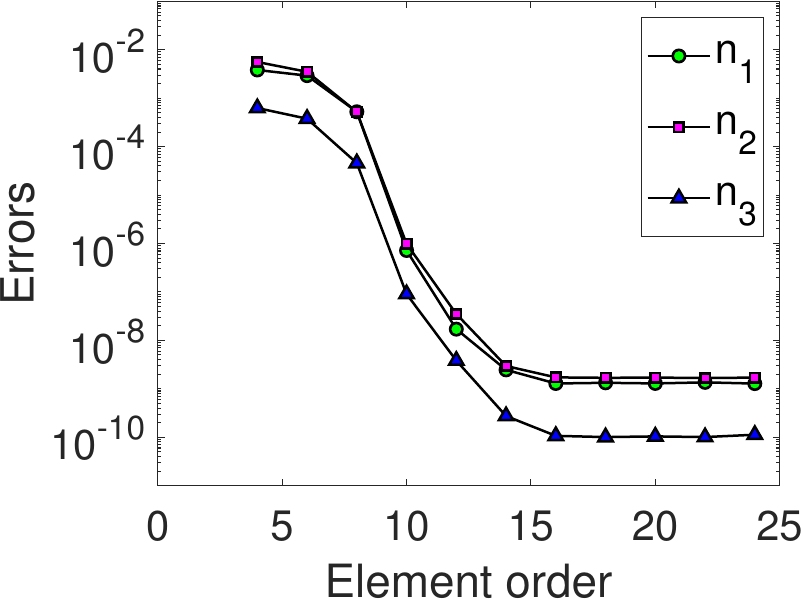}}
	\subfigure[$L^{\infty}$-error of $\bs v$ ]{ \includegraphics[scale=.54]{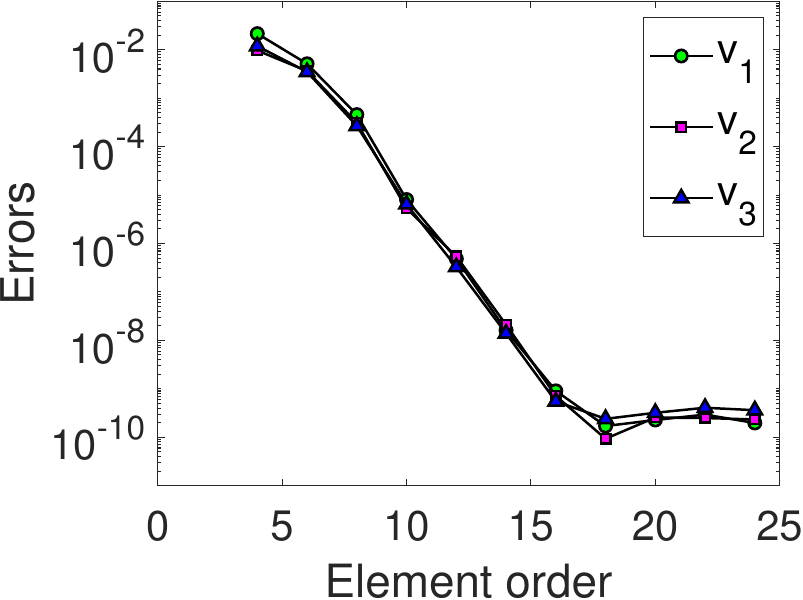}}
   \caption{\small Spatial convergence test: $L^\infty$-errors of the Rdg-spectral method versus the polynomial order $N$ for (a) $\bs n$ and (b) $\bs v$.} 
	 \label{figs: spatialtest}
\end{center}
\end{figure}

For the temporal convergence test, we fix $N=30$ and decrease the time step size from $10^{-1}$ to
$7.8125\times 10^{-4}$ by successive halving.
The resulting $L^\infty$-errors at $t=0.2$ are displayed in Figure~\ref{figs: temporaltest}.
The slopes of the error curves agree with the second-order reference line,
confirming the expected second-order temporal accuracy of the midpoint Rdg discretization.
\begin{figure}[tbp]
\begin{center}
	\subfigure[$L^{\infty}$-error of $\bs n$]{ \includegraphics[scale=.54]{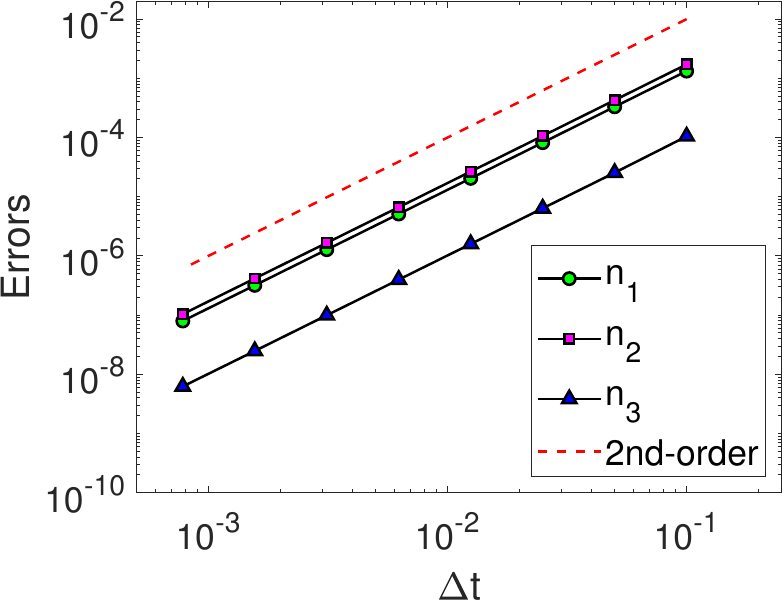}}
	\subfigure[$L^{\infty}$-error of $\bs v$ ]{ \includegraphics[scale=.54]{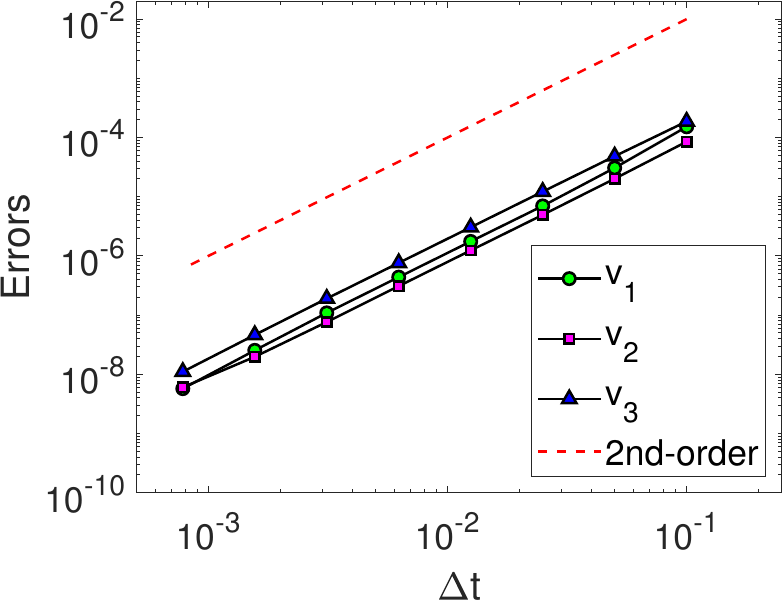}}
   \caption{\small Temporal convergence test: $L^\infty$-errors of the Rdg-spectral method versus the time step size $\tau$ for (a) $\bs n$ and (b) $\bs v$. }
	 \label{figs: temporaltest}
\end{center}
\end{figure}

\subsection{Structure-preserving test}\label{sect: sp}
We next verify the discrete structure-preserving properties of the fully-discrete Rdg scheme.
The computational domain is again taken as $\Omega=[-1,1]^3$, and we restrict attention to solutions
that are independent of the $x_3$-direction.
The polynomial order is set to $N=30$ in the $x_1$- and $x_2$-directions.
For time integration, we use the adaptive time-stepping strategy introduced in Section~\ref{sect: ink} with $\alpha_{\tau}=10^{-3}$, $\tau_{\text{max}}=2\times 10^{-4}$ and $\tau_{\text{min}}=10^{-6}$. 

We first choose the elastic coefficients
$
(\kappa_1,\kappa_2,\kappa_3)=(0.1,\,0.5,\,2.5),
$
and take the initial director field $\bs n_0=(n_1,n_2,n_3)^{\intercal}$ as
\begin{equation}
\label{eq: init-n-1}
\begin{aligned}
n_1 &= \sin\!\big(\pi(1-x_1^2)(1-x_2^2)\big)\cos(\pi x_1),\\
n_2 &= \sin\!\big(\pi(1-x_1^2)(1-x_2^2)\big)\sin(\pi x_1),\\
n_3 &= \cos\!\big(\pi(1-x_1^2)(1-x_2^2)\big),
\end{aligned}
\end{equation}
together with the initial velocity field 
\begin{equation}
\label{eq: init-v-1}
\bs v_0
=
\Big(
10x_2(x_1^2-1)^2(x_2^2-1),\,
-10x_1(x_1^2-1)(x_2^2-1)^2,\,
0
\Big).
\end{equation}
A homogeneous Dirichlet boundary condition is imposed on the velocity field, while a time-independent
nonhomogeneous Dirichlet boundary condition is imposed on the director field.
Under these boundary conditions, the fully-discrete Rdg scheme satisfies the discrete energy law established in Theorem~\ref{thm: fully-discrete-structure}.

We first depict the side views (with azimuth=$120^\circ$, elevation=$30^\circ$) of the initial director field and velocity field in Figure \ref{figs: case1} (a) and (b), respectively. The discrete total energy $E_N^m$ defined  in \eqref{eq: discrete-total-energy} is plotted in Figure~\ref{figs: case1} (c). It shows that the discrete energy decreases monotonically without spurious oscillations, which is in agreement with the discrete energy law of Theorem~\ref{thm: fully-discrete-structure}. Figure~\ref{figs: case1} (d) displays the history of the unit-length error at LGL points, namely, 
$
\max_{\bs x_j\in \mathbb G_N}\big||\bs n^m(\bs x_j)|-1\big|.
$
The error remains at the level of approximately $10^{-10}$ throughout the computation,
which is consistent with the stopping tolerance of the INK solver.
This indicates that the observed unit-length error is dominated by the nonlinear solver tolerance,
rather than by the time or space discretization itself. The snapshots at $t=15$ in Figure~\ref{figs: case1} (e)-(f) show that the director field has relaxed to a stationary low-energy configuration:
it is close to being uniform in the interior, while the velocity magnitude has decayed to the order of $10^{-9}$ and can be characterized by four symmetric rotating vortices. Taken together, these results confirm that the fully-discrete Rdg scheme preserves the unit-length constraint and energy dissipation law at the discrete level.
\begin{figure}[tbp]
\begin{center}
	\subfigure[Initial profile of director]{ \includegraphics[scale=.34]{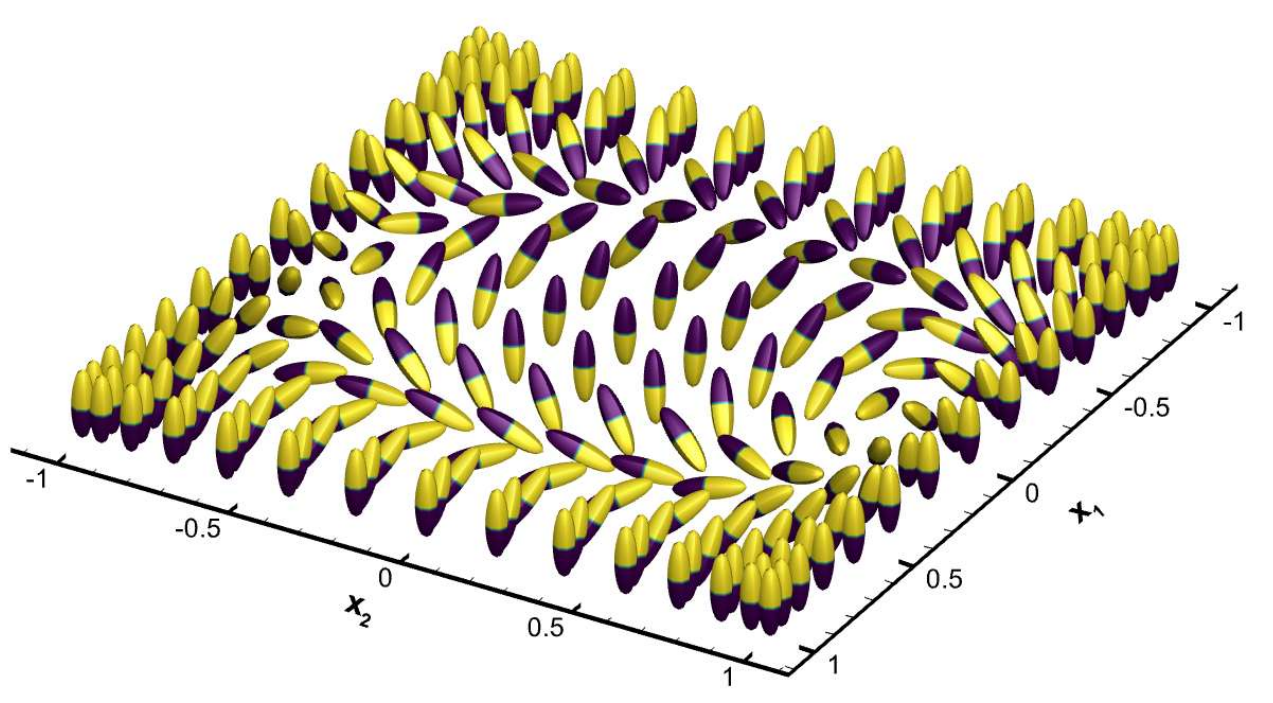}}
	\subfigure[Initial profile of velocity]{ \includegraphics[scale=.38]{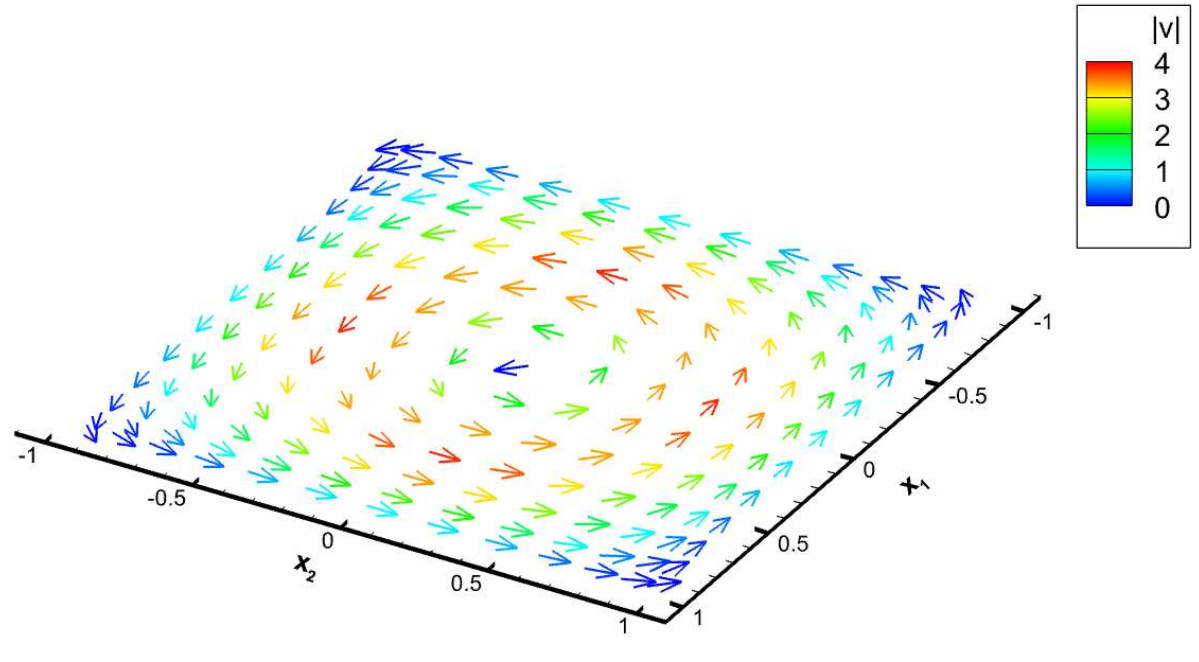}}\\	
	\subfigure[Total energy vs Time]{ \includegraphics[scale=.46]{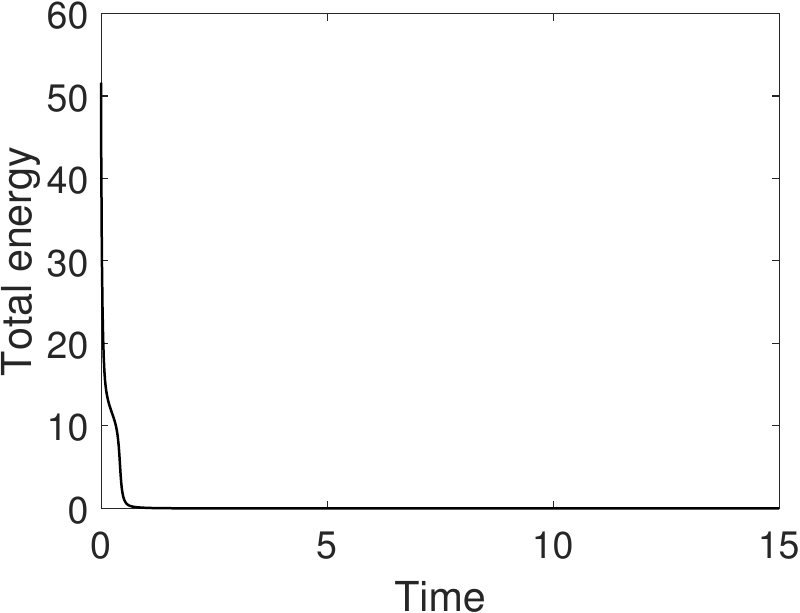}}
	\subfigure[Length error vs Time]{ \includegraphics[scale=.46]{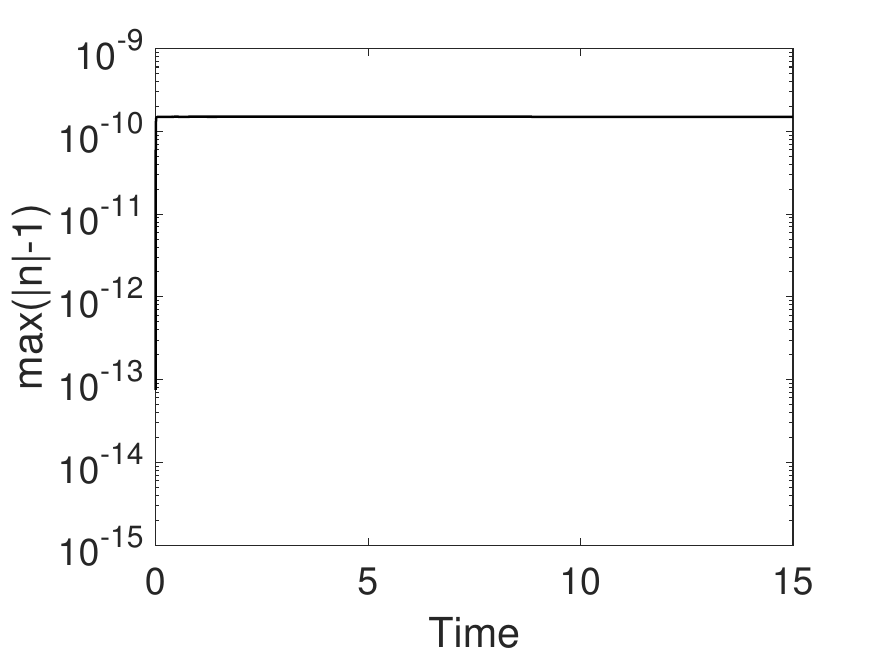}}\\
	\subfigure[$\bs n$ at $t=15$]{ \includegraphics[scale=.34]{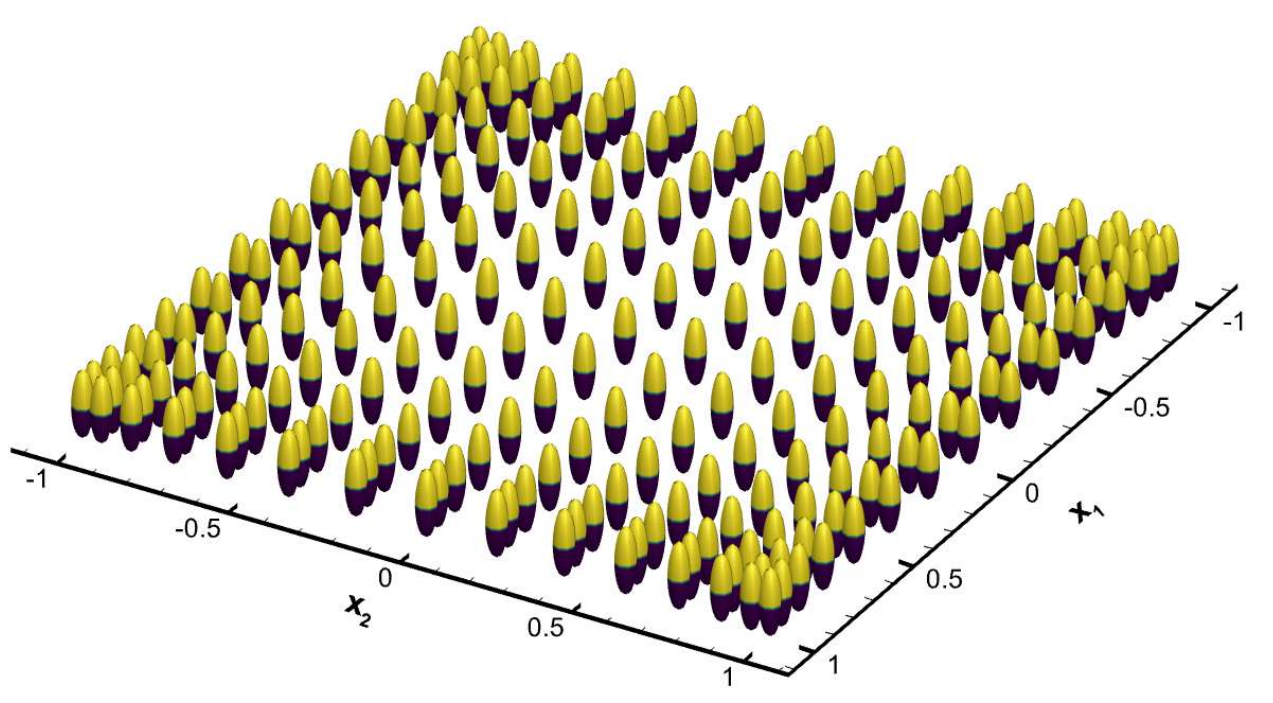}}
	\subfigure[$\bs v$ at $t=15$]{ \includegraphics[scale=.38]{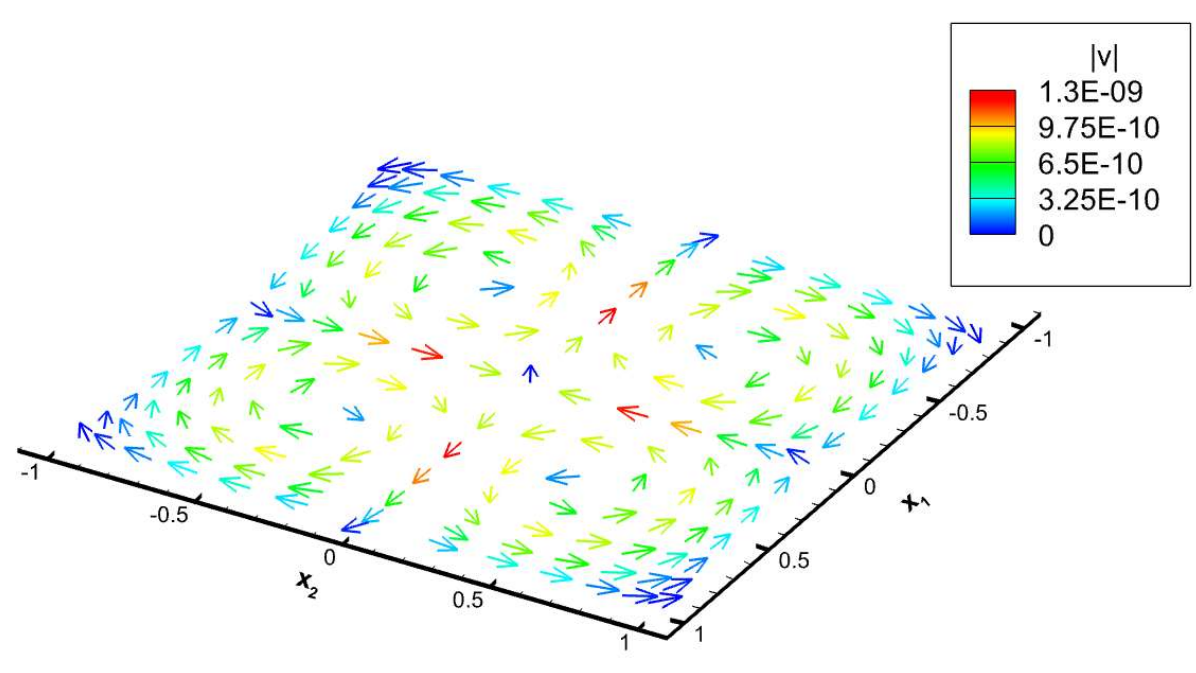}}\\
   \caption{\small Structure-preserving test with initial data \eqref{eq: init-n-1}--\eqref{eq: init-v-1} and elastic coefficients $(\kappa_1,\kappa_2,\kappa_3)=(0.1,0.5,2.5)$. (a) Initial director field; (b) initial velocity field; (c) history of the discrete total energy $E_N^m$; (d) history of the nodal length error $\max_{\bs x_j\in \mathbb G_N}\big||\bs n(\bs x_j)|-1\big|$; (e) director field at $t=15$; (f) velocity field at $t=15$. }
	 \label{figs: case1}
\end{center}
\end{figure}

\subsection{Dynamics under anisotropic elasticity}\label{sect: ae}
In this subsection, we investigate the influence of anisotropic elastic coefficients on the dynamics of the full EL system.     The initial velocity field is kept fixed as in \eqref{eq: init-v-1}, while the initial director field is chosen as 
\begin{equation}\label{eq: ncase2}
	\bs n_0=\Big(\sin\big(2\sin(\pi x_1)\big)\cos\big(\pi x_2\big),\; \sin\big(2\sin(\pi x_1)\big)\sin\big(\pi x_2\big),\;\cos\big(2\sin(\pi x_1)\big)\Big).
\end{equation}
This initial profile is spatially nonuniform along the boundary, as shown in Figure~\ref{figs: anicase0}(a).
\begin{figure}[tbp]
\begin{center}
	\subfigure[Initial profile of director field]{ \includegraphics[scale=.34]{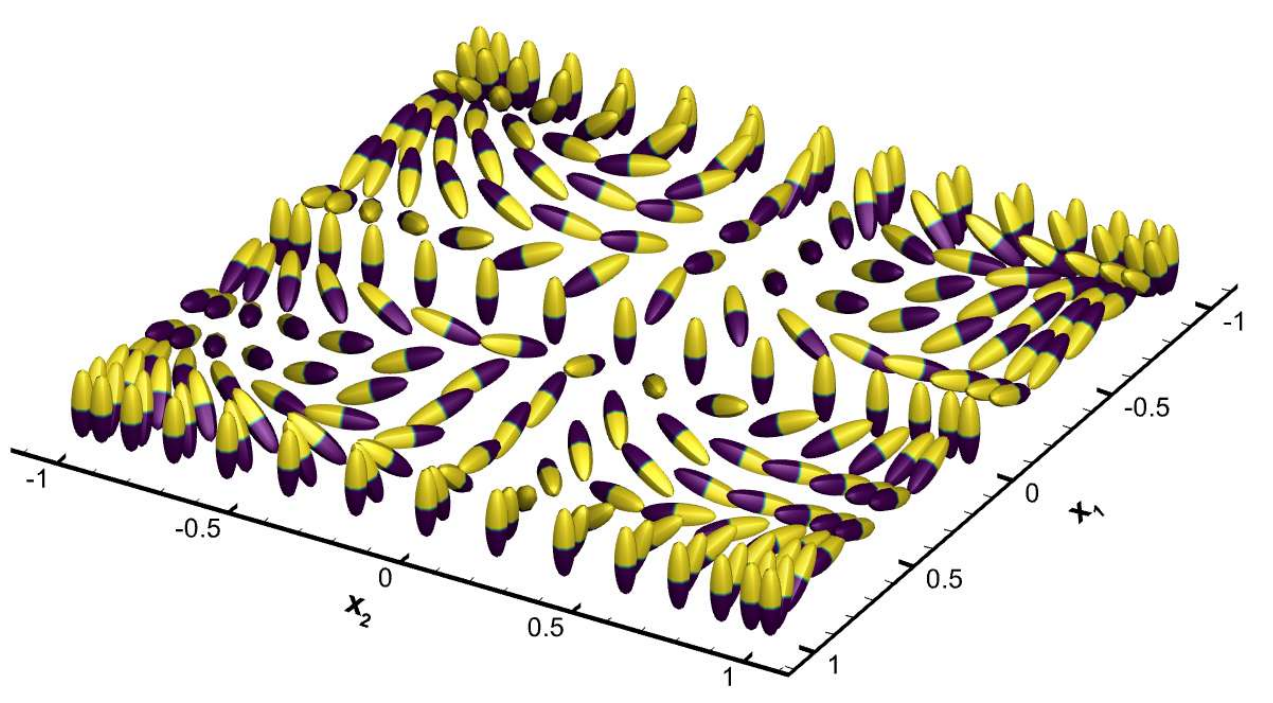}}\;\;\;
	\subfigure[Total energy vs Time]{ \includegraphics[scale=.46]{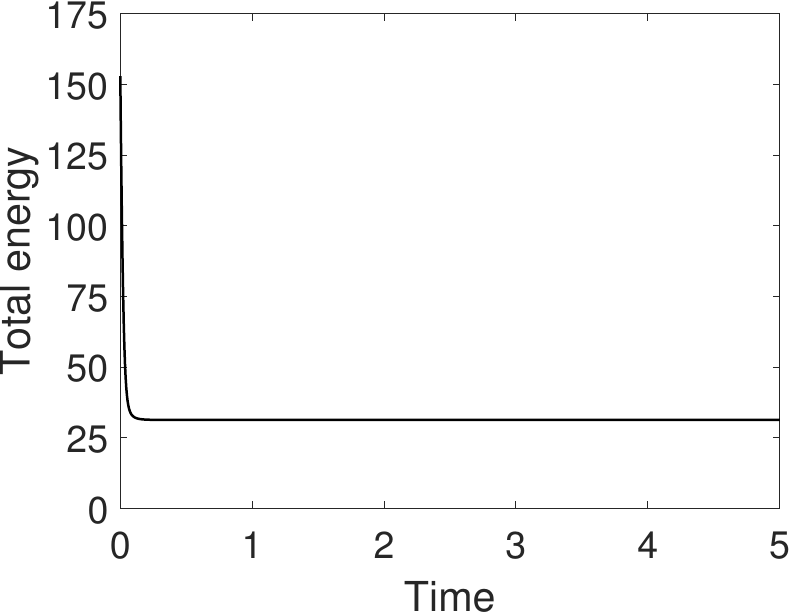}}\\
	\subfigure[$\bs n$ at t=5]{ \includegraphics[scale=.34]{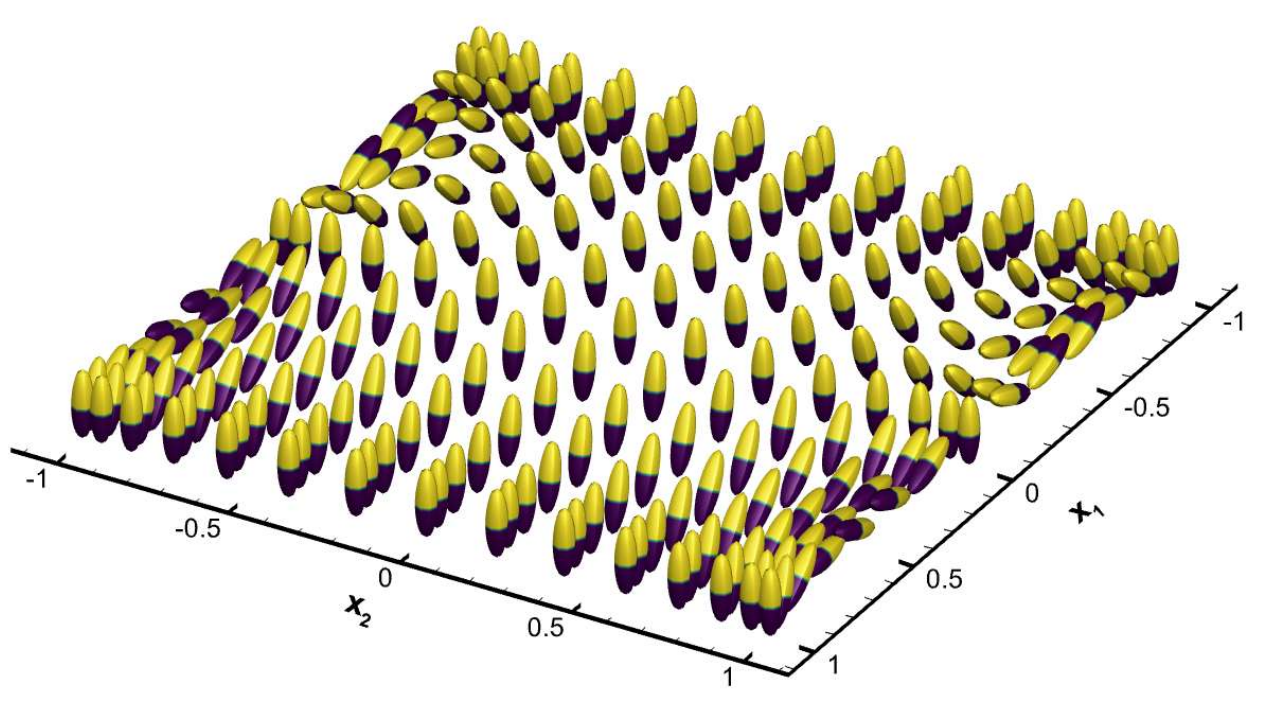}}	
	\subfigure[$\bs v$ at t=5]{ \includegraphics[scale=.34]{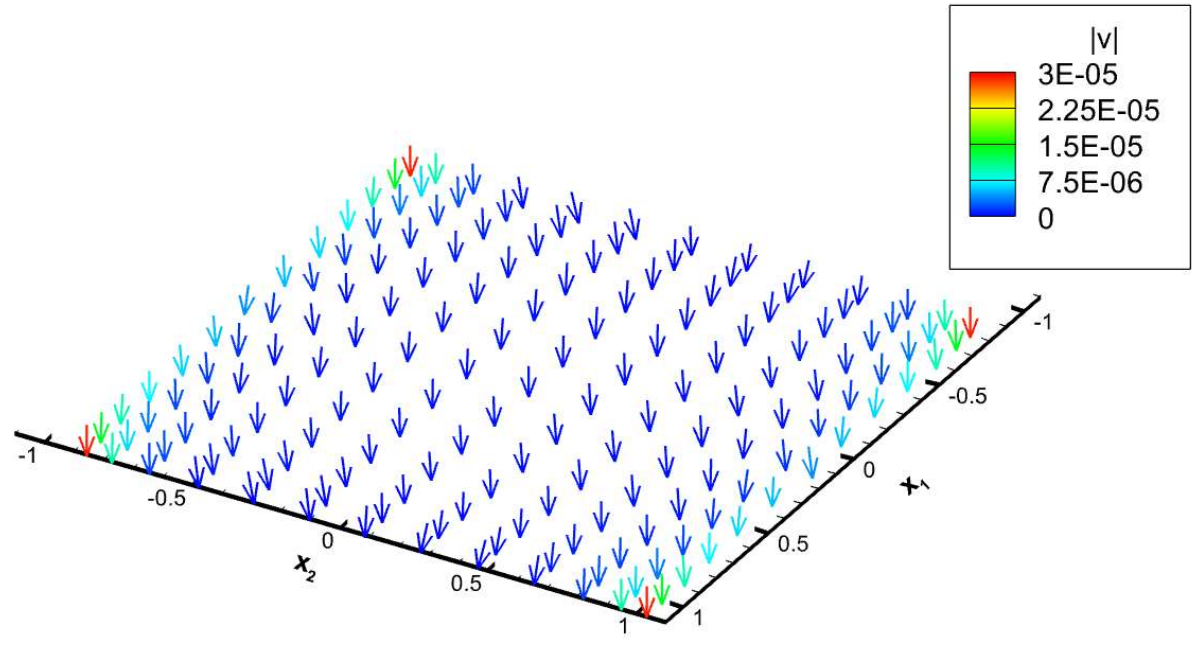}}	
   \caption{\small Reference isotropic case with initial velocity field \eqref{eq: init-v-1}, initial director field \eqref{eq: ncase2} and elastic coefficients $(\kappa_1,\kappa_2,\kappa_3)=(2.5,2.5,2.5)$. (a) Initial director field; (b) history of the discrete total energy $E_N^m$; (c) director field at $t=5$; (d) velocity field at $t=5$.}
	 \label{figs: anicase0}
\end{center}
\end{figure}

We first consider the isotropic reference case $(\kappa_1,\kappa_2,\kappa_3)=(2.5,2.5,2.5)$ and depict the corresponding numerical results in Figure~\ref{figs: anicase0}. The discrete total energy exhibits a rapid transient decay and then approaches a plateau,
indicating that the solution has entered a stationary  regime.
At $t=5$, the director field is already nearly homogeneous in the interior and undergoes a smooth transition to the prescribed nonuniform boundary orientation; see Figure~\ref{figs: anicase0}(c). The velocity field, shown in Figure~\ref{figs: anicase0}(d), is predominantly oriented in the direction normal to the $x_1$--$x_2$ plane,
and its maximum magnitude has decayed to the order of $10^{-5}$.

We next turn to the anisotropic cases $(\kappa_1,\;\kappa_2,\;\kappa_3)=(2.5,\;0.1,\;0.1)$, $(0.1,\;2.5,\;0.1)$ and $(0.1$, $0.1,$ $2.5)$.  The corresponding discrete total energy histories are plotted in Figure~\ref{figs: anicaseenergy}. In contrast to the isotropic reference case, the anisotropic configurations display visibly different relaxation pathways, and in particular may exhibit multistage decay over widely separated time scales.
\begin{figure}[tbp]
\begin{center}
	\subfigure[Total energy vs Time]{\includegraphics[scale=.34]{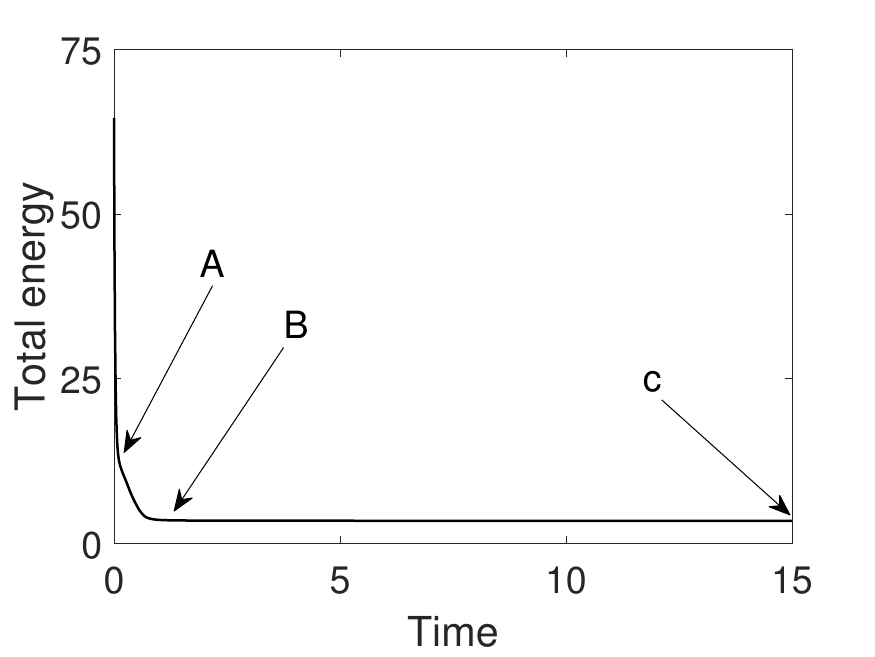}}
	\subfigure[Total energy vs Time]{\includegraphics[scale=.34]{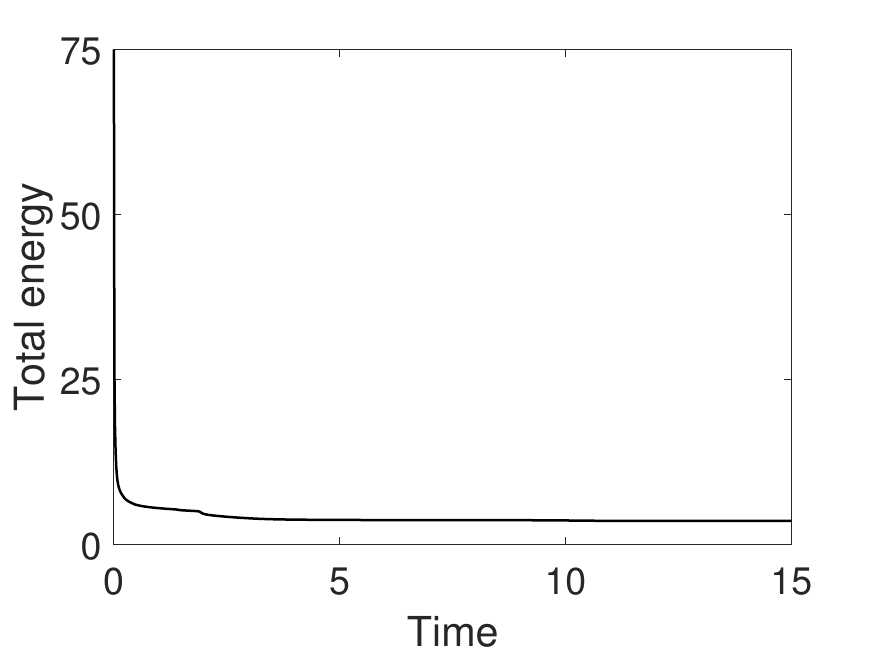}}
	\subfigure[Total energy vs Time]{\includegraphics[scale=.34]{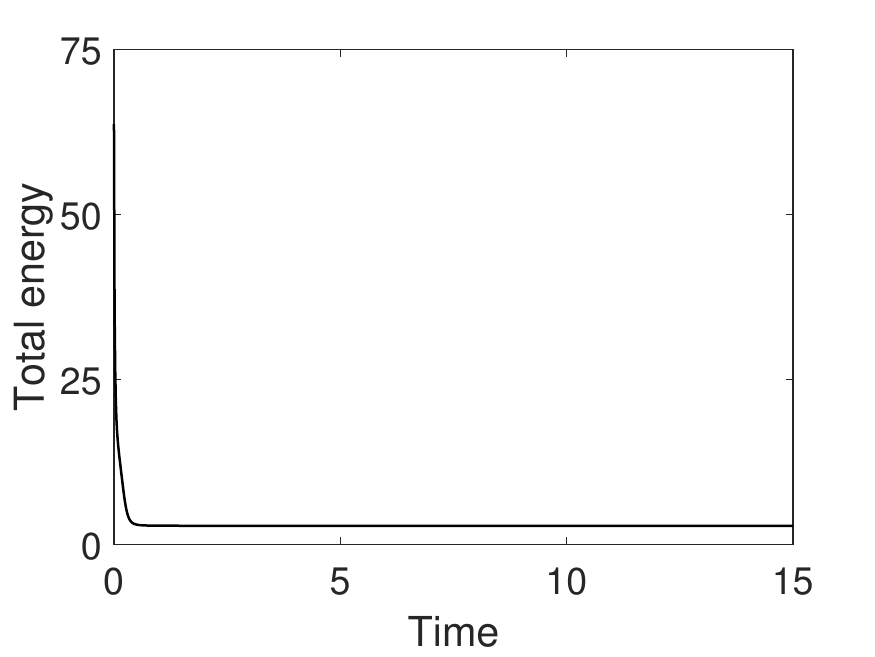}}
   \caption{\small Histories of the discrete total energy $E_N^m$ for the anisotropic cases with initial data \eqref{eq: init-v-1}-\eqref{eq: ncase2}: (a) $(\kappa_1,\kappa_2,\kappa_3)=(2.5,0.1,0.1)$; (b) $(\kappa_1,\kappa_2,\kappa_3)=(0.1,2.5,0.1)$; (c) $(\kappa_1,\kappa_2,\kappa_3)=(0.1,0.1,2.5)$.}
	 \label{figs: anicaseenergy}
\end{center}
\end{figure}

To interpret these dynamics more quantitatively, we further monitor the three discrete elastic components
$$
E_{\mathrm{splay}}^m=\frac{\kappa_1}{2}\|\nabla\cdot \bs n^m\|_N^2,\quad
E_{\mathrm{twist}}^m=\frac{\kappa_2}{2}\|\bs n^m\cdot(\nabla\times \bs n^m)\|_N^2,\quad
E_{\mathrm{bend}}^m=\frac{\kappa_3}{2}\|\bs n^m\times(\nabla\times \bs n^m)\|_N^2.
$$
For the splay-dominated case $(\kappa_1,\kappa_2,\kappa_3)=(2.5,\,0.1,\,0.1),$ the time histories of these three components are shown in Figure~\ref{figs: anicase1energy}, and representative snapshots are displayed in Figure~\ref{figs: anicase1}. A clear multistage relaxation process can be identified. During the initial stage, the total energy decreases rapidly, accompanied by a pronounced decay of the splay component in Figure~\ref{figs: anicase1energy}(a), while the director field becomes smoother in the interior and the most visible changes in the velocity orientation occur near the four corners. At later times, the relaxation becomes much slower and a localized vortex-like director texture develops in the highlighted region of Figure~\ref{figs: anicase1}(c), (e). Although the total energy has nearly flattened by this stage, the component energies continue to evolve slightly, showing that the solution is still undergoing slow internal rearrangement. This behavior is consistent with the fact that a dominant splay coefficient drives the system toward configurations with reduced divergence-dominated distortion.
\begin{figure}[tbp]
\begin{center}
	\subfigure[Splay energy vs Time]{\includegraphics[scale=.34]{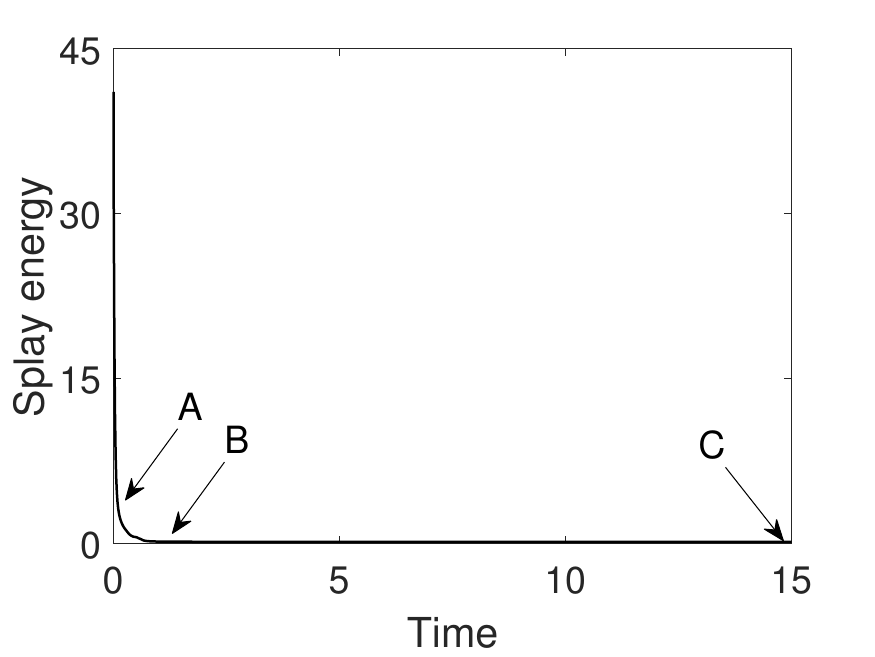}}
	\subfigure[Twist energy vs Time]{\includegraphics[scale=.34]{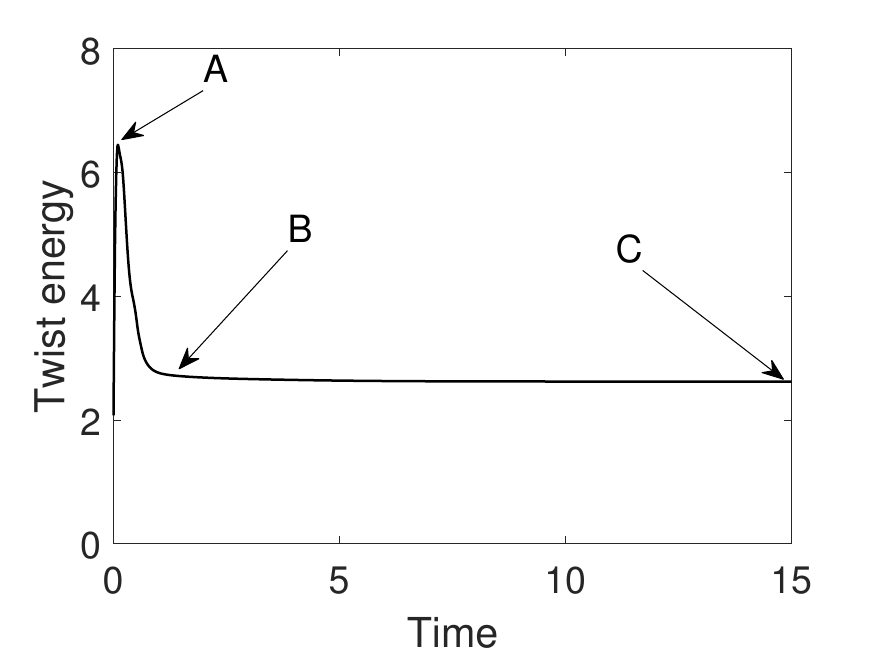}}
	\subfigure[Bend energy vs Time]{\includegraphics[scale=.34]{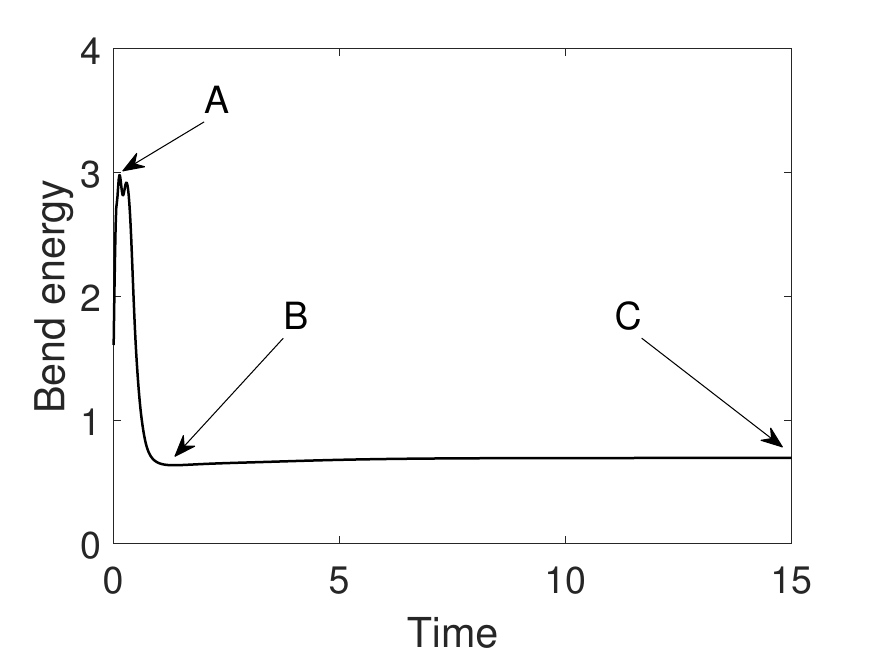}}
   \caption{\small Histories of the three discrete elastic components for the splay-dominated case $(\kappa_1,\kappa_2,\kappa_3)=(2.5,0.1,0.1)$ with initial data \eqref{eq: init-v-1}-\eqref{eq: ncase2}: (a) splay component; (b) twist component; (c) bend component.}
	 \label{figs: anicase1energy}
\end{center}
\end{figure}
\begin{figure}[tbp]
\begin{center}
	\subfigure[$\bs n$ at t=0.1]{ \includegraphics[scale=.34]{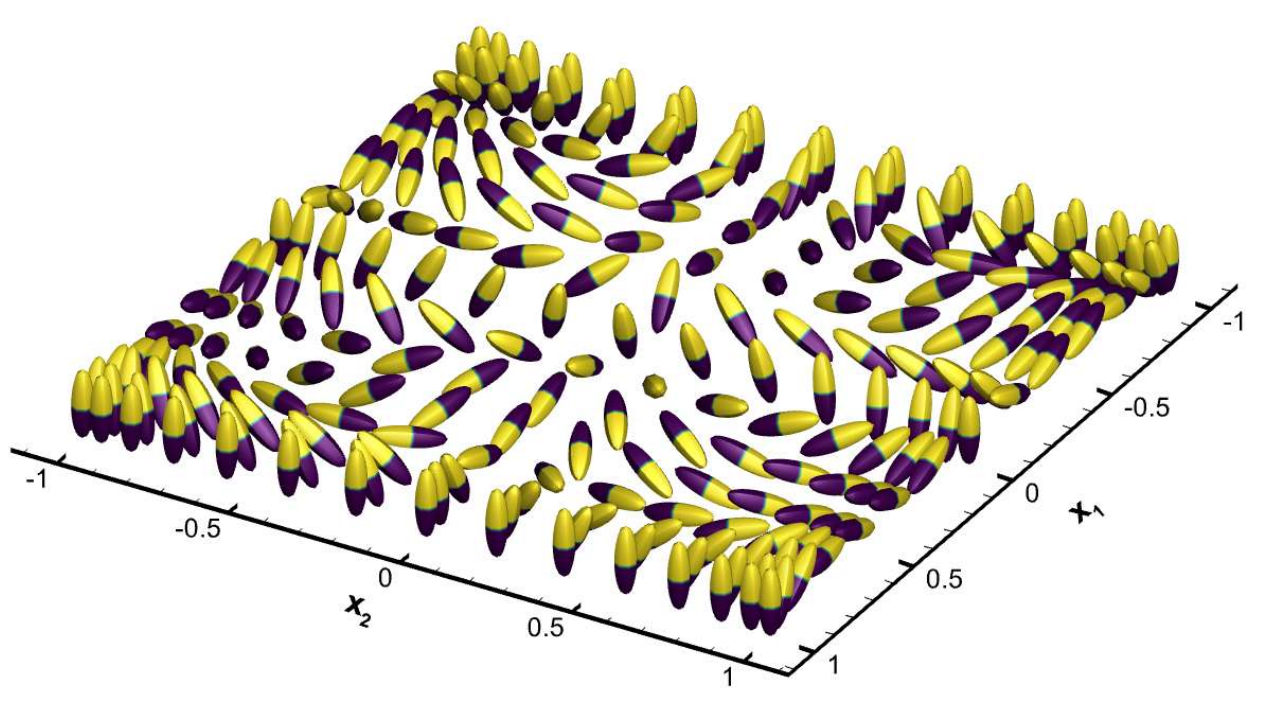}}	
	\subfigure[$\bs v$ at t=0.1]{ \includegraphics[scale=.34]{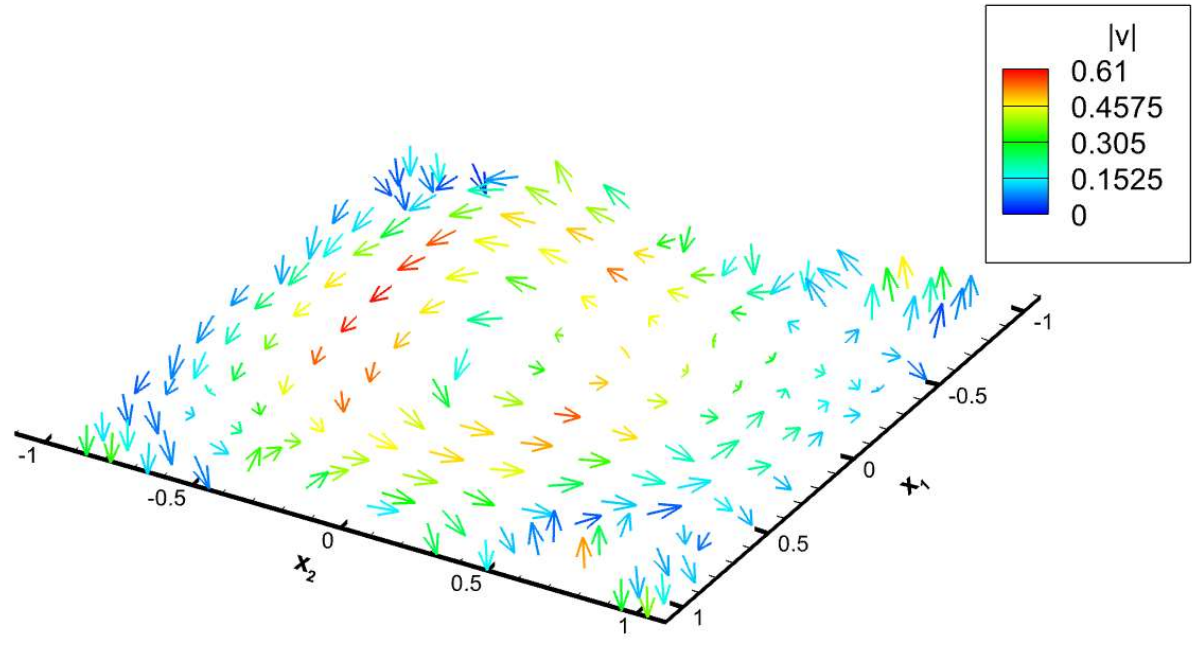}}	\\
	\subfigure[$\bs n$ at t=1.2]{ \includegraphics[scale=.34]{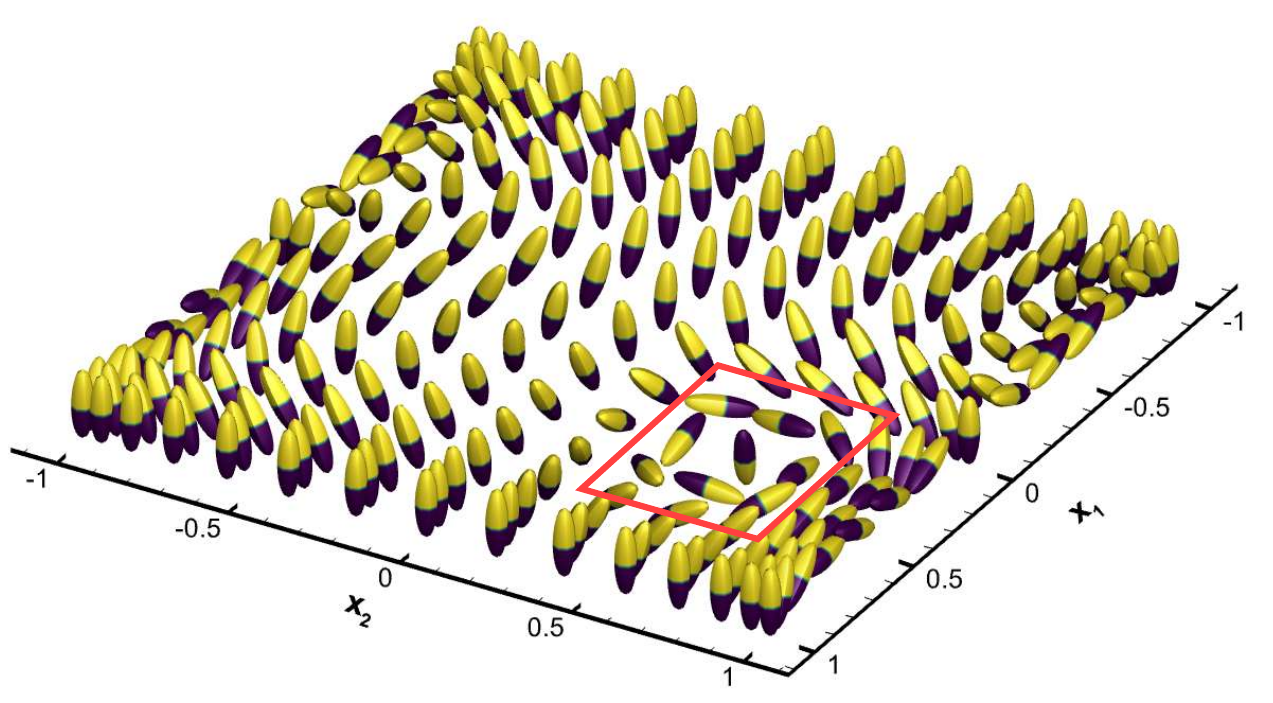}}	
	\subfigure[$\bs v$ at t=1.2]{ \includegraphics[scale=.34]{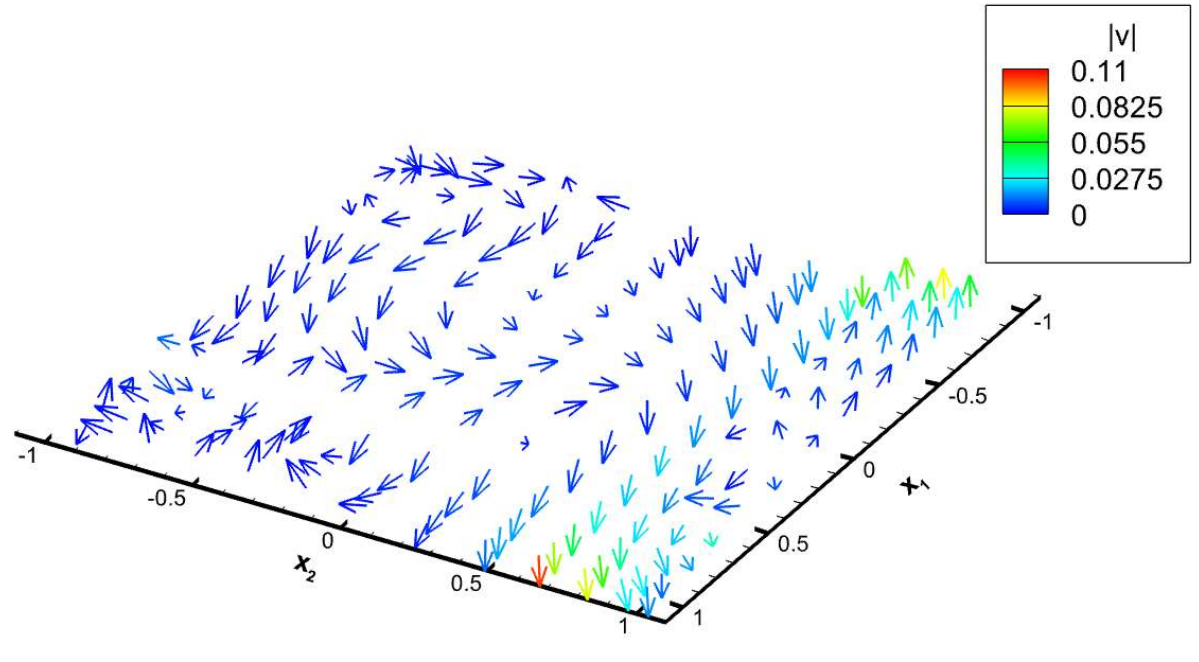}}	\\
	\subfigure[$\bs n$ at t=15]{ \includegraphics[scale=.34]{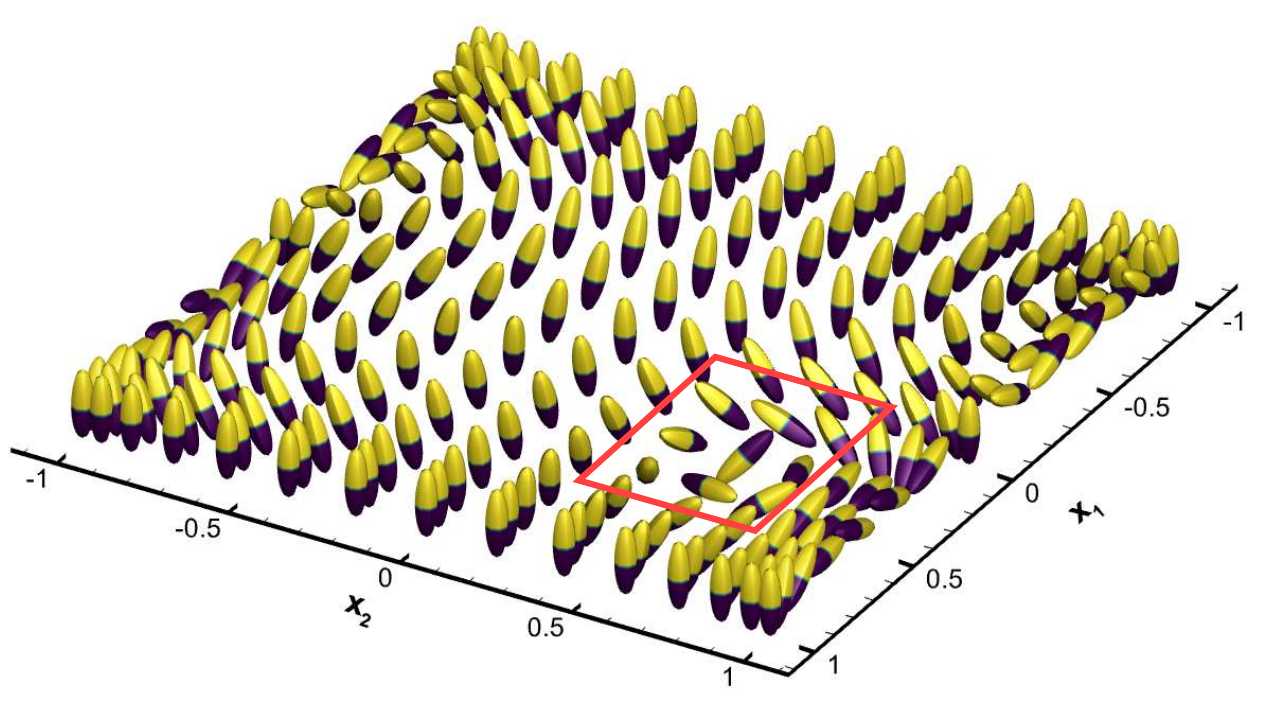}}	
	\subfigure[$\bs v$ at t=15]{ \includegraphics[scale=.34]{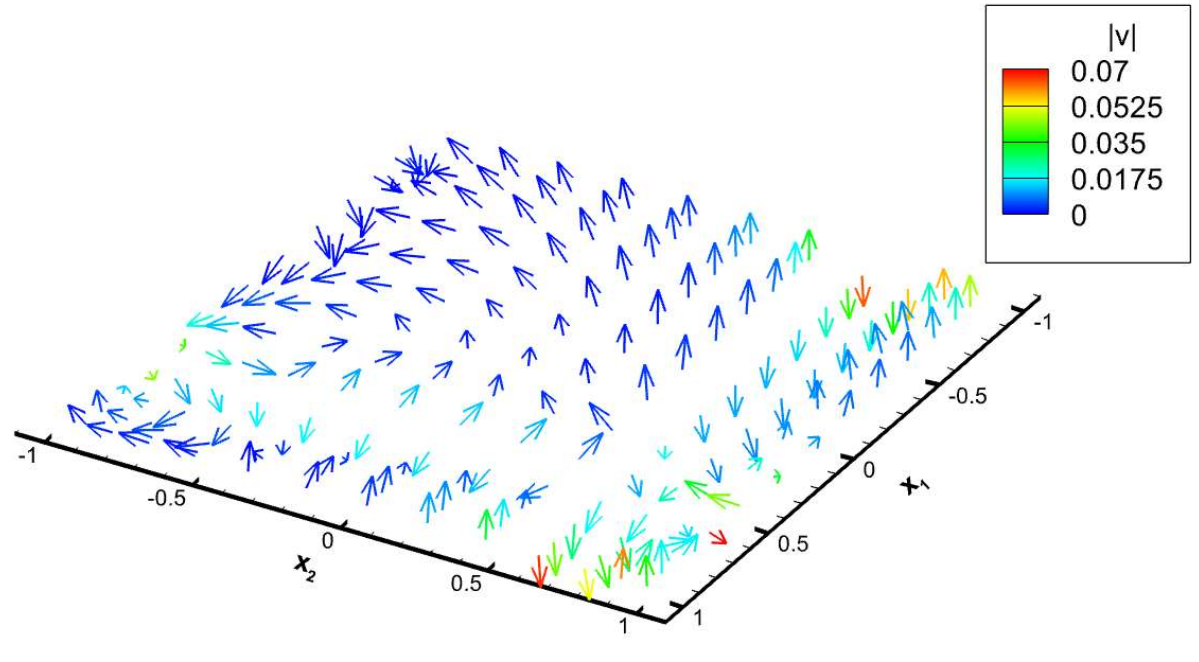}}	
   \caption{\small Representative time snapshots for the splay-dominated case $(\kappa_1,\kappa_2,\kappa_3)=(2.5,0.1,0.1)$ with initial data \eqref{eq: init-v-1}-\eqref{eq: ncase2}. Director field at (a) $t=0.1$, (c) $t=1.2$, and (e) $t=15$; velocity field at (b) $t=0.1$, (d) $t=1.2$, and (f) $t=15$.}
	\label{figs: anicase1}
\end{center}
\end{figure}

To further illustrate the effect of anisotropy, we show in Figure~\ref{figs: anicase_stable} the final states for the twist-dominated and bend-dominated cases. When $(\kappa_1,\kappa_2,\kappa_3)=(0.1,\,2.5,\,0.1),$
the director field becomes nearly uniform along the $x_2$-direction in the highlighted region of Figure~\ref{figs: anicase_stable}(a), which is consistent with the suppression of twist-dominated distortion. The corresponding velocity field in Figure~\ref{figs: anicase_stable}(b) appears more disordered in orientation, but its magnitude is already very small in the central part of the domain. While for the bend-dominated case
$
(\kappa_1,\kappa_2,\kappa_3)=(0.1,\,0.1,\,2.5),
$
the terminal pattern, shown in Figure~\ref{figs: anicase_stable}(c)-(d), is closer to the isotropic reference case: the director is again nearly homogeneous in the interior, but the aligned region is visibly larger, reflecting the tendency of a large bend coefficient to penalize curvature-dominated distortions.
\begin{figure}[tbp]
\centering
    \subfigure[$\bs n$ at $t=15$ with $(\kappa_1,\;\kappa_2,\;\kappa_3)=(0.1,\;2.5,\;0.1)$]{\includegraphics[scale=0.34]{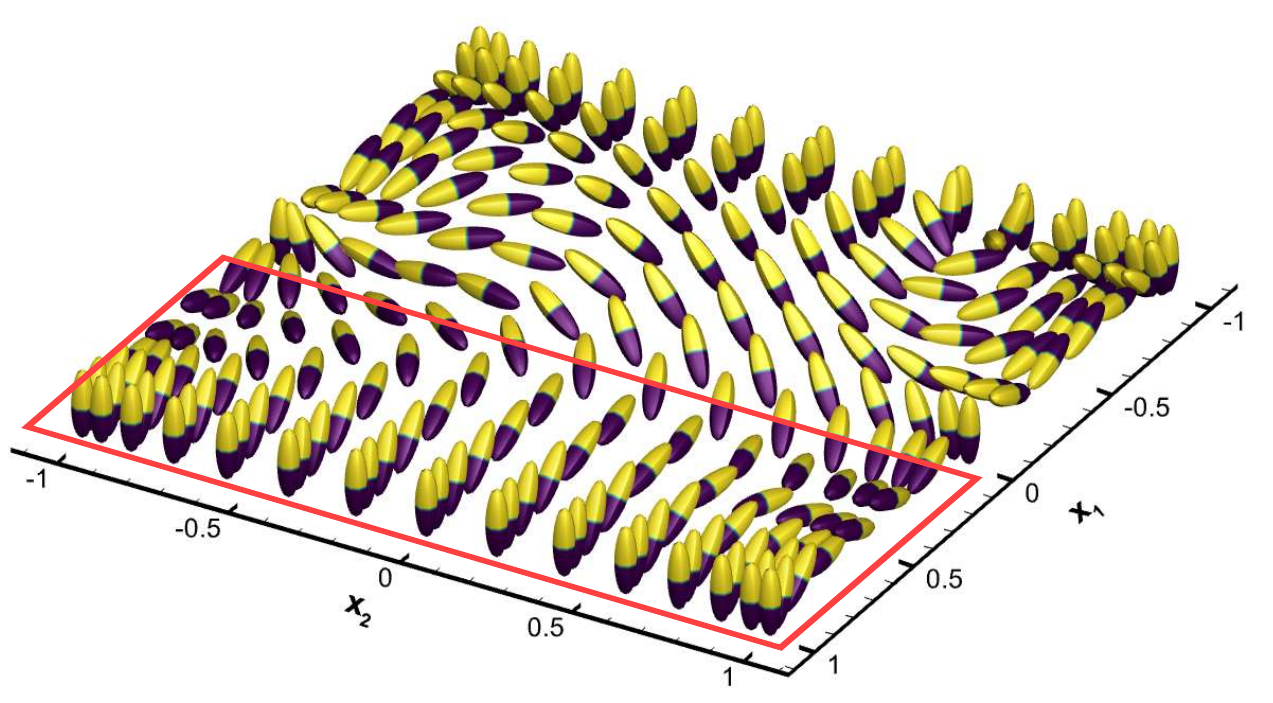} }
    \subfigure[$\bs v$ at $t=15$ with $(\kappa_1,\;\kappa_2,\;\kappa_3)=(0.1,\;2.5,\;0.1)$]{\includegraphics[scale=0.34]{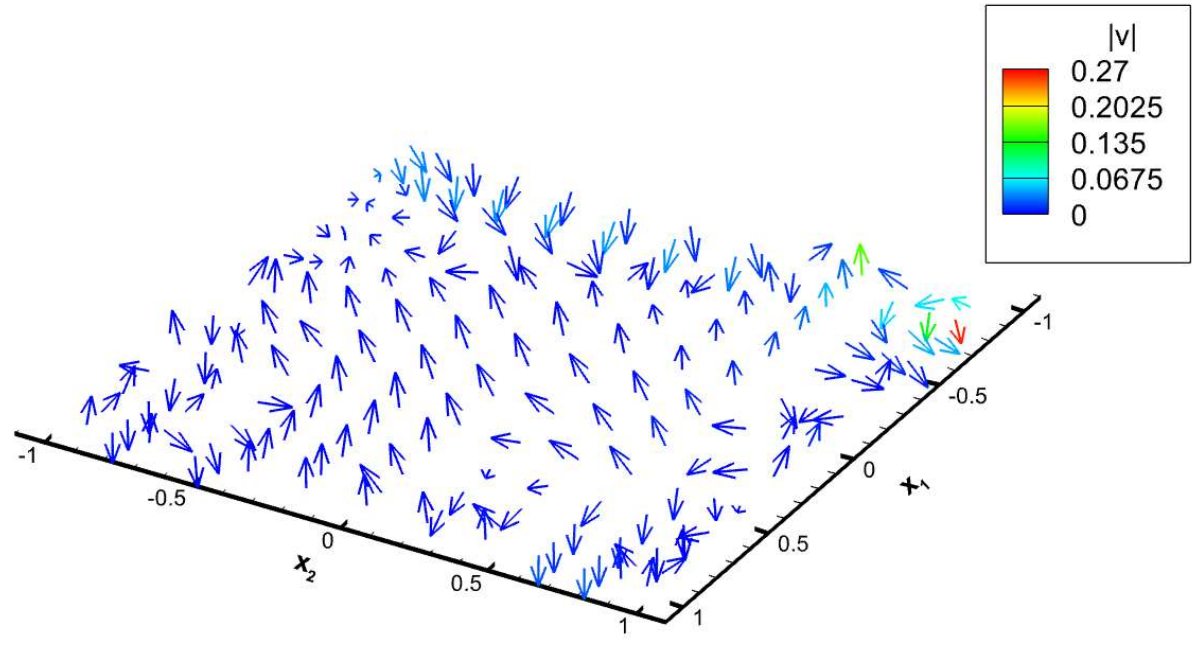} } \\
    \subfigure[$\bs n$ at $t=15$ with $(\kappa_1,\;\kappa_2,\;\kappa_3)=(0.1,\;0.1,\;2.5)$]{\includegraphics[scale=0.34]{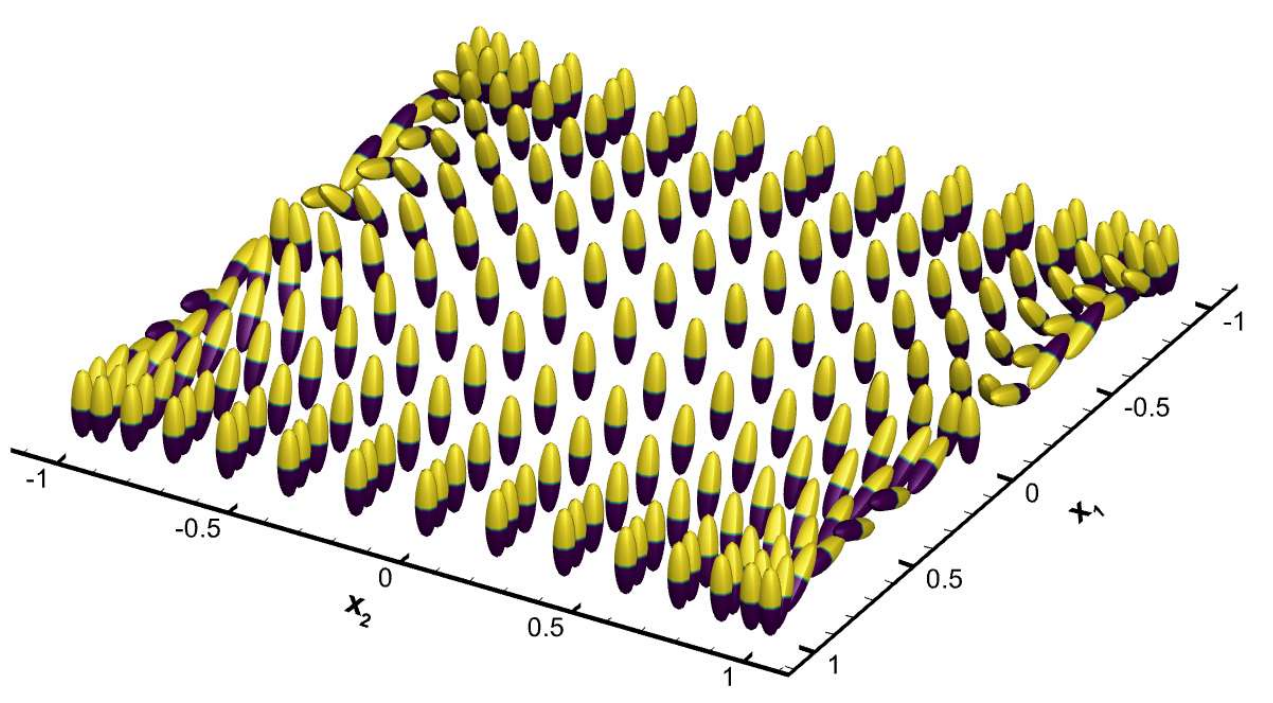} }
    \subfigure[$\bs v$ at $t=15$ with $(\kappa_1,\;\kappa_2,\;\kappa_3)=(0.1,\;0.1,\;2.5)$]{\includegraphics[scale=0.34]{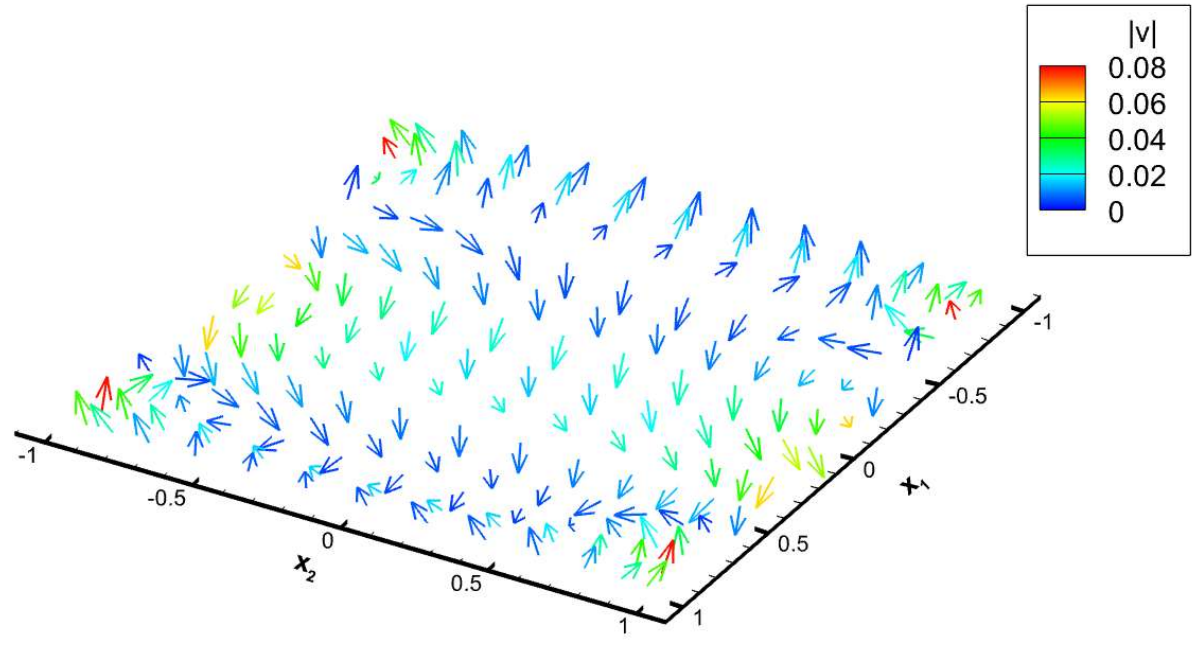} }

    \caption{The representative time snapshots of director field $\bs n$ and velocity field $\bs v$ for two anisotropic cases with initial data \eqref{eq: init-v-1}-\eqref{eq: ncase2}. (a) Director field and (b) velocity field at $t=15$ for $(\kappa_1,\kappa_2,\kappa_3)=(0.1,2.5,0.1)$; (c) director field and (d) velocity field at $t=15$ for $(\kappa_1,\kappa_2,\kappa_3)=(0.1,0.1,2.5)$.
}
    \label{figs: anicase_stable}
\end{figure}

Overall, these experiments show that changing the elastic coefficients does not merely alter the decay rate. Rather, it changes both the evolutionary pathways and steady distributions. This demonstrates that anisotropic Oseen--Frank elasticity plays a decisive role in the dynamics of the full EL model. At the same time, these results highlight the importance of numerical methods that can faithfully resolve general anisotropic elastic energies, rather than relying on one-constant approximation or reduced constitutive approximations.

\subsection{Dynamics under  shear flow}\label{sect: sf}
In this subsection, we investigate the director field reorientation driven by shear flow in the full EL model and the resulting flow-director coupling effects. We keep the elastic coefficients $(\kappa_1,\; \kappa_2,\;\kappa_3)=(0.1,\;0.5,\;2.5)$ and initial director field distribution \eqref{eq: ncase2}. We consider two shear-flow initial profiles, i.e. (i) weak shear flow $\bs v_0=\big(\sin(\pi x_2),\;0,\;0\big)^{\intercal}$ and (ii) strong shear flow $\bs v_0=\big(10\sin(\pi x_2),\;0,\;0\big)^{\intercal}$. We summarize these numerical results in Figure~\ref{figs: shearcase}. 
\begin{figure}[tbp]
\begin{center}
	\subfigure[Total energy vs Time with $\bs v_0=(\sin(\pi x_2),0,0)$]{ \includegraphics[scale=.46]{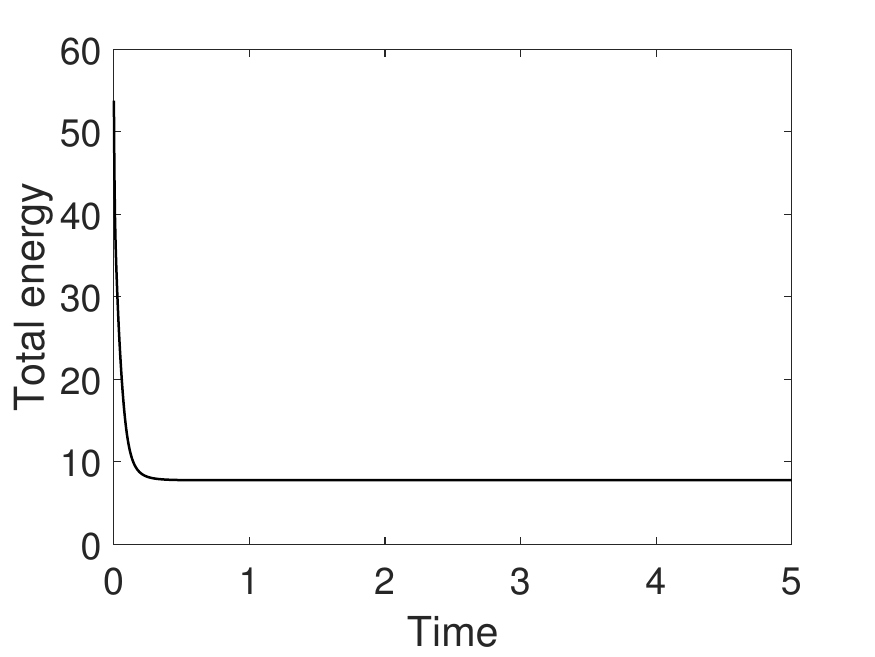}}
	\subfigure[Total energy vs Time with $\bs v_0=(10\sin(\pi x_2),0,0)$]{ \includegraphics[scale=.46]{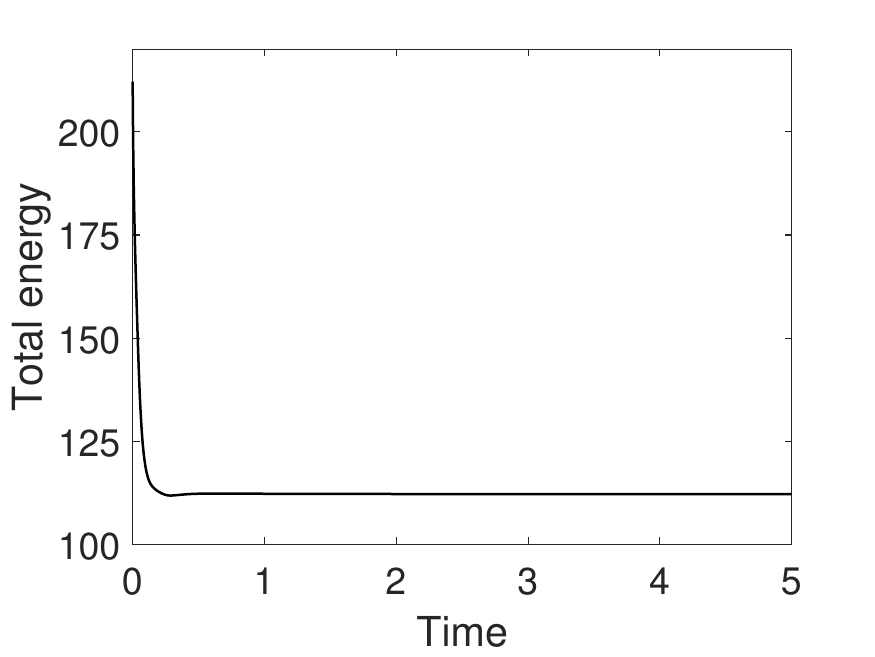}}\\
	\subfigure[$\bs n$ at t=5 with $\bs v_0=(\sin(\pi x_2),0,0)$]{ \includegraphics[scale=.34]{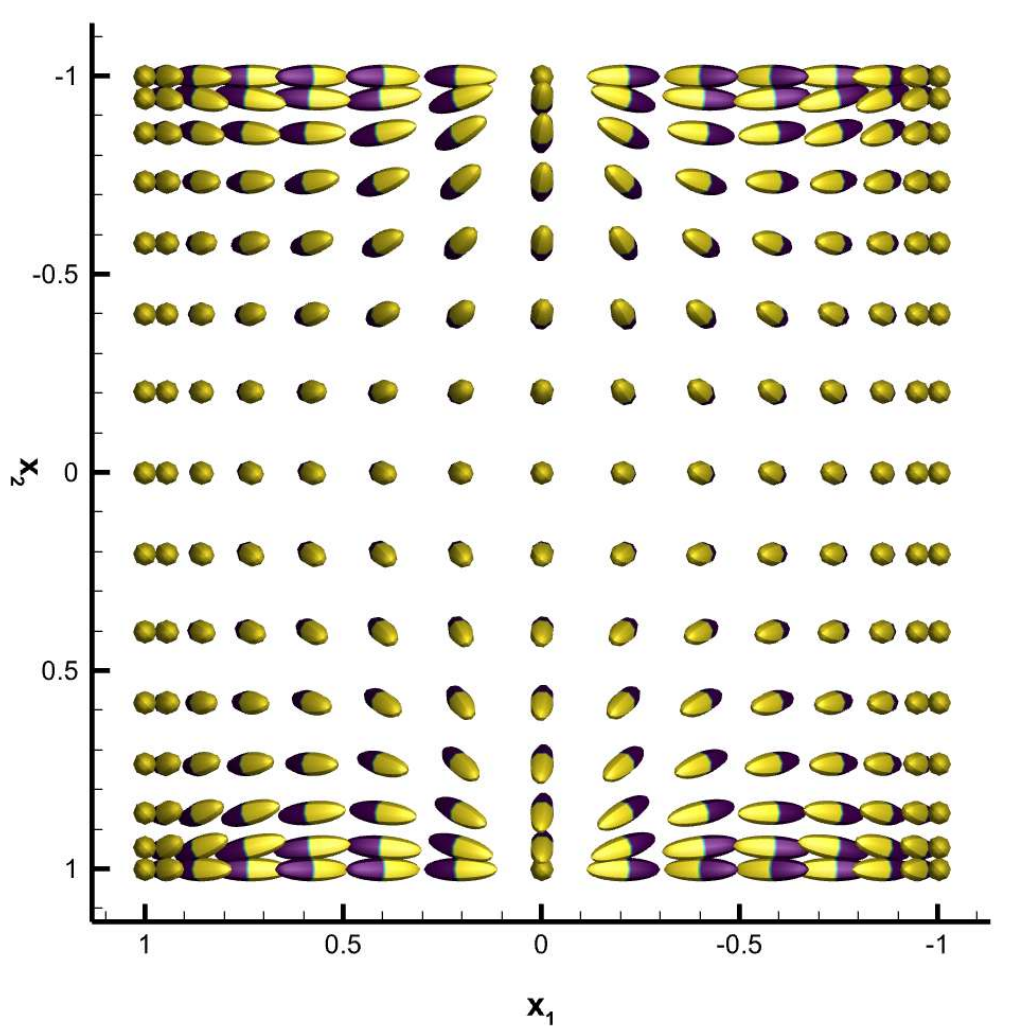}}
	\subfigure[$\bs v$ at t=5 with $\bs v_0=(\sin(\pi x_2),0,0)$]{ \includegraphics[scale=.38]{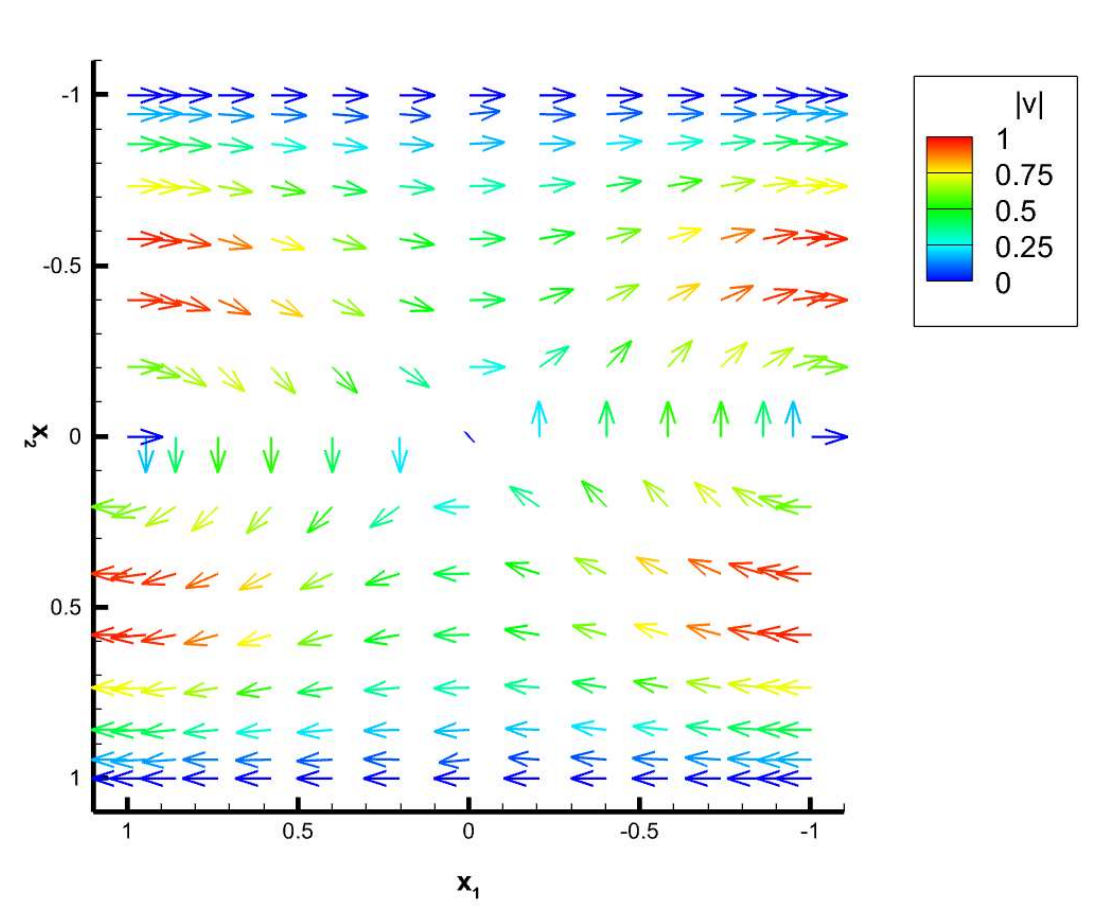}}\\
	\subfigure[$\bs n$ at t=5 with $\bs v_0=(10\sin(\pi x_2),0,0)$]{ \includegraphics[scale=.32]{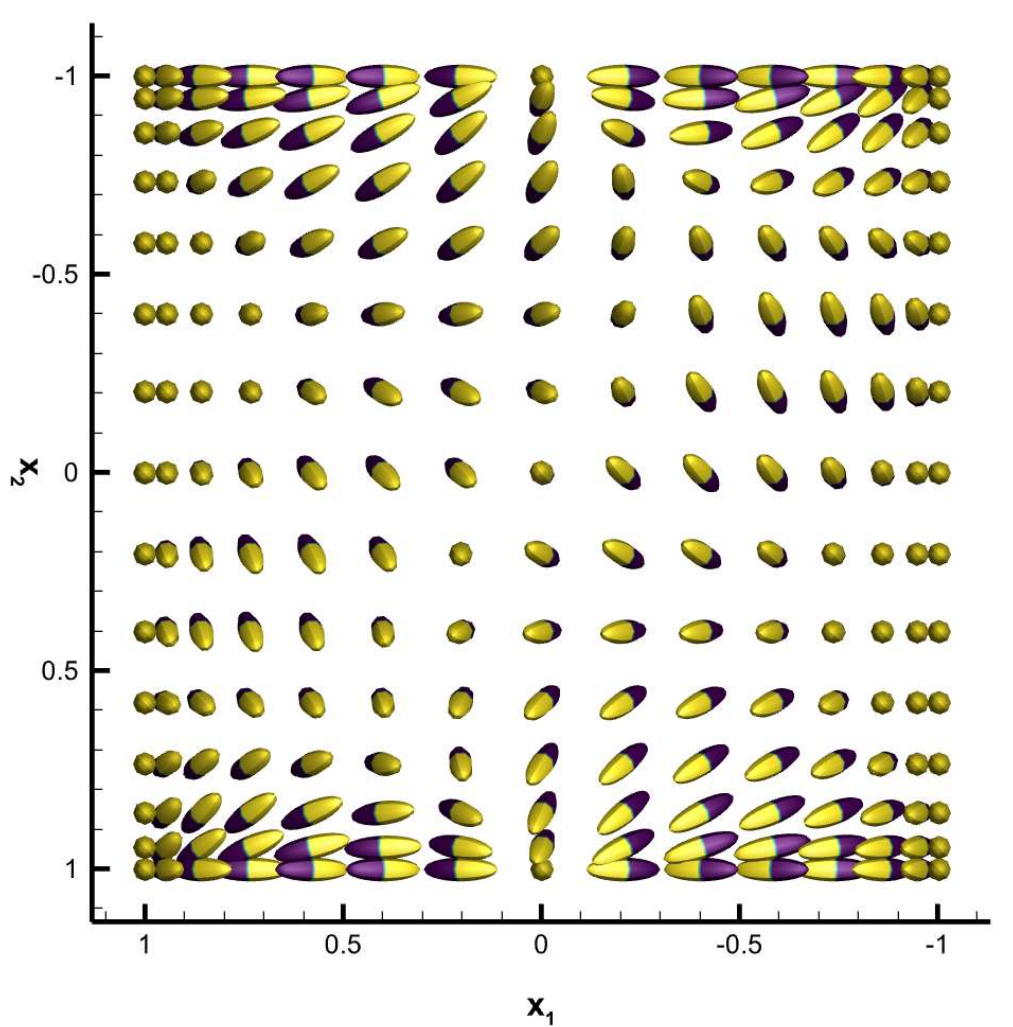}}\;\;
	\subfigure[$\bs v$ at t=5 with $\bs v_0=(10\sin(\pi x_2),0,0)$]{ \includegraphics[scale=.38]{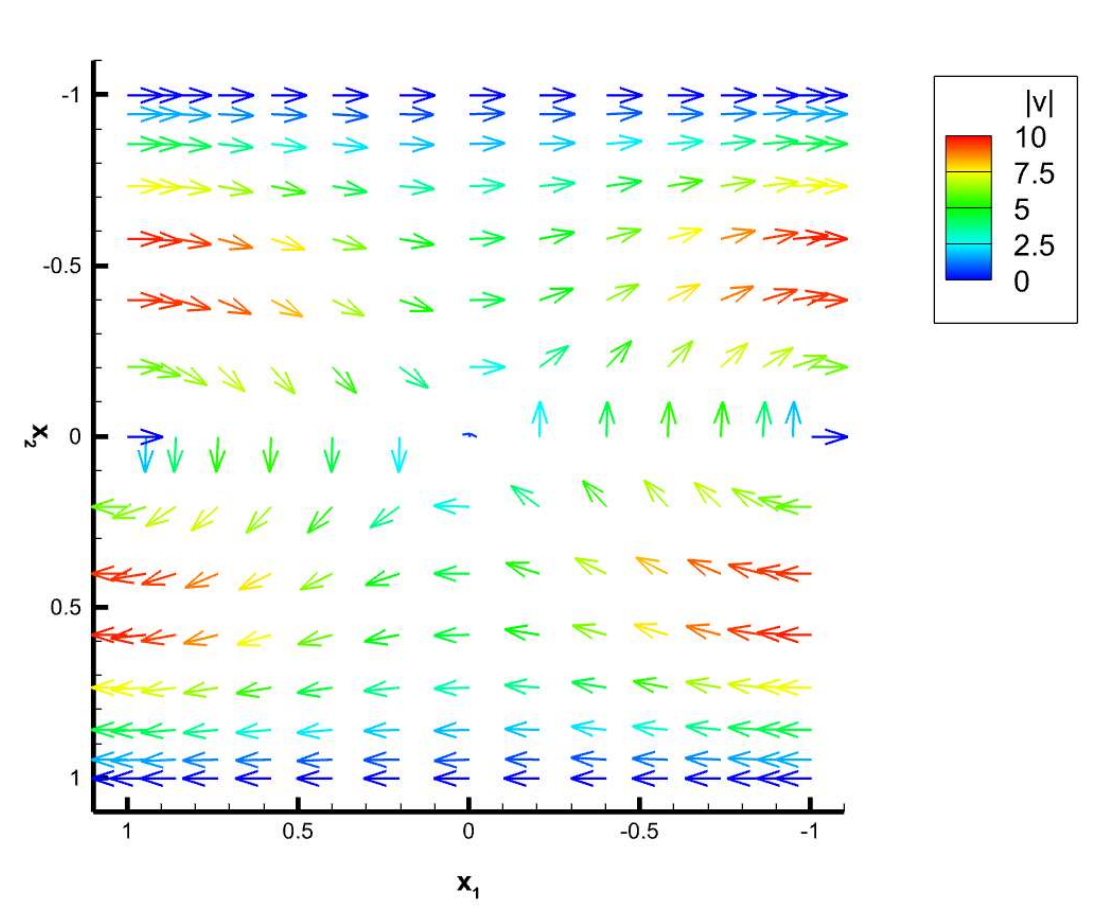}}\\
	\caption{\small Evolution of the director field $\bs n$ and velocity field $\bs v$ under shear flows with initial director field \eqref{eq: ncase2} and elastic coefficients $(\kappa_1,\kappa_2,\kappa_3)=(0.1,0.5,2.5)$. (a) History of the discrete energy $E_N^m$ for $\bs v_0=(\sin(\pi x_2),0,0)$; (b) history of $E_N^m$ for $\bs v_0=(10\sin(\pi x_2),0,0)$; (c) director field and (d) velocity field at $t=5$ for $\bs v_0=(\sin(\pi x_2),0,0)$; (e) director field and (f) velocity field at $t=5$ for $\bs v_0=(10\sin(\pi x_2),0,0)$.}
	 \label{figs: shearcase}
\end{center}
\end{figure}

It is worthwhile to point out that since inhomogeneous Dirichlet boundary conditions are imposed on the velocity field in these experiments, the discrete energy law from Theorem~\ref{thm: fully-discrete-structure} is no longer expected to hold.  Accordingly, the quantity $E_N^m$ is used here only as a diagnostic quantity of the dynamics.
As seen in Figure~\ref{figs: shearcase}(a)-(b), the discrete total energy in both cases undergoes a short transient and then becomes nearly constant by $t=5$.

A comparison of the final states reveals a clear dependence on the shear strength.
For the weaker shear flow, the director field in Figure~\ref{figs: shearcase}(c) remains relatively homogeneous in the interior, and the flow mainly induces a mild reorientation near the top and bottom portions of the domain. For the stronger shear flow, the velocity field in Figure~\ref{figs: shearcase}(f) has essentially the same directional pattern as in Figure~\ref{figs: shearcase}(d), but the director field in Figure~\ref{figs: shearcase}(e) is much more strongly tilted and exhibits a pronounced tendency to align with the imposed horizontal flow over a substantial portion of the domain. Hence, even when the qualitative geometry of the shear flow is unchanged, increasing its amplitude can lead to a markedly different director response.

These experiments show that the Leslie-stress coupling can strongly influence shear-induced reorientation in the full EL model. In particular, the competition between elastic relaxation and externally imposed shear may lead to qualitatively different alignment regimes, which would be difficult to capture reliably within reduced constitutive settings.  From the numerical viewpoint, these results also highlight the importance of structure-preserving methods formulated directly for the full Ericksen--Leslie system. In conclusion, the proposed Rdg method provides an effective structure-preserving numerical scheme for accurately simulating the dynamical effects induced by anisotropic elasticity while preserving the intrinsic unit-length constraint and energy dissipative structures of the full Ericksen--Leslie system.

\section{Concluding remarks}

We have developed a structure-preserving rotational discrete gradient scheme for the full Ericksen--Leslie model with general anisotropic Oseen--Frank elasticity and full Leslie stress. The central ingredient is an equivalent rotational reformulation of the EL model, in which the unit-length geometry of the director evolution and the energy-exchange structure of the coupled system become explicit and can therefore be embedded directly into the discretization. Based on this reformulation, we constructed second-order rotational discrete-gradient schemes that preserve the director length and satisfy a discrete energy law at the time-discrete level, together with a fully discrete exact divergence-free spectral scheme in which a quadrature-compatible LGL nodal discretization is used to preserve the unit-length constraint at the quadrature points, while a discrete energy law is obtained under the homogeneous velocity boundary condition considered in Theorem~\ref{thm: fully-discrete-structure}. The numerical experiments confirm the expected second-order temporal accuracy and spectral spatial convergence, and further illustrate several qualitative effects induced by anisotropic elastic coefficients and shear flow. These results show that the proposed method provides a reliable simulation tool of anisotropic liquid--crystal flows for the full Ericksen--Leslie model beyond the one-constant and reduced-stress settings.

\appendix
\renewcommand{\theequation}{A.\arabic{equation}}
\section{Rdg scheme based on projection method}\label{app: pro}

For spatial discretizations in which the divergence-free constraint is not enforced exactly, one can adopt the following projection-type Rdg variant of the time discretization in Section~\ref{sect: 3}. 
\vspace{5pt}
\paragraph{\textbf{Projection-type Rdg time-discrete scheme}.} \label{sch:projection}
For $m\geq 1$, given $(\bs v^{m-1},\bs v^m, P^m,\bs n^m)$ with $\nabla\cdot \bs v^{m-1}=\nabla\cdot \bs v^{m}=0$ and $|\bs n^m|=1$, and fixed the time step $\tau_m=\tau$, one computes successively the intermediate velocity $\widetilde{\bs v}^{m+1}$ and the director $\bs n^{m+1}$ from
\begin{subequations}\label{eq: projectionsch}
\begin{align}
	\begin{split}
		\frac{\widetilde{\bs v}^{m+1}-\bs v^m}{\tau}
&=
-\bar{\bs v}^{m+\frac12}\cdot \nabla \widetilde{\bs v}^{m+\frac12}
-\nabla P^m
+\frac{\gamma}{\rm Re}\Delta \widetilde{\bs v}^{m+\frac12}
+\frac{1-\gamma}{\rm Re}\nabla\cdot \widehat{\sig}^{\rm L}\big|^{\sim,m+\frac12}
\\
&\quad
+\frac{1-\gamma}{\rm Re}\nabla \bs n^{m+\frac12}\cdot
\Big(
\big(
\bs n^{m+\frac12}\times D_{\mathcal F}^O(\bs n)\big|^{m+\frac12}
\big)\times \bs n^{m+\frac12}
\Big),
	\end{split} \label{eq: projschemp1}\\
	\begin{split}
		\frac{\bs n^{m+1}-\bs n^m}{\tau}
&=
-\frac{1}{\gamma_1}
\Big(
\bs n^{m+\frac12}\times
\Big[
D_{\mathcal F}^O(\bs n)\big|^{m+\frac12}
+\gamma_1\Big(
\widetilde{\bs v}^{m+\frac12}\cdot \nabla \bs n^{m+\frac12}
+\widetilde{\om}^{m+\frac12}\cdot \bs n^{m+\frac12}
\Big)
\\
&\qquad\qquad\qquad\qquad
+\gamma_2\widetilde{\tt}^{m+\frac12}\cdot \bs n^{m+\frac12}
\Big]
\Big)\times \bs n^{m+\frac12}, 
	\end{split}\label{eq: projschemp2}
\end{align}
\end{subequations}
and then updates $(\bs v^{m+1},P^{m+1})$ by the pressure-correction step
\begin{subequations}\label{eq:sch2pre}
\begin{align}
\frac{\bs v^{m+1}-\widetilde{\bs v}^{m+1}}{\tau}
+\frac{\nabla(P^{m+1}-P^m)}{2}&=\bs 0, \label{eq:sch2pre1}\\
\nabla\cdot \bs v^{m+1}&=0. \label{eq:sch2pre2}
\end{align}
\end{subequations}
In the above, $\widetilde{\bs v}^{m+\frac{1}{2}}=\big(\widetilde{\bs v}^{m+1}+\bs v^m\big)/2$, and $\bar{\bs v}^{m+1/2}$ represents the 2nd-order explicit approximation of $\bs v^{m+1/2}$ given by $\bar{\bs v}^{m+1/2}= \big(3\bs v^m-\bs v^{m-1}\big)/2$, where $\bs v^1$ can be obtained by a first-order projection scheme. Moreover, we define $\widehat \sig^{\mathrm{L}}\big|^{\sim,m+\frac{1}{2}}$ by
\eqref{eq: main-discrete-sigmaL}, with $\tt^{m+\frac{1}{2}}$ replaced by
$\widetilde{\tt}^{m+\frac{1}{2}}$.

\begin{thm}\label{thm: property-pre}
Assume that $|\bs n^0|=1$ and $\nabla\cdot \bs v^0=\nabla\cdot \bs v^1=0$ and the material parameters satisfy \eqref{eq: paracons} with $\gamma\in(0,1)$,  impose the boundary conditions either $
\bs v^m|_{\partial\Omega}=\bs 0,
\;
\bs n^m|_{\partial\Omega}=\bs g(\bs x),
\; m\ge 0,
$ or periodic boundary conditions,  and let $(\bs n^{m+1},\widetilde{\bs v}^{m+1},\bs v^{m+1},P^{m+1})$ be computed by the proposed projection-based Rdg scheme, then  for every $m\ge 1$ the following properties hold.

\smallskip
\noindent
(i) \textbf{Unit-length preservation:}  
\begin{equation}\label{eq: lengnprojection}
|\bs n^{m+1}|=1.
\end{equation}

\smallskip
\noindent
(ii) \textbf{Modified discrete energy dissipation:} 
\begin{equation}\label{eq: dissproj}
\begin{aligned}
& \dfrac{\rm Re}{2(1-\gamma)}
\|\bs v^{m+1}\|^2 +\dfrac{\tau^2 {\rm Re}}{8(1-\gamma)}
\|\nabla P^{m+1}\|^2
+\mathcal F[\bs n^{m+1}]\\
&\leq  \dfrac{\rm Re}{2(1-\gamma)}
\|\bs v^{m}\|^2 +\dfrac{\tau^2 {\rm Re}}{8(1-\gamma)}
\|\nabla P^{m}\|^2
+\mathcal F[\bs n^{m}].
\end{aligned}
\end{equation}

\end{thm}

\begin{proof}
We start with the proof of the unit-length preserving property of the proposed scheme. 
Taking the dot product of \eqref{eq: projschemp2} with $\bs n^{m+\frac12}$, using the identity \eqref{id: crossid2} and difference of two squares identity, it follows that $|\bs n^{m+1}|^2=|\bs n^m|^2.$ By induction from $|\bs n^0|=1$, we conclude that
$
|\bs n^m|=1,\; m\ge 0.
$
This proves \eqref{eq: lengnprojection}.

Next, we proceed to the proof of the discrete energy dissipation of the proposed scheme. 
From the facts that $\nabla\cdot \bs v^m=\nabla\cdot \bs v^{m-1}=0$ and $\bar{\bs v}^{m+1/2}= \big(3\bs v^m-\bs v^{m-1}\big)/2$ and the boundary condition $\bs v^m|_{\partial\Omega}=\bs 0$ (resp. periodic boundary condition),  it follows the identities
\begin{equation}\label{eq: divid}
\big(\bar{\bs v}^{m+\frac12}\cdot\nabla \widetilde{\bs v}^{m+\frac12},
\widetilde{\bs v}^{m+\frac12}\big)=0,
\quad
(\nabla P^m,\bs v^m)=0.
\end{equation}
Let us take the $L^2$ inner product of \eqref{eq: projschemp1} with $\big({\tau {\rm Re}}/{(1-\gamma)} \big)\,\widetilde{\bs v}^{m+\frac12}$ and using the identities \eqref{eq: divid},
we obtain
\begin{equation}\label{eq: app-proof-v}
\begin{aligned}
&\frac{\rm Re}{2(1-\gamma)}
\Big(\|\widetilde{\bs v}^{m+1}\|^2-\|\bs v^m\|^2\Big)
=
-\frac{{\rm Re}\tau}{2(1-\gamma)}
(\nabla P^m,\widetilde{\bs v}^{m+1})
-\frac{\gamma\tau}{1-\gamma}\|\nabla \widetilde{\bs v}^{m+\frac12}\|^2
\\
&\qquad \quad+\tau\Big(\nabla\cdot \widehat{\sig}^{\rm L}\big|^{\sim,m+\frac12},
\widetilde{\bs v}^{m+\frac12}\Big)
+\tau\Big(
\nabla \bs n^{m+\frac12}\cdot
\Big[
\big(
\bs n^{m+\frac12}\times D_{\mathcal F}^O(\bs n)\big|^{m+\frac12}
\big)\times \bs n^{m+\frac12}
\Big],
\widetilde{\bs v}^{m+\frac12}
\Big).
\end{aligned}
\end{equation}

One can further reorganize \eqref{eq:sch2pre1} to obtain $\widetilde{\bs v}^{m+1}-\bs v^{m+1}
=
({\tau}/{2})\nabla(P^{m+1}-P^m),
$ and  hence
\begin{equation}\label{eq: app-proof-p2}
\|\widetilde{\bs v}^{m+1}-\bs v^{m+1}\|^2
=
\frac{\tau^2}{4}\|\nabla(P^{m+1}-P^m)\|^2.
\end{equation}
We also take the $L^2$ inner product of \eqref{eq:sch2pre1} with $2\tau\nabla P^m$ and employ $\nabla\cdot \bs v^{m+1}=0$ to obtain
\begin{equation*}
(\widetilde{\bs v}^{m+1},\nabla P^m)
=
\frac{\tau}{2}\big(\nabla(P^{m+1}-P^m),\nabla P^m\big).
\end{equation*}
Together with the identity $2(\bs a-\bs b,\bs b)=\|\bs a\|^2-\|\bs b\|^2-\|\bs a-\bs b\|^2$ with $\bs a=\nabla P^{m+1}$ and $\bs b=\nabla P^m$,
the above equation becomes 
\begin{equation}\label{eq: app-proof-p4}
(\widetilde{\bs v}^{m+1},\nabla P^m)
=
\frac{\tau}{4}
\Big(
\|\nabla P^{m+1}\|^2-\|\nabla P^m\|^2
-\|\nabla(P^{m+1}-P^m)\|^2
\Big).
\end{equation}
On the other hand, taking the $L^2$ inner product of \eqref{eq:sch2pre1} with $\tau \bs v^{m+1}$ and again using $\nabla\cdot \bs v^{m+1}=0$ lead to
\begin{equation}\label{eq: app-proof-p5}
\|\bs v^{m+1}\|^2-\|\widetilde{\bs v}^{m+1}\|^2
+\|\widetilde{\bs v}^{m+1}-\bs v^{m+1}\|^2=0.
\end{equation}
Combining \eqref{eq: app-proof-p2}, \eqref{eq: app-proof-p4}, and \eqref{eq: app-proof-p5}, we arrive at
\begin{equation}\label{eq: app-proof-p6}
\begin{aligned}
&\frac{\rm Re}{2(1-\gamma)}
\Big(\|\bs v^{m+1}\|^2-\|\widetilde{\bs v}^{m+1}\|^2\Big)
-\frac{{\rm Re}\tau}{2(1-\gamma)}
(\nabla P^m,\widetilde{\bs v}^{m+1})
\\
&=
-\frac{{\rm Re}\tau^2}{8(1-\gamma)}
\Big(\|\nabla P^{m+1}\|^2-\|\nabla P^m\|^2\Big).
\end{aligned}
\end{equation}

Next, we take the $L^2$ inner product of \eqref{eq: projschemp2} with
$D_{\mathcal F}^O(\bs n)\big|^{m+\frac12}$. Now add the resulting
equation to \eqref{eq: app-proof-v}, \eqref{eq: app-proof-p6} , and noting that Ericksen-stress contribution and the asymmetric coupling part of the Leslie-stress contribution cancel with the corresponding coupling terms while the remaining terms follow the
same dissipation structure as in Theorem~\ref{thm: main-scheme}, we obtain
\begin{equation*}
\begin{aligned}
&\frac{\rm Re}{2(1-\gamma)}
\Big(\|\bs v^{m+1}\|^2-\|\bs v^m\|^2\Big)
+\mathcal F[\bs n^{m+1}]-\mathcal F[\bs n^m]
+\frac{{\rm Re}\tau^2}{8(1-\gamma)}
\Big(\|\nabla P^{m+1}\|^2-\|\nabla P^m\|^2\Big)
\\
&=
-\frac{\gamma\tau}{1-\gamma}\|\nabla \widetilde{\bs v}^{m+\frac12}\|^2
-\frac{\tau}{\gamma_1}
\Big\|
\bs n^{m+\frac12}\times D_{\mathcal F}^O(\bs n)\big|^{m+\frac12}
\Big\|^2-\Big(\alpha_5+\alpha_6-\frac{\gamma_2^2}{\gamma_1}\Big)\tau
\|\widetilde{\tt}^{m+\frac12}\cdot \bs n^{m+\frac12}\|^2
\\
&\quad
-\Big(\alpha_1+\frac{\gamma_2^2}{\gamma_1}\Big)\tau
\Big\|
\widetilde{\tt}^{m+\frac12}:\bs n^{m+\frac12}\otimes \bs n^{m+\frac12}
\Big\|^2
-\alpha_4\tau\|\widetilde{\tt}^{m+\frac12}\|^2.
\end{aligned}
\end{equation*}
Together with the constraints on coefficients in \eqref{eq: paracons} and the fact that $\gamma\in(0,1)$, it directly leads to the discrete energy dissipation law \eqref{eq: dissproj}. This ends the proof.
\end{proof}

\section{Explicit basis functions for the discrete divergence-free space}\label{app: space}
\renewcommand{\theequation}{B.\arabic{equation}}
In this appendix, we collect explicit basis functions for the exact divergence-free discrete space $\mathbb V_N\subset \bs H_0^1(\mathrm{div}0;\Omega)$ used in Section~\ref{sect: 4}.  We introduce three one-dimensional polynomials families $\{\psi_n\}$ and $\{\varphi_n\}$ and $\{ \zeta_n\}$ on $\Lambda=(-1,1)$
\begin{equation*}\label{eq: app-psi}
\begin{aligned}
& \psi_n(x)=
\frac{\sqrt{2(2n-1)}}{n-1}\,J_n^{(-1,-1)}(x),
\;\; n\ge 2, \quad \varphi_n(x)=
\frac{\sqrt{8(2n-3)}}{(n-3)(n-2)}\,J_n^{(-2,-2)}(x),
\;\; n\ge 4,\\
& \zeta_n(x)=
\frac{1}{\sqrt{4n+6}}\big(L_n(x)-L_{n+2}(x)\big),
\qquad n\ge 0,
\end{aligned}
\end{equation*}
where $J_n^{(-1,-1)}$ and $J_n^{(-2,-2)}$ are given by
\begin{equation*}\label{eq: app-1d-polys}
J_n^{(-1,-1)}(x)=
\frac{x^2-1}{4}\,J_{n-2}^{(1,1)}(x),
\;\; n\ge 2,\qquad J_n^{(-2,-2)}(x)=
\Big(\frac{x^2-1}{4}\Big)^2 J_{n-4}^{(1,1)}(x),
\;\; n\ge 4,
\end{equation*}
with $J_n^{(1,1)}(x)$ represents the classical Jacobi polynomial with hyper parameters $(1,1)$ (cf.~\cite{ShenTangWang2011}) and denote $L_n(x)$ for the Legendre polynomial of degree $n$.  They serve as the building blocks for the construction of exact divergence-free basis. 
Define the index set $\mathbb{Z}_N:=\big\{ 1,\dots,N-3  \big\}$, we now describe explicit bases for $\mathbb V_N$ on the reference domains $\Lambda^2$ and $\Lambda^3$.

In the two-dimensional setting, the incompressibility constraint acts on the first two components for the velocity field $\bs v(x_1,x_2)=\big(v_1(x_1,x_2),v_2(x_1,x_2),v_3(x_1,x_2) \big)^{\intercal}$. Since some two-dimensional numerical simulations employ vector fields with three components depending only on $(x_1,x_2)$, the third component is treated separately. Define
$
\mathbb V_{N,0}(\Lambda^2)
=
\mathrm{span}\big\{
\bs\Phi_{mn}^{(1)}, \bs\Phi_{mn}^{(2)}
\big\},
$
where
\begin{equation*}\label{eq: app-v2d-1}
\begin{aligned}
& \bs\Phi_{mn}^{(1)}(x_1,x_2)
=( \varphi_{m+3}(x_1)\psi_{n+2}(x_2),  -\psi_{m+2}(x_1)\varphi_{n+3}(x_2),0  )^{\intercal},\;
m,n=1,\dots,N-3, \\
& \bs\Phi_{lk}^{(2)}(x_1,x_2)=\big(0,0, \zeta_l(x_1)\zeta_k(x_2) \big)^{\intercal},\qquad
l,k=0,\dots,N-2.
\end{aligned}
\end{equation*}
The first family is exactly divergence-free in the $(x_1,x_2)$-plane, while the second family describes the uncoupled third component in the two-dimensional setting.

For the three-dimensional exact divergence-free space, we define
\[
\mathbb V_{N,0}(\Lambda^3)
=
\mathrm{span}\Big\{
\bs\Phi_{mnl}^{(1)},
\bs\Phi_{mnl}^{(2)},
\bs\Phi_{nl}^{(3)},
\bs\Phi_{ml}^{(4)},
\bs\Phi_{mn}^{(5)}
\Big\},
\]
where the exact divergence-free spectral basis are given by
\begin{equation*}\label{eq: app-v3d-1}
\begin{aligned}
& \bs\Phi_{mnl}^{(1)}=\big(\varphi_{m+3}(x_1)\psi_{n+2}(x_2)\psi_{l+2}(x_3),  -\psi_{m+2}(x_1)\varphi_{n+3}(x_2)\psi_{l+2}(x_3),0  \big)^{\intercal},\;\;m,n,l\in \mathbb{Z}_N, \\
& \bs\Phi_{mnl}^{(2)}=\big(\varphi_{m+3}(x_1)\psi_{n+2}(x_2)\psi_{l+2}(x_3),0,-\psi_{m+2}(x_1)\psi_{n+2}(x_2)\varphi_{l+3}(x_3)  \big)^{\intercal},\;\; m,n,l\in \mathbb{Z}_N,\\
& \bs\Phi_{nl}^{(3)}=\big(0, \psi_2(x_1)\varphi_{n+3}(x_2)\psi_{l+2}(x_3),-\psi_2(x_1)\psi_{n+2}(x_2)\varphi_{l+3}(x_3)   \big)^{\intercal},\;\; n,l\in \mathbb{Z}_N,\\
&\bs\Phi_{ml}^{(4)}=\big( \varphi_{m+3}(x_1)\psi_2(x_2)\psi_{l+2}(x_3),0, -\psi_{m+2}(x_1)\psi_2(x_2)\varphi_{l+3}(x_3)  \big)^{\intercal},\;\; m,l\in \mathbb{Z}_N,\\
& \bs\Phi_{mn}^{(5)}=\big(\varphi_{m+3}(x_1)\psi_{n+2}(x_2)\psi_2(x_3),-\psi_{m+2}(x_1)\varphi_{n+3}(x_2)\psi_2(x_3), 0   \big)^{\intercal},\;\; m,n\in \mathbb{Z}_N.
\end{aligned}
\end{equation*}
Each of the above families belongs to $\bs H_0^1(\mathrm{div}0;\Lambda^3)$, and their span gives the exact divergence-free spectral space used for the velocity approximation. The corresponding basis functions on the physical domain $\Omega$ can be obtained from the reference domain basis functions through the affine mapping and the associated contravariant Piola transformation~\cite[page 39]{ciarlet1988}.

\section*{Acknowledgement}
H. Wang and Z. Yang acknowledge the support from the National Key R\&D Program of China (No. 2024YFA1016100), the National Natural Science Foundation of China (Nos. 12471343, 12531015, 12101399),
and the Key Laboratory of Scientific and Engineering Computing (Ministry of Education).  J. Xu acknowledges the support from the National Key R\&D Program of China, the National Natural Science Foundation of China (No. 12371414, 12288201), the Strategic Priority Research Program of the Chinese Academy of Sciences (No. XDB0510201), Beijing Natural Science Foundation (No. JQ25002).

\bibliography{refpapers}

@book{ShenTangWang2011,
  title = {{Spectral Methods: Algorithms, Analysis and Applications}},
  author = {Shen, J. and Tang, T. and Wang, L. L.},
  SERIES = {Springer Series in  Computational Mathematics},
  VOLUME = {41},
  year = {2011},
  publisher = {Springer, Heidelberg}
}

@book{kelley1995iterative,
  title={Iterative Methods for Linear and Nonlinear Equations},
  author={Kelley, C. T.},
  ADDRESS={Philadelphia},
  year={1995},
  publisher={SIAM}
}

@article{qin2023exact,
  title={An exact divergence-free spectral method for incompressible and resistive magneto-hydrodynamic equations in two and three dimensions},
  author={Qin, L. and Li, H. and Yang, Z.},
  journal={arXiv:2312.12218},
  year={2023}
}

@article{cao2025length,
  title={Length preserving numerical schemes for the nematic liquid crystal flows},
  author={Cao, R. and Yi, N.},
  journal={ESAIM Math. Model. Numer. Anal.},
  fjournal={ESAIM: Mathematical Modelling and Numerical Analysis},
  volume={59},
  number={6},
  pages={3021--3040},
  year={2025},
  publisher={EDP Sciences}
}

@article{Badia2011a,
  title={Finite element approximation of nematic liquid crystal flows using a saddle-point structure},
  author={Badia, S. and Guill{\'e}n-Gonz{\'a}lez, F. and Guti{\'e}rrez-Santacreu, J. V.},
  journal={J. Comput. Phys.},
  fjournal={Journal of Computational Physics},
  volume={230},
  number={4},
  pages={1686--1706},
  year={2011},
  publisher={Elsevier}
}

@article{wang2023error,
  title={Error estimates of a sphere-constraint-preserving numerical scheme for {Ericksen-Leslie} system with variable density.},
  author={Wang, D. and Liu, F. and Jia, H. and Zhang, J.},
  journal={Discrete Contin. Dyn. Syst. Ser. B},
  fjournal={Discrete and Continuous Dynamical Systems-Series B},
  volume={28},
  number={11},
  pages={5814--5838},
  year={2023}
}

@article{bao2021constraint,
  title={Constraint-preserving energy-stable scheme for the {2D} simplified {Ericksen-Leslie} system},
  author={Bao, X. and Chen, R. and Zhang, H.},
  journal={J. Comput. Math.},
  fjournal={Journal of Computational Mathematics},
  volume={39},
  number={1},
  pages={1--21},
  year={2021}
}

@article{gui2022convergence,
  title={Convergence of renormalized finite element methods for heat flow of harmonic maps},
  author={Gui, X. and Li, B. and Wang, J.},
  journal={SIAM J. Numer. Anal.},
  fjournal={SIAM Journal on Numerical Analysis},
  volume={60},
  number={1},
  pages={312--338},
  year={2022},
  publisher={SIAM}
}

@article{bai2025convergence,
  title={Convergence of multistep projection methods for harmonic map heat flows into general surfaces},
  author={Bai, G. and Gui, X. and Li, B.},
  journal={Numer. Math.},
  fjournal={Numerische Mathematik},
  volume={157},
  number={2},
  pages={629--661},
  year={2025},
  publisher={Springer}
}

@article{lin1995nonparabolic,
  title={Nonparabolic dissipative systems modeling the flow of liquid crystals},
  author={Lin, F. and Liu, C.},
  journal={Comm. Pure Appl. Math.},
  fjournal={Communications on Pure and Applied Mathematics},
  volume={48},
  number={5},
  pages={501--537},
  year={1995},
  publisher={Wiley Online Library}
}

@article{zou2023extrapolated,
  title={An extrapolated {Crank-Nicolson} virtual element scheme for the nematic liquid crystal flows},
  author={Zou, G. and Wang, X. and Li, J.},
  journal={Adv. Comput. Math.},
  fjournal={Advances in Computational Mathematics},
  volume={49},
  number={3},
  pages={30},
  year={2023},
  publisher={Springer}
}

@article{zheng2024novel,
  title={A novel discontinuous {Galerkin} projection scheme for the hydrodynamics of nematic liquid crystals},
  author={Zheng, Z. and Zou, G. and Wang, B.},
  journal={Commun. Nonlinear Sci. Numer. Simul.},
  fjournal={Communications in Nonlinear Science and Numerical Simulation},
  volume={137},
  pages={108163},
  year={2024},
  publisher={Elsevier}
}

@article{Zhao_Yang_Li_and_Wang-2016,
  title={Energy stable numerical schemes for a hydrodynamic model of nematic liquid crystals},
  author={Zhao, J. and Yang, X. and Li, J. and Wang, Q.},
  journal={SIAM J. Sci. Comput.},
  fjournal={SIAM Journal on Scientific Computing},
  volume={38},
  number={5},
  pages={A3264--A3290},
  year={2016},
  publisher={SIAM}
}

@article{xu2024second,
  title={A Second-Order Length-Preserving and Unconditionally Energy Stable Rotational Discrete Gradient Method for {Oseen-Frank} Gradient Flows},
  author={Xu, J. and Yang, X. and Yang, Z.},
  journal={Commun. Comput. Phys.},
  fjournal={Communications in Computational Physics},
  volume={35},
  number={2},
  pages={369--394},
  year={2024}
}

@article{walker2020finite,
  title={A finite element method for the generalized {Ericksen} model of nematic liquid crystals},
  author={Walker, S. W.},
  journal={ESAIM Math. Model. Numer. Anal.},
  fjournal={ESAIM: Mathematical Modelling and Numerical Analysis},
  volume={54},
  number={4},
  pages={1181--1220},
  year={2020},
  publisher={EDP Sciences}
}

@article{nochetto2018ericksen,
  title={The {Ericksen} model of liquid crystals with colloidal and electric effects},
  author={Nochetto, R. H. and Walker, S. W. and Zhang, W.},
  journal={J. Comput. Phys.},
  fjournal={Journal of Computational Physics},
  volume={352},
  pages={568--601},
  year={2018},
  publisher={Elsevier}
}

@article{liu2000approximation,
  title={Approximation of liquid crystal flows},
  author={Liu, C. and Walkington, N. J.},
  journal={SIAM J. Numer. Anal.},
  fjournal={SIAM Journal on Numerical Analysis},
  volume={37},
  number={3},
  pages={725--741},
  year={2000},
  publisher={SIAM}
}

@article{liu2001liquid,
  title={On liquid crystal flows with free-slip boundary conditions},
  author={Liu, C. and Shen, J.},
  journal={Discrete Contin. Dyn. Syst.},
  fjournal={Discrete and Continuous Dynamical Systems},
  volume={7},
  number={2},
  pages={307--318},
  year={2001},
  publisher={Southwest Missouri State University}
}

@article{lin2007energy,
  title={An energy law preserving {$C^0$} finite element scheme for simulating the kinematic effects in liquid crystal dynamics},
  author={Lin, P. and Liu, C. and Zhang, H.},
  journal={J. Comput. Phys.},
  fjournal={Journal of Computational Physics},
  volume={227},
  number={2},
  pages={1411--1427},
  year={2007},
  publisher={Elsevier}
}

@article{du2025semi,
  title={Semi-implicit Projection Schemes for Manifold Constraint Gradient Flows},
  author={Du, Q. and Liu, S. and Yang, J.},
  journal={J. Sci. Comput.},
  fjournal={Journal of Scientific Computing},
  volume={103},
  number={3},
  pages={92},
  year={2025},
  publisher={Springer}
}

@article{du2001fourier,
  title={Fourier spectral approximation to a dissipative system modeling the flow of liquid crystals},
  author={Du, Q. and Guo, B. and Shen, J.},
  journal={SIAM J. Numer. Anal.},
  fjournal={SIAM Journal on Numerical Analysis},
  volume={39},
  number={3},
  pages={735--762},
  year={2001},
  publisher={SIAM}
}

@article{cheng2023energy,
  title={An energy stable finite difference scheme for the {Ericksen-Leslie} system with penalty function and its optimal rate convergence analysis},
  author={Cheng, K. and Wang, C. and S. M. Wise},
  journal={Commun. Math. Sci.},
  fjournal={Communications in Mathematical Sciences},
  volume={21},
  number={4},
  pages={1135--1169},
  year={2023}
}

@article{chen2016kinematic,
  title={The kinematic effects of the defects in liquid crystal dynamics},
  author={Chen, R. and Bao, W. and Zhang, H.},
  journal={Commun. Comput. Phys.},
  fjournal={Communications in Computational Physics},
  volume={20},
  number={1},
  pages={234--249},
  year={2016},
  publisher={Cambridge University Press}
}

@article{cabrales2015time,
  title={A Time-Splitting Finite-Element Stable Approximation for the {Ericksen--Leslie} Equations},
  author={Cabrales, R. C. and Guill{\'e}n-Gonz{\'a}lez, F. and Guti{\'e}rrez-Santacreu, J. V.},
  journal={SIAM J. Sci. Comput.},
  fjournal={SIAM Journal on Scientific Computing},
  volume={37},
  number={2},
  pages={B261--B282},
  year={2015},
  publisher={SIAM}
}

@article{becker2008finite,
  title={Finite element approximations of the {Ericksen--Leslie} model for nematic liquid crystal flow},
  author={Becker, R. and Feng, X. and Prohl, A.},
  journal={SIAM J. Numer. Anal.},
  fjournal={SIAM Journal on Numerical Analysis},
  volume={46},
  number={4},
  pages={1704--1731},
  year={2008},
  publisher={SIAM}
}

@article{alouges1997new,
  title={A new algorithm for computing liquid crystal stable configurations: the harmonic mapping case},
  author={Alouges, F.},
  journal={SIAM J. Numer. Anal.},
  fjournal={SIAM Journal on Numerical Analysis},
  volume={34},
  number={5},
  pages={1708--1726},
  year={1997},
  publisher={SIAM}
}

@article{toth2002hydrodynamics,
  title={Hydrodynamics of topological defects in nematic liquid crystals},
  author={T{\'o}th, G. and Denniston, C. and Yeomans, J. M.},
  journal={Phys. Rev. Lett.},
  fjournal={Physical Review Letters},
  volume={88},
  number={10},
  pages={105504},
  year={2002},
  publisher={APS}
}

@article{sircar2009dynamics,
  title={Dynamics and rheology of biaxial liquid crystal polymers in shear flows},
  author={Sircar, S. and Wang, Q.},
  journal={J. Rheol.},
  fjournal={Journal of Rheology},
  volume={53},
  number={4},
  pages={819--858},
  year={2009},
  publisher={AIP Publishing}
}

@article{forest2004weak,
  title={The weak shear kinetic phase diagram for nematic polymers},
  author={Forest, M. G. and Wang, Q. and Zhou, R.},
  journal={Rheol. Acta},
  fjournal={Rheologica Acta},
  volume={43},
  number={1},
  pages={17--37},
  year={2004},
  publisher={Springer}
}

@article{forest2004flow,
  title={The flow-phase diagram of {Doi-Hess} theory for sheared nematic polymers II: finite shear rates},
  author={Forest, M. G. and Wang, Q. and Zhou, R.},
  journal={Rheol. Acta},
  fjournal={Rheologica Acta},
  volume={44},
  number={1},
  pages={80--93},
  year={2004},
  publisher={Springer}
}

@article{li2023frame,
  title={Frame hydrodynamics of biaxial nematics from molecular-theory-based tensor models},
  author={Li, S. and Xu, J.},
  journal={SIAM J. Appl. Math.},
  fjournal={SIAM Journal on Applied Mathematics},
  volume={83},
  number={4},
  pages={1467--1495},
  year={2023},
  publisher={SIAM}
}

@article{Han2015from,
title={From microscopic theory to macroscopic theory: a systematic study on modeling for liquid crystals},
author={Han, J. and Luo, Y. and Wang, W. and Zhang, P. and Zhang, Z.},
journal={Arch. Ration. Mech. Anal.},
fjournal={Archive for Rational Mechanics and Analysis},
volume = "215",
number = "3",
pages = "741--809",
year = "2015",
}

@BOOK{deGennesProst1993,
title = "{The Physics of Liquid Crystals}",
publisher="{Oxford Univ. Press}",
ADDRESS="Oxford",
year = "1993",
author = "P. G. de Gennes and J. Prost"
}

@article{Frank1958,
title = "I. liquid crystals. {On} the theory of liquid crystals",
journal = "Discuss. Faraday Soc.",
fjournal = "Discussions of the Faraday Society",
volume = "25",
number = " ",
pages = "19-28",
year = "1958",
author = "F. C. Frank"
}

@incollection{leslie1979theory,
  title={Theory of flow phenomena in liquid crystals},
  author={Leslie, F. M.},
  booktitle={Advances in liquid crystals},
  Volume={4},
  pages={1--81},
  year={1979},
  publisher={Academic Press},
  address={New York}
}

@article{ericksen1962hydrostatic,
  title={Hydrostatic theory of liquid crystals},
  author={Ericksen, J. L.},
  journal={Arch. Ration. Mech. Anal.},
  fjournal={Archive for Rational Mechanics and Analysis},
  volume={9},
  number={1},
  pages={371--378},
  year={1962},
  publisher={Springer}
}

@article{walton2025orienting,
  title={Orienting field effects on the flow of an active nematic liquid crystal in a channel},
  author={Walton, J. and McKay, G. and Mottram, N. J.},
  journal={Eur. Phys. J. E},
  fjournal={The European Physical Journal E},
  volume={48},
  number={10},
  pages={67},
  year={2025},
  publisher={Springer}
}

@article{Sengupta2013FlowShaping,
  author  = {Sengupta, A. and Tkalec, U. and Ravnik, M. and Yeomans, J. M. and Bahr, C. and Herminghaus, S.},
  title   = {Liquid Crystal Microfluidics for Tunable Flow Shaping},
  journal = {Phys. Rev. Lett.}, 
  fjournal = {Physical Review Letters},
  year    = {2013},
  volume  = {110},
  number  = {4},
  pages   = {048303},
  doi     = {10.1103/PhysRevLett.110.048303}
}

@article{sengupta2012opto,
  title={Opto-fluidic velocimetry using liquid crystal microfluidics},
  author={Sengupta, A. and Herminghaus, S. and Bahr, C.},
  journal={Appl. Phys. Lett.},
  fjournal={Applied Physics Letters},
  volume={101},
  number={16},
  pages = {164101},
  year={2012},
  publisher={AIP Publishing}
}

@article{Peng2011LCSpatialLightModulator,
  author  = {Peng, Z. and Liu, Y. and Yao, L. and Cao, Z. and Mu, Q. and Hu, L. and Xuan, L.},
  title   = {Improvement of the switching frequency of a liquid-crystal spatial light modulator with optimal cell gap},
  journal={Opt. Lett.},
  fjournal = {Optics Letters},
  year    = {2011},
  volume  = {36},
  number  = {18},
  pages   = {3608--3610},
  doi     = {10.1364/OL.36.003608}
}

@book{ciarlet1988,
  title={Mathematical elasticity: Three-dimensional elasticity},
  author={Ciarlet, P. G.},
  year={1988},
  publisher={Elsevier},
  address={New York}
}

@article{bouck2024projection,
  title={Projection-Free Method for the Full {Frank-Oseen} Model of Liquid Crystals},
  author={Bouck, L. and Nochetto, R. H.},
  journal={arXiv:2405.03145},
  year={2024}
}

\end{document}